\documentclass{article}

\title{\huge\textbf{New insights into non-central beta distributions}}
\author{
Carlo Orsi\footnote{\mail{c.orsi@campus.unimib.it}, \mail{orsi.carlo@gmail.com}}
}
\usepackage[utf8]{inputenc}
\usepackage[english]{babel}
\usepackage{geometry}
\usepackage{fancyhdr}
\usepackage{hyperref}
\hypersetup{colorlinks=true, linkcolor=black, citecolor=black, urlcolor=black, plainpages=false}
\usepackage{color}
\usepackage{nameref}
\usepackage{graphicx}
\usepackage{subfigure}
\usepackage{float}
\usepackage{amsthm}
\usepackage{amssymb}
\usepackage{amsmath}
\usepackage{verbatim}
\usepackage{lscape}
\usepackage[font=small,labelfont=bf]{caption}
\usepackage{multirow}
\usepackage{wrapfig}
\newcommand{\mail}[1]{\href{mailto:#1}{\texttt{#1}}}

\usepackage[table]{xcolor}
\usepackage{array}
\newcolumntype{C}[1]{>{\centering\let\newline\\\arraybackslash\hspace{0pt}}m{#1}}
\newcolumntype{L}[1]{>{\raggedright\let\newline\\\arraybackslash\hspace{0pt}}m{#1}}
\newcolumntype{R}[1]{>{\raggedleft\let\newline\\\arraybackslash\hspace{0pt}}m{#1}}
\newtheoremstyle{plain}  
  {\topsep}   
  {\topsep}   
  {\itshape}  
  {}       
  {\bfseries} 
  {}         
  {\newline}  
  {}          
\theoremstyle{plain}
\newtheorem{property}{Property}[section]
\newtheorem{proposition}{\textbf{Proposition}}[section]

\newtheoremstyle{mystyle}  
  {\topsep}   
  {\topsep}   
  {\normalfont}  
  {0pt}       
  {\bfseries} 
  {.}         
  {\newline}  
  {}          
\theoremstyle{mystyle}
\newtheorem{myproof}{Proof}[section]
\newtheorem{funct}{Function}[section]

\begin{document}
\baselineskip24pt
\maketitle

\begin{abstract}
\baselineskip24pt
The beta family owes its privileged status within unit interval distributions to several relevant features such as, for example, easyness of interpretation and versatility in modeling different types of data. However, its flexibility at the unit interval endpoints is poor enough to prevent from properly modeling the portions of data having values next to zero and one. Such a drawback can be overcome by resorting to the class of the non-central beta distributions. Indeed, the latter allows the density to take on arbitrary positive and finite limits which have a really simple form. That said, new insights into such class are provided in this paper. In particular, new representations and moments expressions are derived. Moreover, its potential with respect to alternative models is highlighted through applications to real data.

\vspace*{0.2cm}
\end{abstract}

\section{Introduction}
\label{sec:introduc}

The beta distribution plays a prominent role in the analysis of random phenomena which take on values with lower and upper bounds. Indeed, allowing its probability density function to have a great variety of shapes, such a distribution is versatile enough to model data arisen from a wide range of fields. In this regard, for example see \cite{EttMef02}, \cite{HasPac95}, \cite{MalRos59}, \cite{WilHer89}.

However, the density of the latter shows poor flexibility at the unit interval endpoints. In fact, its limiting values are equal to one if the shape parameters are unitary (in this case it reduces to the uniform one) and are equal to zero or infinity otherwise. As a consequence of this, the beta distribution prevents from properly modeling the portions of data having values next to zero and one.

In this regard, in the literature there exists some generalizations of the beta model that enable to overcome this limitation thanks to a richer parametrization. For instance, we recall the Libby and Novick's generalized beta \cite{LibNov82}, the Gauss hypergeometric \cite{ArmBay94} and the confluent hypergeometric \cite{Gor98} distributions. Indeed, the densities of the aforementioned models can take on positive and finite values at zero and one when the shape parameters are unitary. See for example \cite{NadGup04} to get an overview of such distributions. 

That said, the present paper aims at getting an insight into the class of the non-central beta distributions. The latter is another extension of the beta model that exhibits the aforementioned peculiarity. Indeed, its density shows positive and finite limits that, interestingly, have a really simple form \cite{OngOrs15}. Hence, such class is considered to be worthy of further investigating. More specifically, our intent is to provide a valid point of reference for the study of such distributions. In this regard, we are supported by the fact that in recent years the non-central beta distributions have attracted many applications. For example, \cite{KimSch98} pointed out that the semblance of a single wave propagating across a receiver array with added Gaussian noise is distributed according to a special case of non-central beta, called type 1. In the setting of magnetic resonance image reconstruction, \cite{Sta15} introduced a new estimating method for coil sensitivity profiles that uses spatial smoothing and additional body coil data for phase normalization. Upon providing detailed information on the statistical distribution of this estimator, they showed that the square of the random variable $\mathcal{R}_k\left(\mbox{\textbf{x}}\right)$, which plays a relevant role in the definition of such a method, follows a doubly non-central beta distribution, the latter being the most general non-central extension of the beta one. In order to analyze the bias and the variance of this estimator, the calculation of the first two raw moments of the aforementioned distribution was needed.

Finally, the present paper is organized as follows. In Section \ref{sec:ncchidistr}, in order to go into the matter of interest in due depth, we shall focus on the non-central chi-squared distribution. Indeed, the latter covers a crucial role in the study of the family of generalizations of the beta distribution we are interested in. More precisely, its definition and some useful properties are briefly recalled and a new general expression for its moments about zero is derived. In Section~\ref{subsec:dnc.beta.def.repres} the definition and various representations of the doubly non-central beta distribution are provided. In particular, a new representation of a random variable distributed as previously said is here obtained in terms of a convex linear combination of a central component and a purely non-central one. In Section~\ref{subsec:dnc.beta.dens.plots} some significant plots of the density are shown. In this regard, a special focus is given to the case in which both the shape parameters are unitary; in fact, in this case the density shows the attractive feature of taking on arbitrary finite and positive limits at zero and one. Section~\ref{subsec:dnc.beta.dens.pat.approx} presents how a simple approximation for the doubly non-central beta distribution can be determined by applying the Patnaik's approximation for the non-central chi-squared one \cite{Pat49}. Section~\ref{subsec:dnc.beta.mom} sheds new light on the issue of moments expression. As a matter of fact, a more straightforward general formula for the moments about zero of such distribution is derived. Last but not the least interesting, in Section~\ref{subsec:dnc.beta.examples} the potential of the doubly non-central beta distribution is highlighted through applications to real data with respect to the above mentioned alternative models on the real interval $\left(0,1\right)$. The issue of the parameters estimation is here addressed using both the moments and the maximum-likelihood methods. Some concluding remarks are provided in Section~\ref{sec:concl}.

For clarity of exposition, the proofs of all the following results are given in the Appendix~\ref{sec:app.proofs}, while in the Appendix~\ref{sec:app.r.func} the implementation of the major issues dealt with in this paper is provided in \texttt{R} programming language.


\section{Preliminaries on the non-central chi-squared distribution}
\label{sec:ncchidistr}


\subsection{Definition, representations and properties}
\label{sec:ncchidistr.defreprprop}

In this Section we shall recall the definition and some useful properties of the non-central chi-squared distribution. The latter represents the main ingredient for the study of the class of distributions on the real interval $(0,1)$ we are interested in. In fact, some results included in the remainder of this paper, such as Propositions~\ref{propo:rappr.clc.beta.dnc} and~\ref{propo:dens.beta.dnc.lims}, ensue from analogous results regarding the present distribution (Properties~\ref{prope:sumrepres.ncchisq} and~\ref{prope:limit.0.nc.chisq}, respectively), while others, such as Propositions~\ref{propo:dnc.beta.pat.approx} and~\ref{propo:momr.beta.dnc.rapp}, are strongly implied by some of its properties (Properties~\ref{prope:pat.approx.ncchisq} and~\ref{prope:mixrepres.ncchisq},~\ref{prope:repr.prop.ncchisq}, respectively).

That said, the non-central extension of the chi-squared distribution is defined as follows.

Let $W_k$, $k=1,\ldots,g$, be independent and normally distributed random variables with expectations $\mu_k$ and unitary variances. Then, a random variable is said to have a non-central chi-squared distribution with $g>0$ degrees of freedom and non-centrality parameter $\lambda=\sum_{k=1}^{g}\mu^2_k\ge 0$, denoted by $\chi'^{\,2}_g \left(\lambda \right)$, if it is distributed as $Y'=\sum_{k=1}^{g}W^2_k$ \cite{JohKotBal95}. The case $\lambda=0$ clearly corresponds to the $\chi^2_g$ distribution.

The density function $f_{Y'}$ of $Y' \sim \chi'^{\,2}_g \left(\lambda \right)$ can be expressed as:
\begin{equation}
f_{Y'}\left(y;g,\lambda\right)=\sum_{i=0}^{+\infty}\frac{e^{-\frac{\lambda}{2}}\left(\frac{\lambda}{2}\right)^i}{i!}
\frac{y^{\frac{g+2i}{2}-1} \, e^{-\frac{y}{2}}}{\Gamma\left(\frac{g+2i}{2}\right)2^{\frac{g+2i}{2}}}, \quad y>0,
\label{eq:dens.chi.nc.gn}
\end{equation}
i.e. as the series of the $\chi^2_{g+2i}$ densities, $i \in \mathbb{N}\cup \{0\}$, weighted by the probabilities of a Poisson random variable with mean $\lambda/2$, $\lambda\ge0$ (the case $\lambda=0$ corresponding to a random variable degenerate at zero).

In view of Eq.~(\ref{eq:dens.chi.nc.gn}), the $\chi'^{\,2}_g \left(\lambda \right)$ distribution admits the following mixture representation.
\begin{property}[Mixture representation of $\chi'^{\,2}_g \left(\lambda \right)$]
\label{prope:mixrepres.ncchisq}
Let $Y'$ have a $\chi'^{\,2}_g \left(\lambda \right)$ distribution and $M$ be a Poisson random variable with mean $\lambda/2$. Then, $Y'$ admits the following representation:
\begin{equation}
Y'\,| \, M \sim \chi^2_{g+2M} \; .
\label{eq:mixrepres.ncchisq}
\end{equation}
\end{property}

Interestingly, a non-central chi-squared random variable with $g$ degrees of freedom and non-centrality parameter $\lambda$ can be additively decomposed into two components, a central one with $g$ degrees of freedom and a purely non-central one with non-centrality parameter $\lambda$ \cite{Hjo88}. The latter can be easily obtained from Property~\ref{prope:mixrepres.ncchisq} by making use of the reproductive property of the chi-squared distribution with respect to degrees of freedom. Such representation turns out to be as follows.

\begin{property}[Sum of a central part and a purely non-central part]
\label{prope:sumrepres.ncchisq}
Let $Y' \sim \chi'^{\,2}_g \left(\lambda \right)$. Then:
\begin{equation}
Y'=Y+\sum_{j=1}^{M}F_j,
\label{eq:sumrepres.ncchisq}
\end{equation}
where:
\begin{itemize}
\item[i)] $Y$, $M$, $\left\{F_j\right\}$ are mutually independent,
\item[ii)] $Y \sim \chi^2_g$, $M \sim \mbox{\normalfont{Poisson}}\left(\lambda/2\right)$ and $\left\{F_j\right\}$ is a sequence of independent random variables with $\chi^2_2$ distribution.  
\end{itemize}
\end{property}

In the notation of Property~\ref{prope:sumrepres.ncchisq}, the random variable $Y'_{pnc}=\sum_{j=1}^{M}F_j$ is said to have a purely non-central chi-squared distribution with non-centrality parameter $\lambda$. Indeed, it is denoted by $\chi'^{\,2}_0 \left(\lambda \right)$, the degrees of freedom being equal to zero.

The case $g=2$ is of prominent interest in the present setting; in fact, in such case the limit at $0$ of the non-central chi-squared density is decreasing in $\lambda$.
\begin{property}[Limit at $0$ of the $\chi'^{\,2}_g \left(\lambda \right)$ density when $g=2$]
\label{prope:limit.0.nc.chisq}
Let $Y'$ be a $\chi'^{\,2}_2 \left(\lambda \right)$ random variable and $f_{Y'}\left(y; 2,\lambda\right)$ denote its density function. Then $\lim_{y \rightarrow 0^+}f_{Y'}\left(y; 2,\lambda\right)=\frac{1}{2} \, e^{-\frac{\lambda}{2}}.
$
\end{property}  

Plots of the latter are displayed in Figure~\ref{fig:NCCHISQ} for selected values of the non-centrality parameter.

\begin{figure}[ht]
  \centering
  \includegraphics[width=8.5cm,keepaspectratio=true]{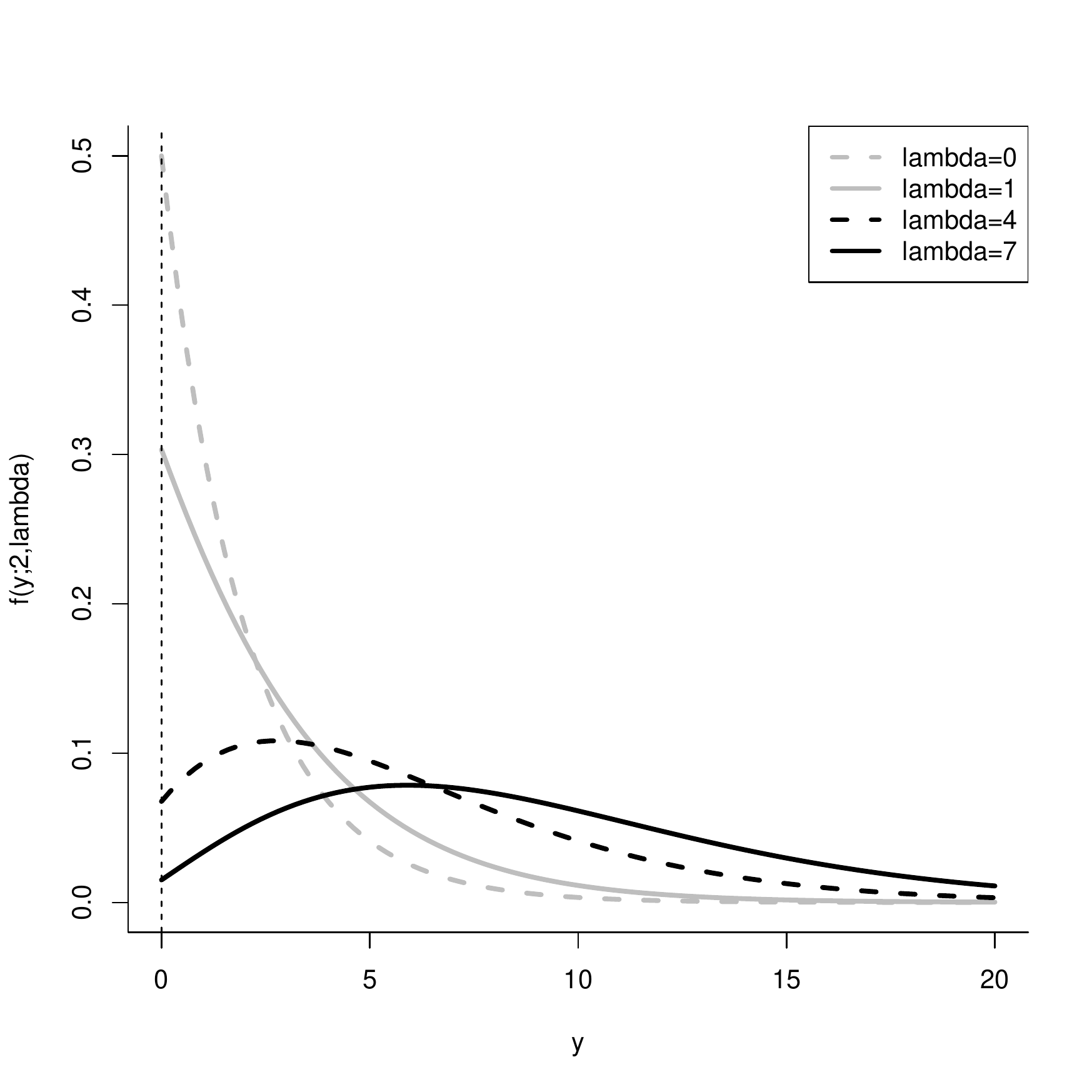}\\
\caption{Plots of the density of $Y' \sim \chi'^{\,2}_2 \left(\lambda \right)$ for selected values of $\lambda$.}\label{fig:NCCHISQ}
\end{figure}

The non-central chi-squared distribution is reproductive with respect to both degrees of freedom and non-centrality parameter. The latter can be easily derived from its characteristic function \cite{JohKotBal95}.

\begin{property}[Reproductive property of $\chi'^{\,2}_g \left(\lambda \right)$]
\label{prope:repr.prop.ncchisq}
If $Y'_j$, $j=1,\ldots,m$, are independent with $\chi'^2_{g_j}(\lambda_j)$ distributions, then $Y'^{+}=\sum_{j=1}^m Y'_j \sim \chi'^2_{g^+}(\lambda^+)$, with $g^+=\sum_{j=1}^m g_j$ and $\lambda^+=\sum_{j=1}^m \lambda_j$. 
\end{property}

Finally, we recall the simple approximation for the non-central chi-squared distribution suggested by Patnaik \cite{Pat49}. 


\begin{property}[Patnaik's approximation for $\chi'^{\,2}_g \left(\lambda \right)$]
\label{prope:pat.approx.ncchisq}
Let $Y'$ have a $\chi'^2_{g}(\lambda)$ distribution with $g>0$ and $\lambda>0$ and $Y$ have a $\chi^2_{\nu}$ distribution with $\nu=\frac{\left(g+\lambda\right)^2}{g+2\lambda}$. Furthermore, let $Y'_P=\rho \, Y \sim \mbox{\normalfont{Gamma}}\left(\frac{\nu}{2},\frac{1}{2 \rho}\right)$, with $\rho=\frac{g+2\lambda}{g+\lambda}$. Then, one can approximate $Y' \stackrel{d}{\approx} Y'_P$.
\end{property}

In the notation of Property~\ref{prope:pat.approx.ncchisq}, observe that as $\lambda$ tends to $0^+$, $\nu$ tends to $g$ and $\rho$ tends to 1; therefore, the distributions of both $Y'$ and $Y'_P$ tend to the $\chi^2_g$ one.

\subsection{A note on the moments about zero}
\label{sec:ncchidistr.mom}

The $r$-th moment about zero of $Y' \sim \chi'^{\,2}_g \left(\lambda \right)$, $g>0$, can be evaluated according to the following formula set out by \cite{JohKotBal95}:
\begin{equation}
\mathbb{E}\left[\left(Y' \right)^r \right]=2^r \, \Gamma\left(r+\frac{g}{2}\right) \sum_{j=0}^{r} {r \choose j} \frac{\left(\frac{\lambda}{2}\right)^j}{\Gamma\left(j+\frac{g}{2}\right)}.
\label{eq:mom.literat.ncchisq}
\end{equation}

A new moment formula for the non-central chi-squared distribution can be derived regardless of Eq.~(\ref{eq:mom.literat.ncchisq}) by means of the following simple expansion of the ascending factorial of a binomial, which, as far as we know, has never been discussed in the literature.

In this regard, we recall that:
\begin{equation}
\left(a\right)_0=1, \quad \left(a\right)_l=a\left(a+1\right)\ldots\left(a+l-1\right), \qquad l \in \mathbb{N}
\label{eq:poch.symb}
\end{equation}
is the ascending factorial or Pochhammer's symbol of $a \in \mathbb{R}$ \cite{JohKemKot05}. Observe that for every $a \in \mathbb{R}-\{0\}$ Eq.~(\ref{eq:poch.symb}) is tantamount to:
\begin{equation}
\left(a\right)_l=\frac{\Gamma\left(a+l\right)}{\Gamma\left(a\right)}, \qquad l \in \mathbb{N} \cup \{0\}.
\label{eq:poch.symb2}
\end{equation}
Furthermore, in light of Eq.~(\ref{eq:poch.symb2}), one has:
\begin{equation}
\left(a\right)_{l+m}=\frac{\Gamma\left(a+l+m\right)}{\Gamma\left(a\right)}=\left\{\begin{array}{l} \frac{\Gamma\left(a+l\right)}{\Gamma\left(a\right)} \, \frac{\Gamma\left(a+l+m\right)}{\Gamma\left(a+l\right)}=\left(a\right)_l \, \left(a+l\right)_m \\ \\ \frac{\Gamma\left(a+m\right)}{\Gamma\left(a\right)} \, \frac{\Gamma\left(a+m+l\right)}{\Gamma\left(a+m\right)}=\left(a\right)_m \, \left(a+m\right)_l\end{array}\right.
\label{eq:poch.symb.sum}
\end{equation}
for every $l,m \in \mathbb{N}\cup \{0\}$.

That said, the aforementioned expansion follows.
\begin{proposition}[Expansion of the ascending factorial of a binomial]
\label{propo:expans.poch.symb.binom}
Let $a$, $b \in \mathbb{R}-\{0\}$. Then, for every $l \in \mathbb{N} \cup \{0\}$:
\begin{equation}
\left(a+b\right)_l=\sum_{i=0}^{l}\frac{1}{i!}\left[\frac{d^i}{d a^i} \left(a\right)_l\right] \, b^i,
\label{eq:expans.poch.symb.binom}
\end{equation}
where $d^i f/d a^i$ denotes the $i$-th derivative of $f$ with respect to $a$ (the case $i=0$ corresponding to $f$) and $\left(a\right)_l$ is defined as in Eq.~(\ref{eq:poch.symb}).
\end{proposition}
\begin{proof}
For the proof see~\ref{proof:expans.poch.symb.binom} in the Appendix. 
\end{proof}

The latter result and the mixture representation in Eq.~(\ref{eq:mixrepres.ncchisq}) lead to the following new general formula for the moments of the non-central chi-squared distribution.

\begin{proposition}[Moments about zero of $\chi'^{\,2}_g \left(\lambda \right)$]
\label{propo:mom.ncchisq}
Let $Y'$ have a $\chi'^{\,2}_g \left(\lambda \right)$ distribution with $g>0$. Then, for every $r \in \mathbb{N}$, the $r$-th moment about zero of $Y'$ can be written as:
\begin{equation}
\mathbb{E}\left[\left(Y' \right)^r \right]=2^r \, \sum_{i=0}^{r}\sum_{j=0}^{i} \mathcal{S}\left(i,j\right) \frac{1}{i!}\left[\frac{d^i}{d h^i} \left(h\right)_r\right] \left(\frac{\lambda}{2}\right)^j,
\label{eq:mom.ncchisq}
\end{equation}
where $\mathcal{S}\left(i,j\right)$ is a Stirling number of the second kind, $h=g/2$ and $\left(h\right)_r$ is defined as in Eq.~(\ref{eq:poch.symb}).
\end{proposition}
\begin{proof}
For the proof see~\ref{proof:mom.ncchisq} in the Appendix.
\end{proof}

However, neither the moments formula available in the literature nor the one herein derived apply in case of zero degrees of freedom. As far as the computation of the $r$-th moment about zero of the purely non-central chi-squared distribution is concerned, the following formula can be used.

\begin{proposition}[Moments about zero of $\chi'^{\,2}_0 \left(\lambda \right)$]
\label{propo:mom.ncchisq.zero}
Let $Y'_{pnc}$ have a $\chi'^{\,2}_0 \left(\lambda \right)$ distribution. Then, for every $r \in \mathbb{N}$, the $r$-th moment about zero of $Y'_{pnc}$ can be written as:
\begin{equation}
\mathbb{E}\left[\left(Y'_{pnc} \right)^r \right]=2^r \, \sum_{i=0}^{r}\sum_{j=0}^{i} \left|s\left(r,i\right)\right|\mathcal{S}\left(i,j\right) \left(\frac{\lambda}{2}\right)^j,
\label{eq:mom.ncchisq.zero}
\end{equation}
where $\left|s\left(r,i\right)\right|$ is an unsigned Stirling number of the first kind and $\mathcal{S}\left(i,j\right)$ is a Stirling number of the second kind.
\end{proposition}
\begin{proof}
For the proof see~\ref{proof:mom.ncchisq.zero} in the Appendix.
\end{proof}

Finally, the comparison between Eq.~(\ref{eq:mom.literat.ncchisq}) and Eq.~(\ref{eq:mom.ncchisq}) leads to the following identity. The latter will be used in Section~\ref{subsec:dnc.beta.mom} in order to derive a new general formula for the moments about zero of the non-central beta distributions.

\begin{proposition}[Identity]
\label{propo:ident.mom}
Let $r \in \mathbb{N}$ and $h \in \mathbb{R} - \{0\}$. Then:
\begin{equation}
{r \choose j} \frac{\left(h\right)_r}{\left(h\right)_j}=\sum_{i=j}^{r} \mathcal{S}\left(i,j\right) \frac{1}{i!}\left[\frac{d^i}{d h^i} \left(h\right)_r\right], \quad \forall j=0,\ldots,r,
\label{eq:ident.mom}
\end{equation}
where $\left(h\right)_r$ is defined as in Eq.~(\ref{eq:poch.symb}) and $\mathcal{S}\left(i,j\right)$ is a Stirling number of the second kind.
\end{proposition}
\begin{proof}
For the proof see~\ref{proof:ident.mom} in the Appendix.
\end{proof}

\section{The non-central beta distributions}
\label{sec:dnc.beta.def.repres}

\subsection{Definitions and representations}
\label{subsec:dnc.beta.def.repres}

It is well known that if $Y_i$, $i=1,2$, are independent chi-squared random variables with $2 \alpha_i>0$ degrees of freedom, then the random variable:
\begin{equation}
X=\frac{Y_1}{Y_1+Y_2}
\label{eq:def.beta}
\end{equation}
has a beta distribution with shape parameters $\alpha_1$, $\alpha_2$, denoted by Beta$(\alpha_1, \alpha_2)$. We point out that a $\mbox{Beta}\left(\alpha_1,0\right)$ random variable with $\alpha_1>0$ is degenerate at one: in fact, the chi-squared random variable present only at denominator in Eq.~(\ref{eq:def.beta}) is degenerate at zero. Similarly, a $\mbox{Beta}\left(0,\alpha_2\right)$ random variable with $\alpha_2>0$ is degenerate at zero: in fact, the chi-squared random variable present at both numerator and denominator in Eq.~(\ref{eq:def.beta}) is degenerate at zero, too. Then, we recall that the beta density function takes the following form:
\begin{equation}
\mbox{Beta}\left(x;\alpha_1,\alpha_2\right)=\frac{x^{\alpha_1-1} \, \left(1-x\right)^{\alpha_2-1}}{B\left(\alpha_1,\alpha_2\right)}, \quad 0<x<1 \; .
\label{eq:dens.beta}
\end{equation}

That said, in order to go into the matter of interest in due depth, we recall herein a characterizing property of independent chi-squared (and, more generally, gamma) random variables. The latter is a matter of great consequence for our interests; in fact, it will be largely used in the derivation of all the results proved in the sequel. Thus, it is given a special reference.

\begin{property}[Characterizing property of independent $\chi^2$ random variables]
\label{prope:char.prop.chisq}
$Y_i$, $i=1,2$, are independent chi-squared random variables if and only if the compositional ratio $X=Y_1/\left(Y_1+Y_2\right)$ is independent of $Y_1+Y_2$.
\end{property}

By replacing the two chi-squared random variables involved in Eq.~(\ref{eq:def.beta}) with two independent non-central ones, we obtain the definition of the ``doubly'' non-central beta distribution, that is the most general non-central extension of the beta one. The latter is defined as follows.

Let $Y'_i$, $i=1,2$, be independent $\chi'^{\,2}_{2\alpha_i}\left(\lambda_i\right)$ random variables. Then, a random variable is said to have a doubly non-central beta distribution with shape parameters $\alpha_1,\alpha_2$ and non-centrality parameters $\lambda_1,\lambda_2$, denoted by $\mbox{\normalfont{B}}''\left(\alpha_1,\alpha_2,\lambda_1,\lambda_2\right)$, if it is distributed as
\begin{equation}
X'=\frac{Y'_1}{Y'_1+Y'_2}
\label{eq:def.dnc.beta}
\end{equation}
\cite{JohKotBal95}. The case $\lambda_1=\lambda_2=0$ clearly corresponds to the beta distribution. Moreover, by taking $\alpha_1=\alpha_2=0$ in Eq.~(\ref{eq:def.dnc.beta}), the latter degenerates into the compositional ratio $X'_{pnc}$ of two purely non-central chi-squared independent random variables with non-centrality parameters $\lambda_1$, $\lambda_2$. Its distribution is denoted by $\mbox{\normalfont{B}}''\left(0,0,\lambda_1,\lambda_2\right)$. 

The $\mbox{B}''$ density can be easily derived by using the mixture representation of the non-central chi-squared distribution. Specifically, let $M_i$, $i=1,2$, be independent Poisson random variables with means $\lambda_i/2$. Conditionally on $(M_1,M_2)$, $X'$ has a Beta$(\alpha_1+M_1, \alpha_2 + M_2)$ distribution, $Y'_i|(M_1,M_2)$ being independent with distributions $\chi^2_{2\alpha_i + 2 M_i}$, $i=1,2$. Therefore, the density function $f_{X'}$ of $X' \sim \mbox{B}''\left(\alpha_1,\alpha_2,\lambda_1,\lambda_2\right)$ can be stated as:
\begin{eqnarray}
\lefteqn{f_{X'}\left(x;\alpha_1,\alpha_2,\lambda_1,\lambda_2\right)=}\nonumber \\
& = & \sum_{j=0}^{+\infty} \sum_{k=0}^{+\infty} \frac{e^{-\frac{\lambda_1}{2}} \left(\frac{\lambda_1}{2}\right)^j}{j!} \frac{e^{-\frac{\lambda_2}{2}} \left(\frac{\lambda_2}{2}\right)^k}{k!} \frac{x^{\alpha_1+j-1} \left(1-x\right)^{\alpha_2+k-1}}{B\left(\alpha_1+j,\alpha_2+k\right)},\quad 0<x<1,
\label{eq:dens.beta.dnc}
\end{eqnarray}
i.e. as the double series of the Beta$(\alpha_1+j,\alpha_2+k)$ densities, $j,k \in \mathbb{N} \cup \{0\}$, weighted by the joint probabilities of the bivariate random variable $\left(M_1,M_2\right)$, where $M_i$, $i=1,2$, are independent with $\mbox{Poisson}\left(\lambda_i/2\right)$ distributions.

By analogy with the density, the distribution function $F_{X'}$ of $X' \sim \mbox{\normalfont{B}}''\left(\alpha_1,\alpha_2,\lambda_1,\lambda_2\right)$ can be stated as:
\begin{eqnarray}
\lefteqn{F_{X'}\left(x;\alpha_1,\alpha_2,\lambda_1,\lambda_2\right)=} \nonumber \\
& = & \sum_{j=0}^{+\infty} \sum_{k=0}^{+\infty} \frac{e^{-\frac{\lambda_1}{2}} \left(\frac{\lambda_1}{2}\right)^j}{j!} \frac{e^{-\frac{\lambda_2}{2}} \left(\frac{\lambda_2}{2}\right)^k}{k!} \frac{B\left(x; \, \alpha_1+j,\alpha_2+k\right)}{B\left(\alpha_1+j,\alpha_2+k\right)},\quad 0<x<1,
\label{eq:distr.beta.dnc}
\end{eqnarray}
i.e. as the double series of the Beta$(\alpha_1+j,\alpha_2+k)$ distribution functions, $j,k \in \mathbb{N} \cup \{0\}$, weighted by the joint probabilities of the bivariate random variable $\left(M_1,M_2\right)$, where $M_i$, $i=1,2$, are independent with $\mbox{Poisson}\left(\lambda_i/2\right)$ distributions. We recall that $B\left(x; \, a,b\right)=\int_{0}^{x} t^{a-1} \, \left(1-t\right)^{b-1} \, dt$ is the incomplete beta function. An implementation of Eq.~(\ref{eq:distr.beta.dnc}) in \texttt{R} language is proposed in \ref{funct:int.pdncbeta}, \ref{funct:pdncbeta}.

The above discussion directly leads to the following mixture representation.
\begin{property}[Mixture representation of B$''$]
\label{prope:mixrepres.dncb}
Let $X'$ have a $\mbox{\normalfont{B}}''\left(\alpha_1,\alpha_2,\lambda_1,\lambda_2\right)$ distribution and $M_i$, $i=1,2$, be independent Poisson random variables with means $\lambda_i/2$. Then, $X'$ admits the following representation:
\begin{equation}
X'\,|\left(M_1, \, M_2\right) \sim \mbox{\normalfont{Beta}}\left(\alpha_1+M_1, \alpha_2+M_2\right) \; .
\label{eq:mixrepres.dncb}
\end{equation}
\end{property}

In view of the foregoing arguments, it's clear that Property~\ref{prope:char.prop.chisq} is no longer valid in the non-central setting. However, an interesting generalization of the latter holds true. As a matter of fact, a doubly non-central beta random variable is herein proved to be independent of the sum of the two non-central chi-squared random variables involved in its definition in a suitable conditional form. More precisely, in the notation of Eq.~(\ref{eq:def.dnc.beta}), the latter occurs conditionally on the sum $M^+$ of the two Poisson random variables on which both $X'$ and $Y'_1+Y'_2$ depend. As a side effect, the distribution of $X'$ given $M^+$ is also obtained.   

\begin{proposition}[Conditional independence]
\label{propo:dncb.cond.ind}
Let $X' \sim \mbox{\normalfont{B}}''\left(\alpha_1,\alpha_2,\lambda_1,\lambda_2\right)$ and $Y'_i$, $i=1,2$, be independent $\chi'^{\,2}_{2\alpha_i}\left(\lambda_i\right)$ random variables, with $Y'^+=Y'_1+Y'_2$. Furthermore, let $M_i$, $i=1,2$, be independent Poisson random variables with means $\lambda_i/2$ and $M^+=M_1+M_2$. Then, $X'$ and $Y'^+$ are conditionally independent given $M^+$ and the density of $X'$ given $M^+$ is:
\begin{equation}
f_{\left.X'\right|M^+}\left(x\right)=\sum_{i=0}^{M^+} \mbox{\normalfont{Binomial}}\left(i;M^+,\frac{\lambda_1}{\lambda^+}\right) \cdot \mbox{\normalfont{Beta}}\left(x;\alpha_1+i,\alpha_2+M^+-i\right),
\label{eq:dncb.distr.cond.m}
\end{equation}
where:
\begin{equation}
\mbox{\normalfont{Binomial}}\left(i;M^+,\frac{\lambda_1}{\lambda^+}\right)={M^+ \choose i} \left(\frac{\lambda_1}{\lambda^+}\right)^i \left(1-\frac{\lambda_1}{\lambda^+}\right)^{M^+-i}, \qquad i=0,\ldots,M^+.
\label{eq:notat.binom.distr}
\end{equation}
\end{proposition}
\begin{proof}
For the proof see~\ref{proof:dncb.cond.ind} in the Appendix.
\end{proof}
 
The doubly non-central beta density can be equivalently written as a perturbation of the corresponding central case, i.e. the beta one, as follows.
\begin{proposition}[Perturbation representation of B$''$]
\label{propo:pertrepres.dncb}
Let $X'$ have a $\mbox{\normalfont{B}}''$ distribution with shape parameters $\alpha_1$, $\alpha_2$ and non-centrality parameters $\lambda_1$, $\lambda_2$. Then, the density $f_{X'}$ of $X'$ can be written as:
\begin{equation}
f_{X'}\left(x;\alpha_1,\alpha_2,\lambda_1,\lambda_2\right)= \mbox{\normalfont{Beta}}\left(x; \alpha_1,\alpha_2 \right) \cdot e^{-\frac{\lambda^+}{2}} \,  \Psi_2\left[\alpha^+;\alpha_1,\alpha_2;\frac{\lambda_1}{2}x,\frac{\lambda_2 }{2}\left(1-x\right)\right] \; ,
\label{eq:c.ncb}
\end{equation}
where $\mbox{\normalfont{Beta}}\left(x;\alpha_1,\alpha_2\right)$ is defined as in Eq.~(\ref{eq:dens.beta}), $\alpha^+=\alpha_1+\alpha_2$, $\lambda^+=\lambda_1+\lambda_2$ and
\begin{equation}
\Psi_2\left[\alpha;\gamma,\gamma';x,y\right]=\sum_{j=0}^{+\infty}\sum_{k=0}^{+\infty}\frac{(\alpha)_{j+k}}{(\gamma)_j \, (\gamma')_k}  \frac{x^j}{j!}\frac{y^k}{k!}, \quad x,y \geq 0
\label{eq:perturb.ncb}
\end{equation}
is the Humbert's confluent hypergeometric function \cite{SriKar85}.
\end{proposition}
\begin{proof}
For the proof see~\ref{proof:pertrepres.dncb} in the Appendix.
\end{proof}

Unfortunately, the perturbation representation of the doubly non-central beta density is not so easily tractable and interpretable. Indeed, Eq.~(\ref{eq:c.ncb}) shows that, unless a constant term, the beta density is perturbed by a function in two variables given by the sum of the double power series reported in Eq.~(\ref{eq:perturb.ncb}). The latter has not a simple behavior and, to our knowledge, is not reducible into a more tractable analytical form. Therefore, the effect of such perturbation is not easy to understand. However, it can be clearly seen when $\alpha_1=\alpha_2=1$, because in this case the beta density reduces to the uniform one (see Section~\ref{subsec:dnc.beta.dens.plots}). In this regard, note that, in light of Eq.~(\ref{eq:poch.symb.sum}), one obtains:
\begin{eqnarray}
\lefteqn{\Psi_2\left[\alpha^+;\alpha_1,\alpha_2;\frac{\lambda_1}{2}x,\frac{\lambda_2 }{2}\left(1-x\right)\right]=} \nonumber \\
& = & \sum_{j=0}^{+\infty} \sum_{k=0}^{+\infty} \frac{\left(\alpha^+\right)_{j+k}}{\left(\alpha_1\right)_j \left(\alpha_2\right)_k} \frac{\left(\frac{\lambda_1}{2}x\right)^j}{j!} \frac{\left[\frac{\lambda_2}{2}(1-x)\right]^k}{k!}= \nonumber \\
& = & \sum_{j=0}^{+\infty} \frac{\left(\alpha^+\right)_j}{\left(\alpha_1\right)_j} \frac{\left(\frac{\lambda_1}{2}x\right)^j}{j!} \sum_{k=0}^{+\infty} \frac{\left(\alpha^++j\right)_{k}}{\left(\alpha_2\right)_k} \frac{\left[\frac{\lambda_2}{2}(1-x)\right]^k}{k!}= \nonumber \\
& = & \sum_{j=0}^{+\infty} \frac{\left(\alpha^+\right)_j}{\left(\alpha_1\right)_j} \frac{\left(\frac{\lambda_1}{2}x\right)^j}{j!} \, _1 F_1\left[\alpha^++j;\alpha_2;\frac{\lambda_2}{2} \, \left(1-x\right)\right],
\label{eq:perturb.ncb.alter}
\end{eqnarray}
where $_1F_1\left(a;b;x\right)=\sum_{k=0}^{+\infty}\frac{(a)_k}{(b)_k}\frac{x^k}{k!}$ is the Kummer's confluent hypergeometric function \cite{SriKar85}. From Eq.~(\ref{eq:perturb.ncb.alter}) it's immediate to see that the $\Psi_2$ function can be equivalently expressed as a series of weighted Kummer's confluent hypergeometric functions. Therefore, this formula can be usefully adopted as a natural basis for implementing such a function in any statistical package where the $_1F_1$ function is already implemented, for instance the \texttt{R} programming environment. In this regard, an implementation of Eq.~(\ref{eq:perturb.ncb.alter}) in \texttt{R} language is proposed in \ref{funct:dperturb}.

A new representation of $X'\sim\mbox{\normalfont{B}}''\left(\alpha_1,\alpha_2,\lambda_1,\lambda_2\right)$ is now introduced. According to the latter, a doubly non-central beta random variable can be expressed in terms of a convex linear combination of a central component and a purely non-central one. These two additive components are given random weights that can be fully understood by recalling the type 1 and the type 2 non-central beta distributions, namely two special cases of the doubly non-central one. The latter are now briefly recalled but for more details the reader can refer for example to \cite{NadGup04}.

If two random variables $Y'_1$ and $Y_2$ are independently distributed according to $\chi'^{\,2}_{2\alpha_1}\left(\lambda\right)$ and $\chi^{2}_{2\alpha_2}$, respectively, then the random variable:
\begin{equation}
X'_1=\frac{Y'_1}{Y'_1+Y_2}
\label{eq:def.nc1.beta}
\end{equation}
is said to have a type 1 non-central beta distribution, denoted by $\mbox{\normalfont{B}}'_1\left(\alpha_1,\alpha_2,\lambda\right)$. The density function $f_{X'_1}$ of $X'_1 \sim \mbox{\normalfont{B}}'_1\left(\alpha_1,\alpha_2,\lambda\right)$ can be derived by means of a reasoning analogous to the one leading to Eq.~(\ref{eq:dens.beta.dnc}) and it is given by:
\begin{equation}
f_{X'_1}\left(x_1;\alpha_1,\alpha_2,\lambda\right)=\sum_{j=0}^{+\infty} \frac{e^{-\frac{\lambda}{2}} \left(\frac{\lambda}{2}\right)^j}{j!} \frac{x_1^{\alpha_1+j-1} \left(1-x_1\right)^{\alpha_2-1}}{B\left(\alpha_1+j,\alpha_2\right)},\quad 0<x_1<1,
\label{eq:dens.beta.nc1}
\end{equation}
i.e. the series of the Beta$(\alpha_1+j,\alpha_2)$ densities, $j \in \mathbb{N} \cup \{0\}$, weighted by the probabilities of $M \sim \mbox{Poisson}\left(\lambda/2\right)$. Roughly speaking, Eq.~(\ref{eq:dens.beta.nc1}) can be intuitively established by taking $\lambda_2=0$ and renaming $\lambda_1$ with $\lambda$ in Eq.~(\ref{eq:dens.beta.dnc}). Such a distribution admits the following mixture and perturbation representations.

\begin{property}[Mixture representation of $\mbox{\normalfont{B}}'_1$]
\label{prope:mixrepres.ncb1}
Let $X'_1$ have a $\mbox{\normalfont{B}}'_1\left(\alpha_1,\alpha_2,\lambda\right)$ distribution and $M$ be a Poisson random variable with mean $\lambda/2$. Then, $X'_1$ admits the following representation: $X'_1| \, M \sim \mbox{\normalfont{Beta}}\left(\alpha_1+M, \alpha_2\right)$.
\end{property}

\begin{proposition}[Perturbation representation of $\mbox{\normalfont{B}}'_1$]
\label{propo:pertrepres.ncb1}
Let $X'_1\sim\mbox{\normalfont{B}}'_1\left(\alpha_1,\alpha_2,\lambda\right)$ and $\alpha^+=\alpha_1+\alpha_2$. Then, the density $f_{X'_1}$ of $X'_1$ can be written as:
\begin{equation}
f_{X'_1}\left(x_1;\alpha_1,\alpha_2,\lambda\right)= \mbox{\normalfont{Beta}}\left(x_1; \alpha_1,\alpha_2 \right) \cdot e^{-\frac{\lambda}{2}} \, _1 F_1\left(\alpha^+;\alpha_1;\frac{\lambda}{2} \, x_1\right).
\label{eq:c.ncb1}
\end{equation}
\end{proposition}
\begin{proof}
The proof follows the same lines as the proof of Proposition~\ref{propo:pertrepres.dncb}.
\end{proof}

By integrating Eq.~(\ref{eq:dens.beta.nc1}) or, roughly speaking, by taking $\lambda_2=0$ and renaming $\lambda_1$ with $\lambda$ in Eq.~(\ref{eq:distr.beta.dnc}), the distribution function $F_{X'_1}$ of $X'_1 \sim \mbox{\normalfont{B}}'_1\left(\alpha_1,\alpha_2,\lambda\right)$ can be obtained as follows:
$$
F_{X'_1}\left(x_1;\alpha_1,\alpha_2,\lambda\right)=\sum_{j=0}^{+\infty} \frac{e^{-\frac{\lambda}{2}} \left(\frac{\lambda}{2}\right)^j}{j!} \frac{B\left(x_1; \, \alpha_1+j,\alpha_2\right)}{B\left(\alpha_1+j,\alpha_2\right)},\quad 0<x_1<1,
$$
i.e., by analogy with the $\mbox{\normalfont{B}}'_1$ density, as the series of the Beta$(\alpha_1+j,\alpha_2)$ distribution functions, $j \in \mathbb{N} \cup \{0\}$, weighted by the probabilities of $M \sim \mbox{Poisson}\left(\lambda/2\right)$. As previously said, the case of $\alpha_1=\alpha_2=1$ is hugely important in this context. In the latter, the $\mbox{B}'_1$ density becomes significantly easier. Indeed, by considering Eq.~(\ref{eq:c.ncb1}) and letting $a=2$, $z=\frac{\lambda}{2}x$ in the following formula (\href{http://functions.wolfram.com/HypergeometricFunctions/Hypergeometric1F1/03/01/02/}{link}):
\begin{equation}
_1F_1\left(a;a-1;z\right)=e^z \left(1+\frac{z}{a-1}\right),
\label{eq:form1.1f1}
\end{equation}
we obtain:
\begin{equation}
f_{X'_1}\left(x_1;1,1,\lambda\right)=e^{-\frac{\lambda}{2}} \, _1F_1\left(2;1;\frac{\lambda}{2}x_1\right)=e^{-\frac{\lambda}{2}\left(1-x_1\right)} \left(1+\frac{\lambda}{2}x_1\right), \quad 0<x_1<1.
\label{eq:c.ncb1.11}
\end{equation}
Hence, by integrating Eq.~(\ref{eq:c.ncb1.11}), it's easy to see that the $\mbox{B}'_1$ distribution function turns out to be:
$$
F_{X'_1}\left(x_1;1,1,\lambda\right)=x_1 \, e^{-\frac{\lambda}{2}\left(1-x_1\right)}, \qquad 0<x_1<1.
$$ 

Finally, the type 2 non-central beta, denoted by $\mbox{\normalfont{B}}'_2\left(\alpha_1,\alpha_2,\lambda\right)$, is the distribution of the random variable:
\begin{equation}
X'_2=\frac{Y_1}{Y_1+Y'_2},
\label{eq:def.nc2.beta}
\end{equation}
where $Y_1$ and $Y'_2$ are independently distributed according to $\chi^{2}_{2\alpha_1}$ and $\chi'^{\,2}_{2\alpha_2}\left(\lambda\right)$, respectively.

The type 1 and the type 2 non-central beta random variables are connected by the following relationship.

\begin{property}[Relationship between $\mbox{\normalfont{B}}'_1$ and $\mbox{\normalfont{B}}'_2$]
\label{prope:relat.ncb1.ncb2}
Let $X'_2 \sim \mbox{\normalfont{B}}'_2\left(\alpha_1,\alpha_2,\lambda\right)$ and $X'_1 \sim \mbox{\normalfont{B}}'_1\left(\alpha_2,\alpha_1,\lambda\right)$. Then:
\begin{equation}
X'_1=1-X'_2.
\label{eq:relat.ncb1.ncb2}
\end{equation}
\end{property}
\begin{proof}
For the proof see~\ref{proof:relat.ncb1.ncb2} in the Appendix.
\end{proof}

Hence, the density function $f_{X'_2}$ of $X'_2 \sim \mbox{\normalfont{B}}'_2\left(\alpha_1,\alpha_2,\lambda\right)$ can be easily derived by making use of the transformation of variable in Eq.~(\ref{eq:relat.ncb1.ncb2}) and it is given by:
\begin{equation}
f_{X'_2}\left(x_2;\alpha_1,\alpha_2,\lambda\right)=\sum_{k=0}^{+\infty} \frac{e^{-\frac{\lambda}{2}} \left(\frac{\lambda}{2}\right)^k}{k!} \frac{x_2^{\alpha_1-1} \left(1-x_2\right)^{\alpha_2+k-1}}{B\left(\alpha_1,\alpha_2+k\right)},\quad 0<x_2<1,
\label{eq:dens.beta.nc2}
\end{equation}
i.e. the series of the Beta$(\alpha_1,\alpha_2+k)$ densities, $k \in \mathbb{N} \cup \{0\}$, weighted by the probabilities of $M \sim \mbox{Poisson}\left(\lambda/2\right)$. Roughly speaking, Eq.~(\ref{eq:dens.beta.nc2}) can be intuitively established by taking $\lambda_1=0$ and renaming $\lambda_2$ with $\lambda$ in Eq.~(\ref{eq:dens.beta.dnc}). Such a distribution admits the following mixture and perturbation representations.

\begin{property}[Mixture representation of $\mbox{\normalfont{B}}'_2$]
\label{prope:mixrepres.ncb2}
Let $X'_2$ have a $\mbox{\normalfont{B}}'_2\left(\alpha_1,\alpha_2,\lambda\right)$ distribution and $M$ be a Poisson random variable with mean $\lambda/2$. Then, $X'_2$ admits the following representation:
\begin{equation}
X'_2| \, M \sim \mbox{\normalfont{Beta}}\left(\alpha_1, \alpha_2+M\right).
\label{eq:mixrepres.ncb2}
\end{equation}
\end{property}

\begin{proposition}[Perturbation representation of $\mbox{\normalfont{B}}'_2$]
\label{propo:pertrepres.ncb2}
Let $X'_2\sim\mbox{\normalfont{B}}'_2\left(\alpha_1,\alpha_2,\lambda\right)$. Then, the density $f_{X'_2}$ of $X'_2$ can be written as:
\begin{equation}
f_{X'_2}\left(x_2;\alpha_1,\alpha_2,\lambda\right)= \mbox{\normalfont{Beta}}\left(x_2; \alpha_1,\alpha_2 \right) \cdot e^{-\frac{\lambda}{2}} \, _1 F_1\left[\alpha^+;\alpha_2;\frac{\lambda}{2}\left(1-x_2\right)\right] \; .
\label{eq:c.ncb2}
\end{equation}
\end{proposition}
\begin{proof}
The proof follows the same lines as the proof of Proposition~\ref{propo:pertrepres.dncb}.
\end{proof}

By integrating Eq.~(\ref{eq:dens.beta.nc2}) or, roughly speaking, by taking $\lambda_1=0$ and renaming $\lambda_2$ with $\lambda$ in Eq.~(\ref{eq:distr.beta.dnc}), the distribution function $F_{X'_2}$ of $X'_2 \sim \mbox{\normalfont{B}}'_2\left(\alpha_1,\alpha_2,\lambda\right)$ can be obtained as follows:
$$
F_{X'_2}\left(x_2;\alpha_1,\alpha_2,\lambda\right)=\sum_{k=0}^{+\infty} \frac{e^{-\frac{\lambda}{2}} \left(\frac{\lambda}{2}\right)^k}{k!} \frac{B\left(x_2; \, \alpha_1,\alpha_2+k\right)}{B\left(\alpha_1,\alpha_2+k\right)},\quad 0<x_2<1,
$$
i.e., by analogy with the $\mbox{\normalfont{B}}'_2$ density, as the series of the Beta$(\alpha_1,\alpha_2+k)$ distribution functions, $k \in \mathbb{N} \cup \{0\}$, weighted by the probabilities of $M \sim \mbox{Poisson}\left(\lambda/2\right)$. Finally, in view of Property~\ref{prope:relat.ncb1.ncb2}, for $\alpha_1=\alpha_2=1$ we have that the $\mbox{B}'_2$ density takes on the following simple form:
\begin{equation}
f_{X'_2}\left(x_2;1,1,\lambda\right)=e^{-\frac{\lambda}{2} x_2} \left[1+\frac{\lambda}{2}\left(1-x_2\right)\right], \quad 0<x_2<1;
\label{eq:c.ncb2.11}
\end{equation}
moreover, the following holds true for the $\mbox{B}'_2$ distribution function:
$$
F_{X'_2}\left(x_2;1,1,\lambda\right)=1- e^{-\frac{\lambda}{2}x_2}\left(1-x_2\right), \qquad 0<x_2<1.
$$ 

That said, we are now ready to establish the aforementioned representation of a $\mbox{B}''$ random variable.

\begin{proposition}[Representation of $\mbox{\normalfont{B}}''$ as a convex linear combination]
\label{propo:rappr.clc.beta.dnc}
Let $X' \sim \mbox{\normalfont{B}}''\left(\alpha_1,\alpha_2,\lambda_1,\lambda_2\right)$ and $\alpha^+=\alpha_1+\alpha_2$, $\lambda^+=\lambda_1+\lambda_2$. Furthermore, let $M_r$, $r=1,2$, be independent Poisson random variables with means $\lambda_r/2$ and $M^+=M_1+M_2$. Then:
\begin{equation}
X'=X'_2 \, X +\left(1-X'_2\right) \,  X'_{pnc},
\label{eq:rappr.clc.beta.dnc}
\end{equation}
where:
\begin{itemize}
\item[i)] $X$ and $\left(X'_2,X'_{pnc}\right)$ are mutually independent and $X \sim \mbox{\normalfont{Beta}}\left(\alpha_1,\alpha_2\right)$, 
\item[ii)] $\left(X'_2, X'_{pnc}\right)$ are conditionally independent given $M^+$ with:
$$
\left.X'_2\right|M^+ \sim \mbox{\normalfont{Beta}}\left(\alpha^+,M^+\right),
$$
$$
\left.X'_{pnc}\right|M^+ \sim \sum_{i=0}^{M^+} \mbox{\normalfont{Binomial}}\left(i;M^+,\frac{\lambda_1}{\lambda^+}\right) \cdot \mbox{\normalfont{Beta}}\left(x;i,M^+-i\right),
$$
\item[iii)] $X'_2 \sim \mbox{\normalfont{B}}'_2\left(\alpha^+,0,\lambda^+\right)$ and $X'_{pnc} \sim \mbox{\normalfont{B}}''\left(0,0,\lambda_1,\lambda_2\right)$.
\end{itemize}
\end{proposition} 
\begin{proof}
For the proof see~\ref{proof:rappr.clc.beta.dnc} in the Appendix.
\end{proof}

It's apparent that the doubly non-central beta model can be easily simulated in different ways. Until now we have seen that this can be done by means of its definition, its mixture representation and its conditional distribution given $M^+$ in Eq.~(\ref{eq:dncb.distr.cond.m}). However, such issue can be alternatively addressed by resorting to the above proved representation.

To this end, in the notation of Proposition~\ref{propo:rappr.clc.beta.dnc}, it's necessary to generate the random variables $X$, $M^+$ and simulate accordingly from the distributions of $X'_2|M^+$ and $X'_{pnc}|M^+$. Firstly, it's to be noted that $X'_2|(M^+=0)$ is degenerate at one and $X'_{pnc}|(M^+=0)$ is degenerate at zero. Secondly, in case of $M^+ \neq 0$, the distribution of $X'_{pnc}|M^+$ is given by a mixture of $M^++1$ beta distributions, two of which have one shape parameter equal to zero. To sample from such mixture, one chooses an index $i^*$ from $\{0,\ldots,M^+\}$ according to the probabilities of the binomial distribution referred to hereinabove and then simulates a value from the corresponding $\mbox{Beta}\left(i^*,M^+-i^*\right)$ distribution. An implementation of this algorithm in \texttt{R} language is proposed in \ref{funct:rndncbeta}.
 
That said, Figures~\ref{fig:SIMUL2}, \ref{fig:SIMUL3}, \ref{fig:SIMUL5} show the results of the simulation from the $\mbox{B}''$ model for selected values of the shape and the non-centrality parameters. The generating process of the $\mbox{B}''$ random variate was carried out by means of two algorithms: the former is based on the definition while the latter on the new representation we have just introduced. In all the cases considered, the histogram of the simulated values was plotted together with the true density, thus anticipating the matters that will be discussed in the subsequent Section relating to the variety of shapes taken on by it. The graphs show that the results of the two approaches are indeed comparable.

Finally, in Section~\ref{subsec:dnc.beta.mom} the latter representation will be used in the derivation of an interesting expression for the mean of the doubly non-central beta distribution in terms of a convex linear combination of the mean of the beta distribution and a compositional ratio of the non-centrality parameters.

\begin{figure}[ht]
 \centering
 \subfigure
   {\includegraphics[width=6.3cm]{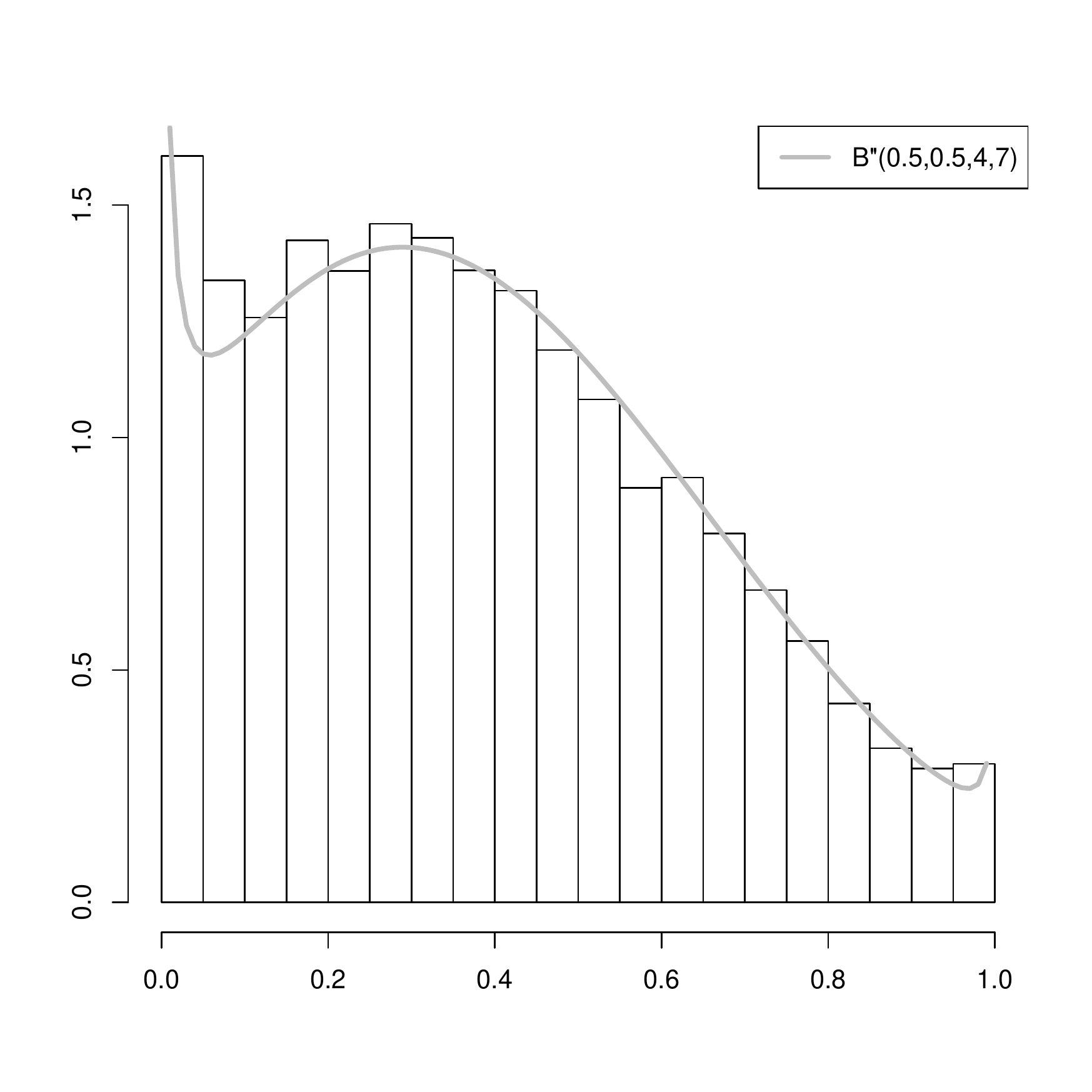}}
 \hspace{5mm}
 \subfigure
   {\includegraphics[width=6.3cm]{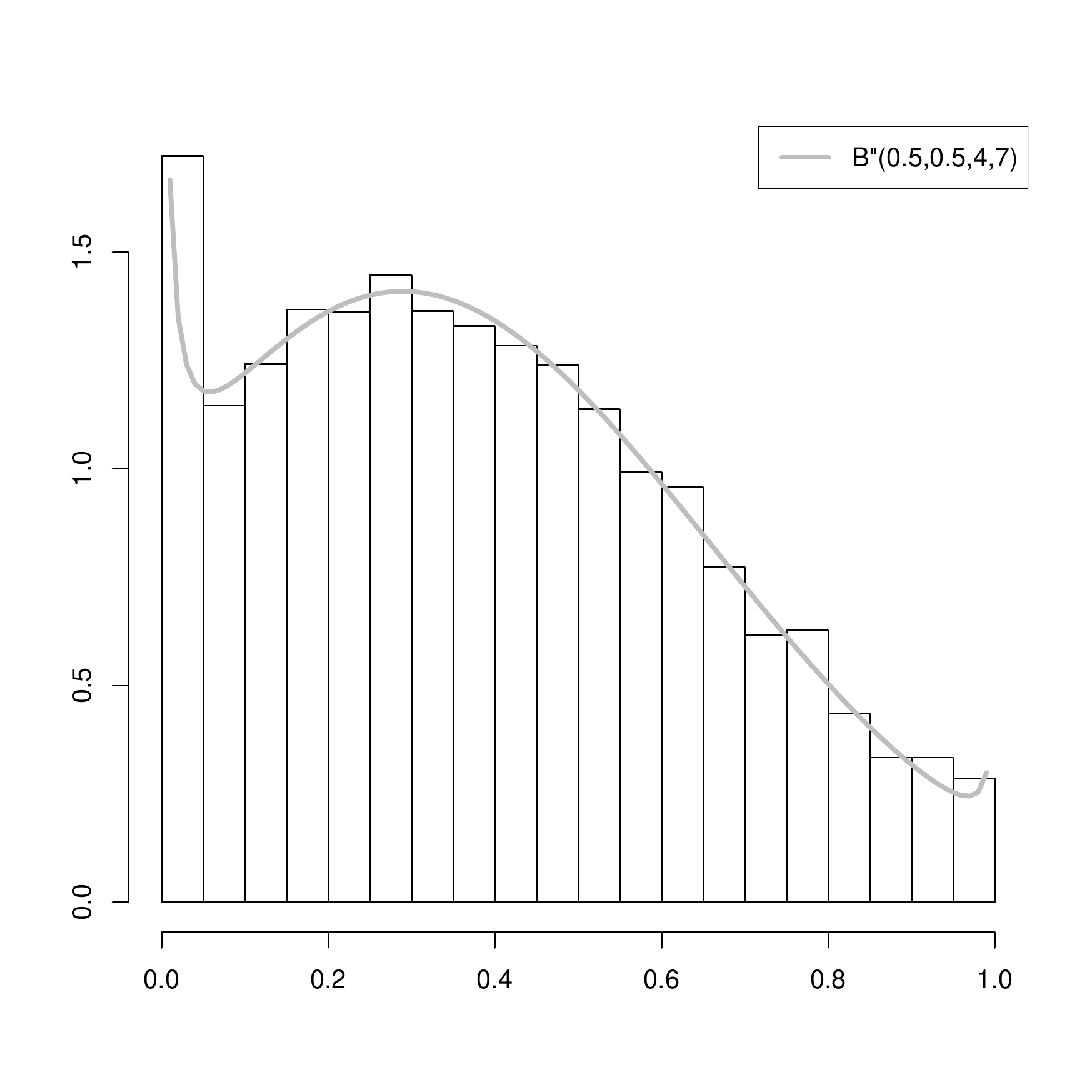}}
  \caption{Histogram of 10000 random draws simulated from $\mbox{B}''\left(0.5,0.5,4,7\right)$ by means of the algorithms based on its definition (left-hand panel) and its representation in Proposition~\ref{propo:rappr.clc.beta.dnc} (right-hand panel); the plot of the true density is superimposed in gray.}
 \label{fig:SIMUL2}
\end{figure}
 
\begin{figure}[ht]
 \centering
 \subfigure
   {\includegraphics[width=6.3cm]{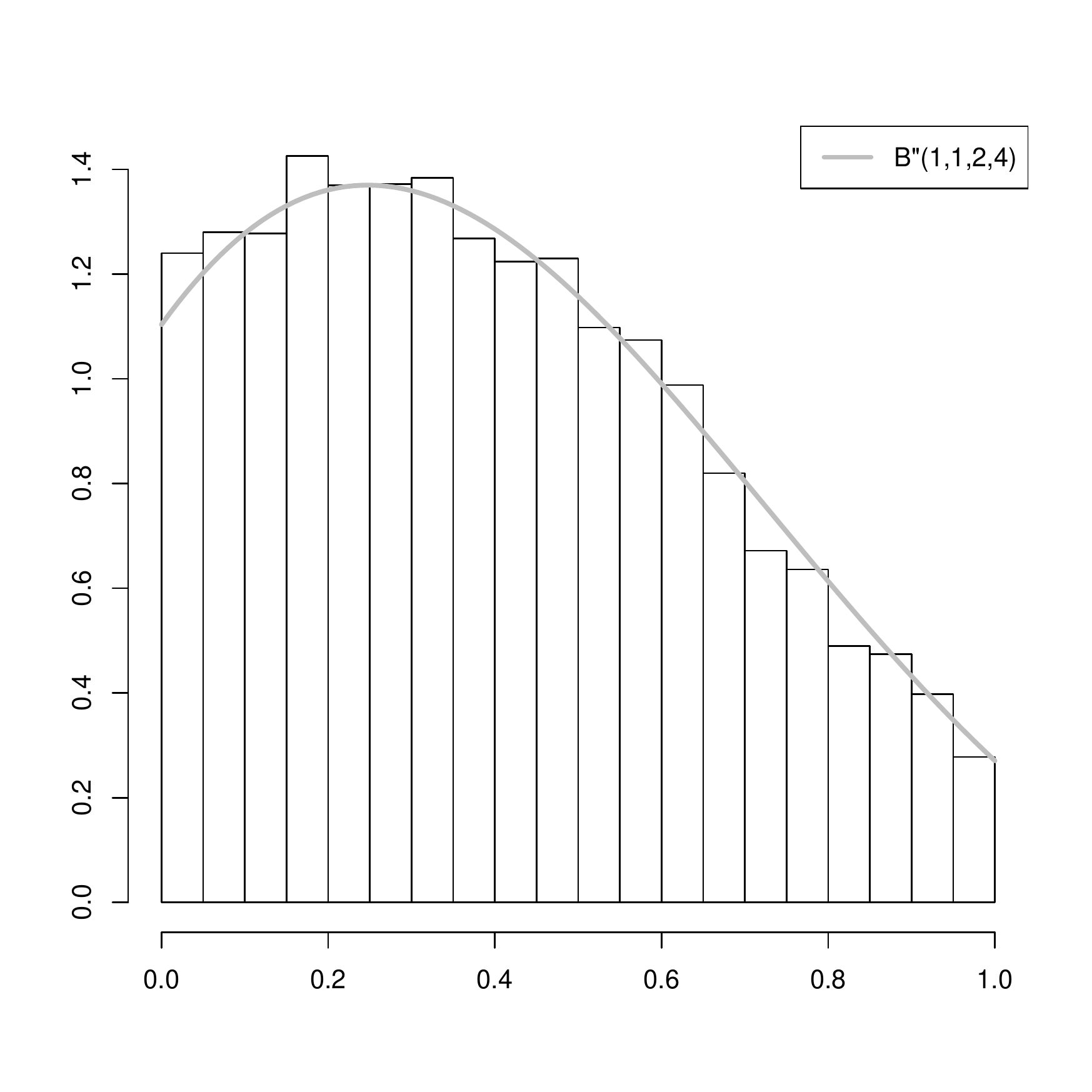}}
 \hspace{5mm}
 \subfigure
   {\includegraphics[width=6.3cm]{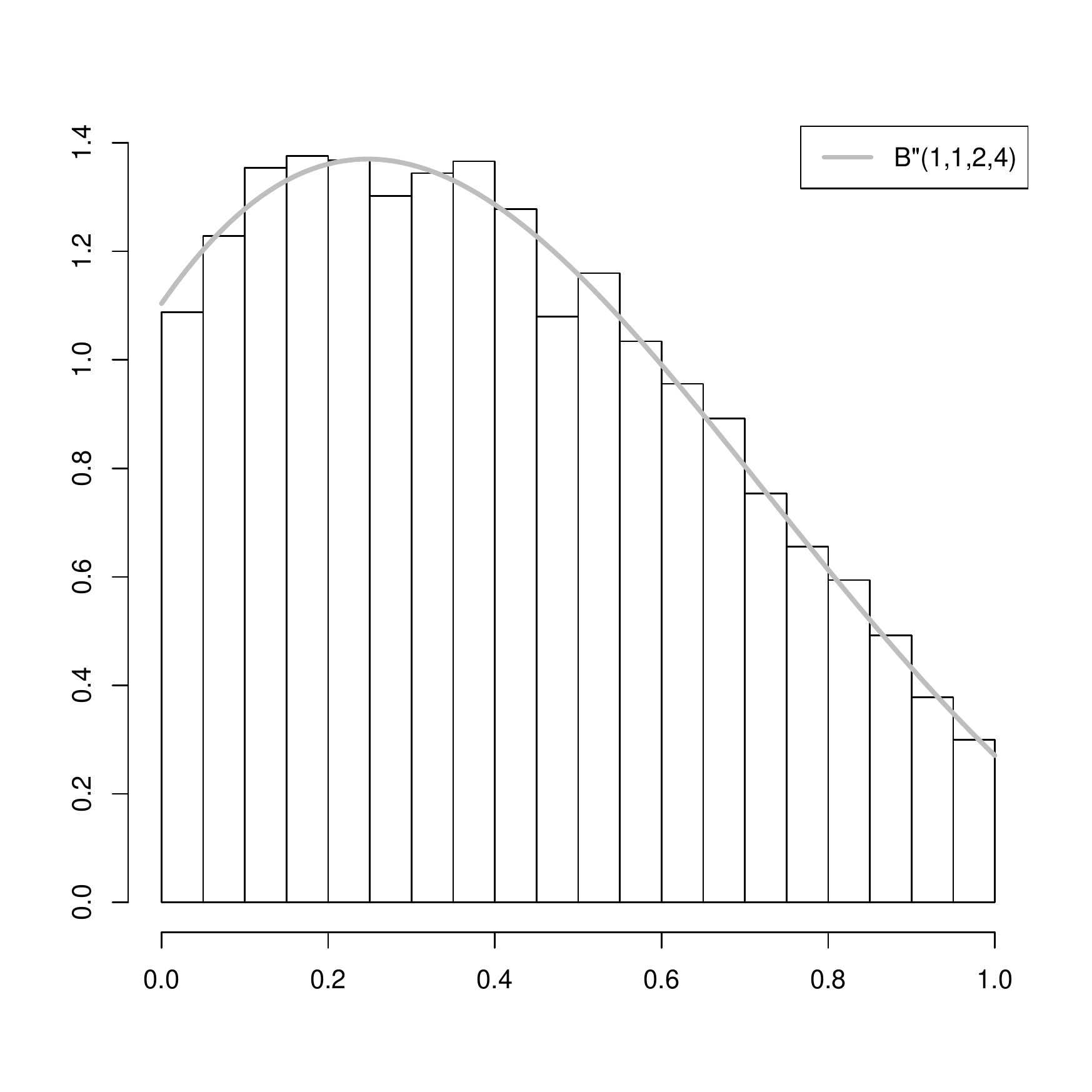}}
  \caption{Histogram of 10000 random draws simulated from $\mbox{B}''\left(1,1,2,4\right)$ by means of the algorithms based on its definition (left-hand panel) and its representation in Proposition~\ref{propo:rappr.clc.beta.dnc} (right-hand panel); the plot of the true density is superimposed in gray.}
 \label{fig:SIMUL3}
\end{figure}

\begin{figure}[ht]
 \centering 
 \subfigure
   {\includegraphics[width=6.3cm]{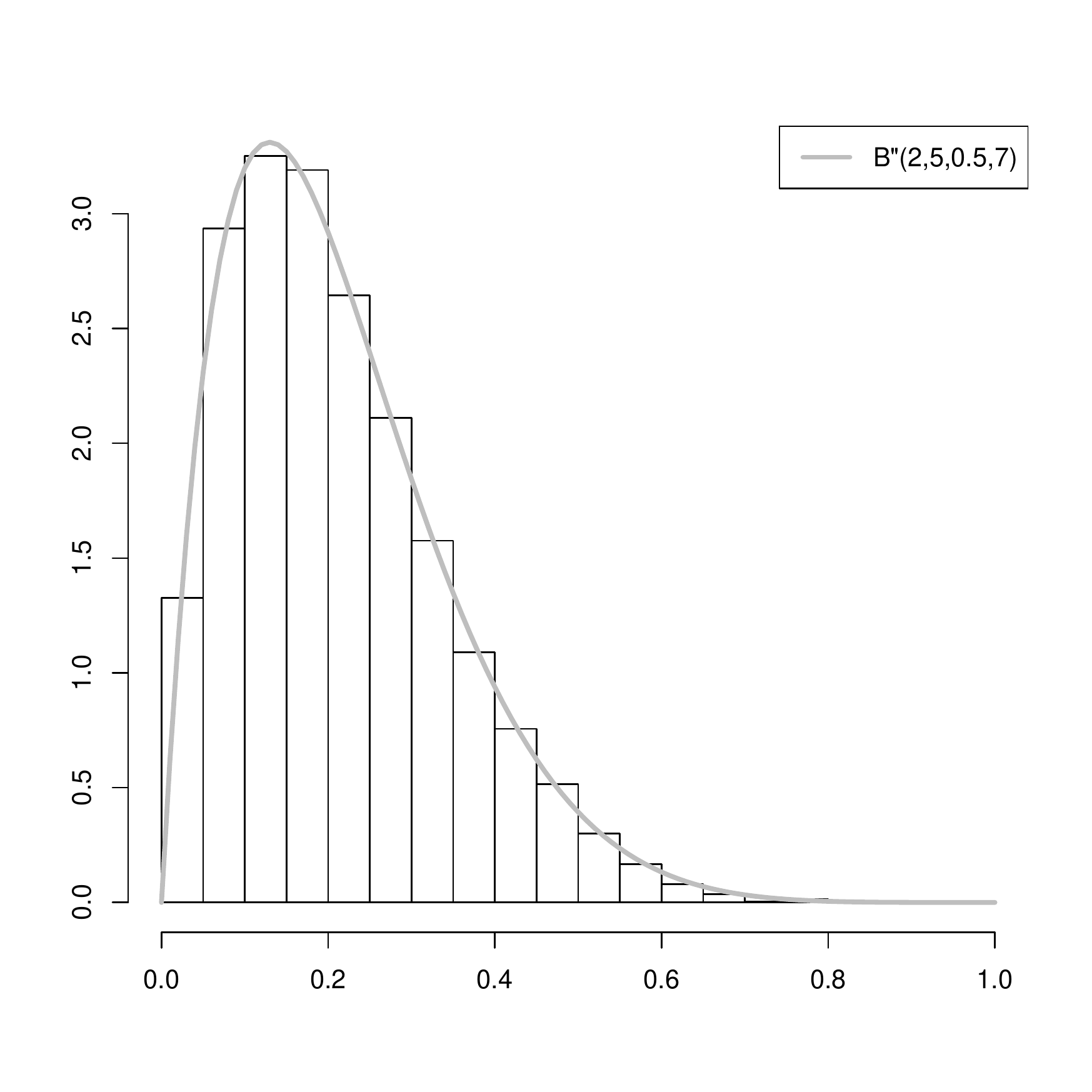}}
 \hspace{5mm}
 \subfigure
   {\includegraphics[width=6.3cm]{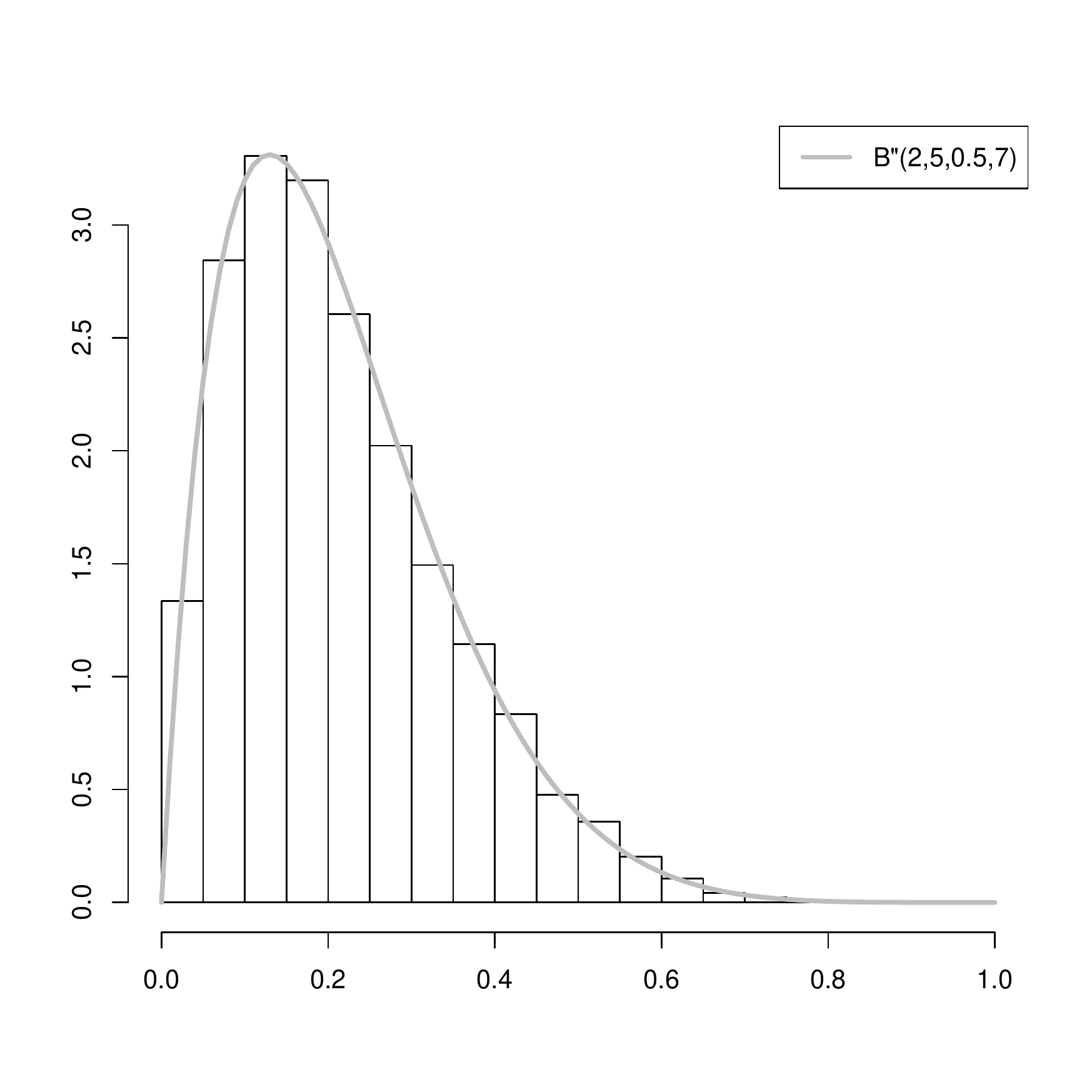}}
  \caption{Histogram of 10000 random draws simulated from $\mbox{B}''\left(2,5,0.5,7\right)$ by means of the algorithms based on its definition (left-hand panel) and its representation in Proposition~\ref{propo:rappr.clc.beta.dnc} (right-hand panel); the plot of the true density is superimposed in gray.}
 \label{fig:SIMUL5}
 \end{figure}

\subsection{Density plots}
\label{subsec:dnc.beta.dens.plots}

A key feature of the $\mbox{B}''$ distribution over the beta one lies in the much larger variety of shapes reachable by its density on varying the non-centrality parameters.

In this regard, it's worth recalling that by reversing both the shape and the non-centrality parameters the $\mbox{B}''$ density turns out to be symmetrical with respect to the midpoint of the interval $(0,1)$.

\begin{property}
\label{prope:relat.dncb1}
Let $X' \sim \mbox{\normalfont{B}}''\left(\alpha_1,\alpha_2,\lambda_1,\lambda_2\right)$. Then $1-X' \sim \mbox{\normalfont{B}}''\left(\alpha_2,\alpha_1,\lambda_2,\lambda_1\right)$.
\end{property}
\begin{proof}
For the proof see~\ref{proof:relat.dncb1} in the Appendix.
\end{proof}
As a special case of Property~\ref{prope:relat.dncb1}, the $\mbox{\normalfont{B}}''$ density with $\alpha_1=\alpha_2$ and $\lambda_1=\lambda_2$ is symmetrical with respect to $x=\frac{1}{2}$.

Some significant plots of the B$''$ density are displayed in the following Figures~\ref{fig:DNCB1},~\ref{fig:DNCB2} for selected values of the shape and the non-centrality parameters.

\begin{figure}[ht]
 \centering
 \subfigure
   {\includegraphics[width=6.3cm]{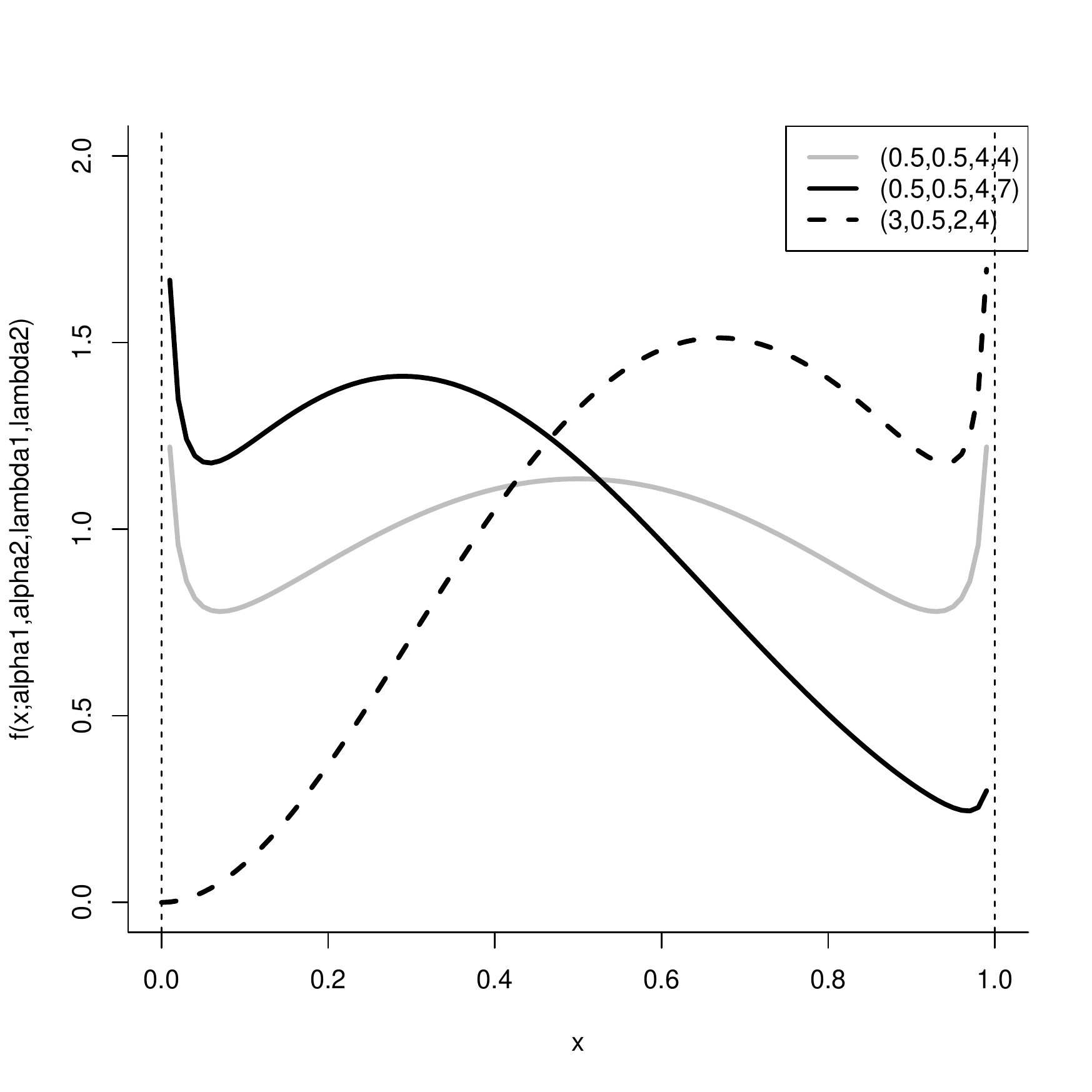}}
 \hspace{5mm}
 \subfigure
   {\includegraphics[width=6.3cm]{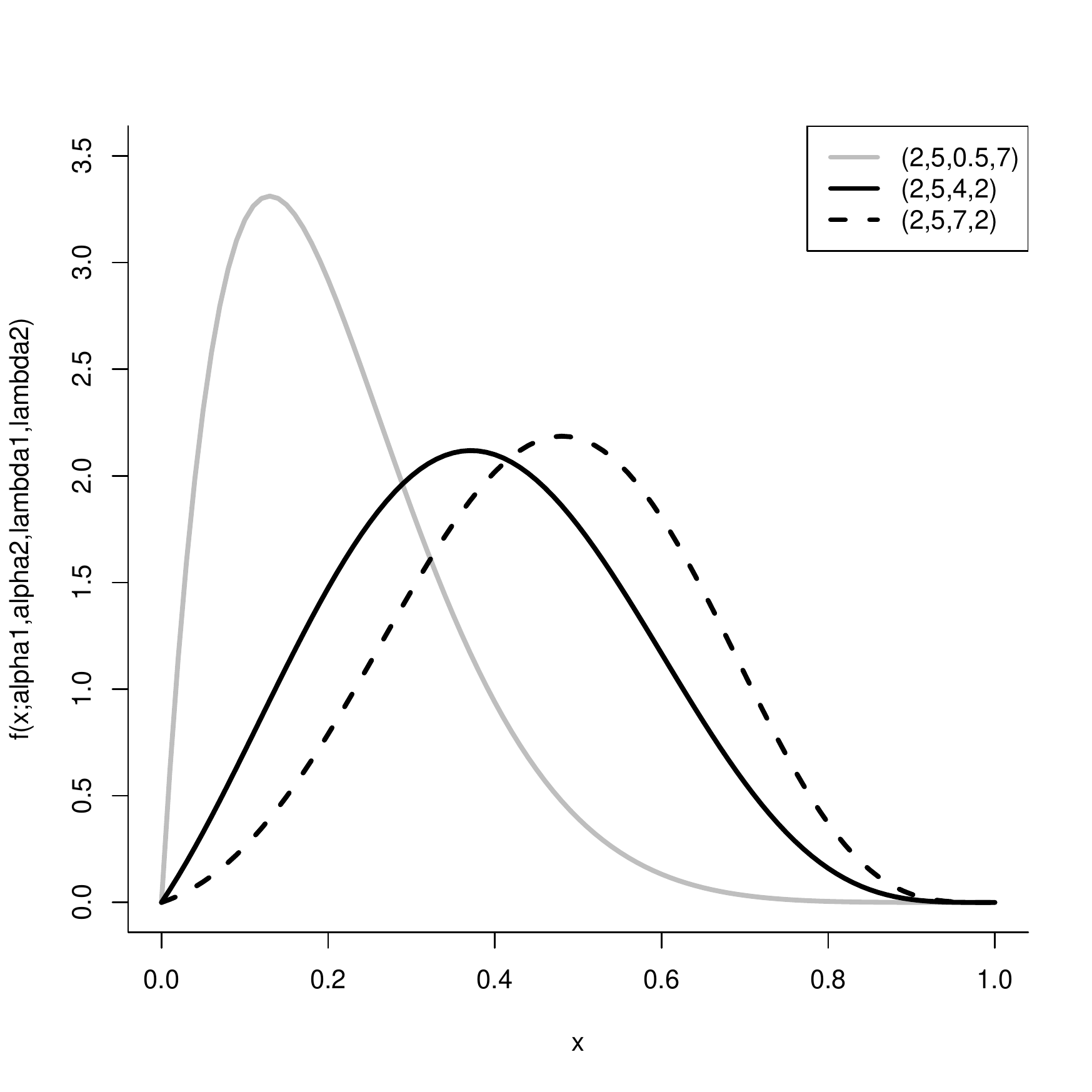}}
 \caption{Plots of the density of $X' \sim \mbox{\normalfont{B}}''\left(\alpha_1,\alpha_2,\lambda_1,\lambda_2\right)$ for selected values of $\left(\alpha_1,\alpha_2,\lambda_1,\lambda_2\right)$.}
 \label{fig:DNCB1}
\end{figure}

In this regard, we recall that the limits at $0$ and $1$ of the beta density are equal to $0$ or $+\infty$ if $\alpha_i \neq 1$ and are equal to $1$ if $\alpha_i=1$, $i=1,2$. In the latter case the beta density reduces to the uniform one. When $\alpha_i \neq 1$ the $\mbox{\normalfont{B}}''$ density shows the same limiting characteristics as the beta model ones. On the contrary, when $\alpha_i=1$ the doubly non-central beta density shows a more flexible behavior at the unit interval endpoints by taking on arbitrary finite and positive limits at $0$ and $1$. Some examples of this particularly relevant feature of the $\mbox{B}''$ density are shown in Figures~\ref{fig:DNCB2} for selected values of the non-centrality parameters $\lambda_1$, $\lambda_2$. Such peculiarity follows from Remark~\ref{prope:limit.0.nc.chisq} and was set out in \cite{OngOrs15}.

More specifically, the limits of the $\mbox{B}''$ density have the following expressions, that, interestingly, are really simple.

\begin{figure}[ht]
 \centering
 \subfigure
   {\includegraphics[width=6.3cm]{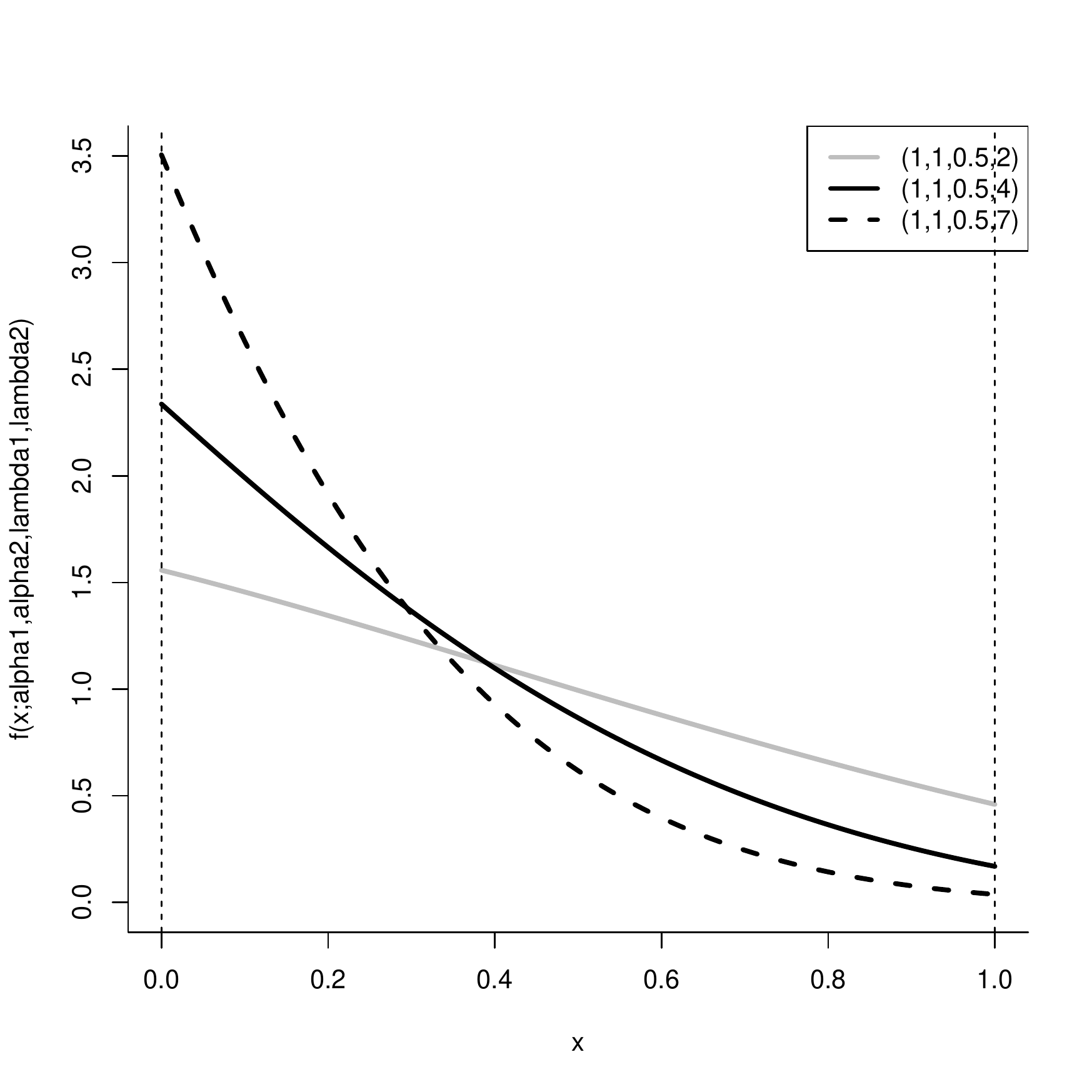}}
 \hspace{5mm}
 \subfigure
   {\includegraphics[width=6.3cm]{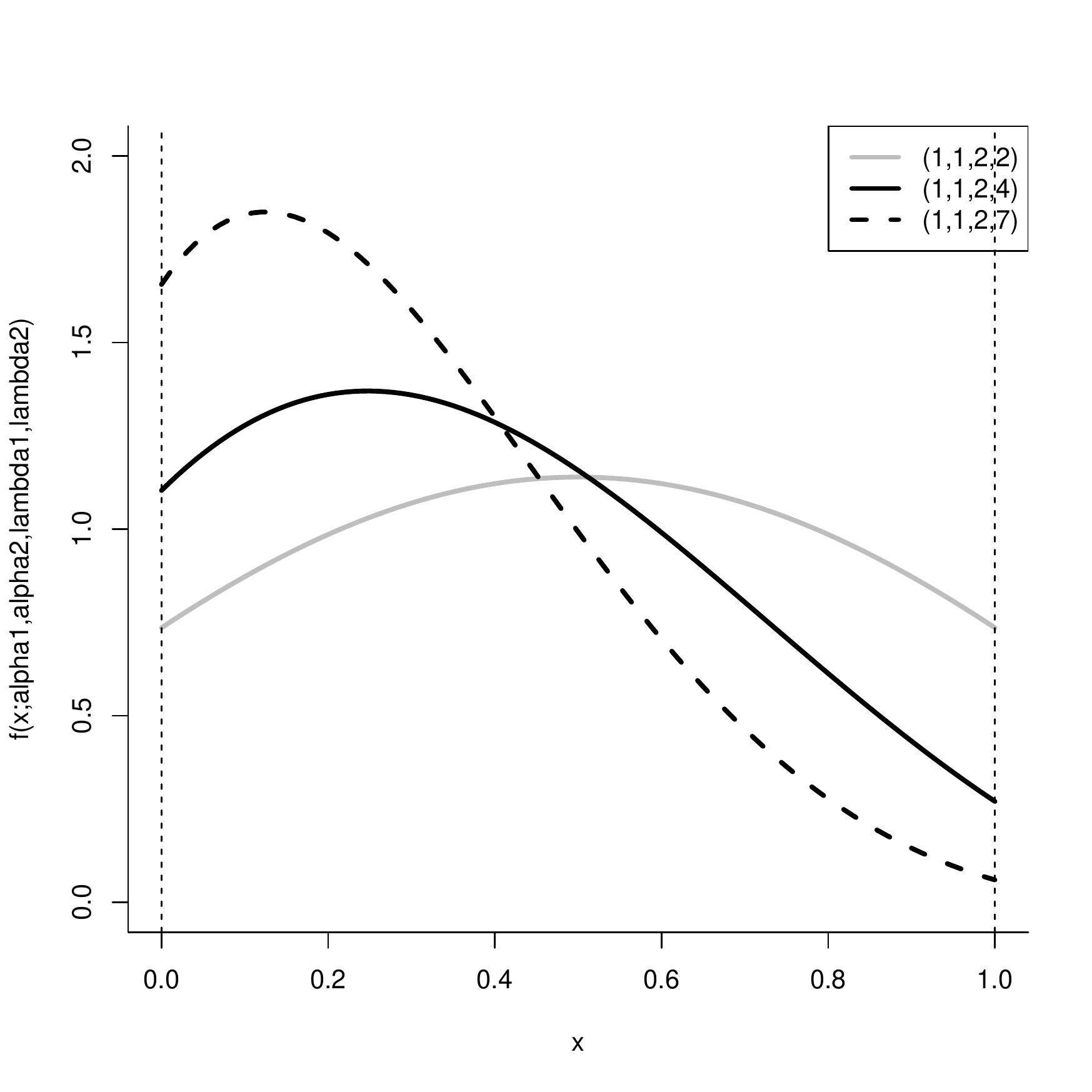}}
  \subfigure
   {\includegraphics[width=6.3cm]{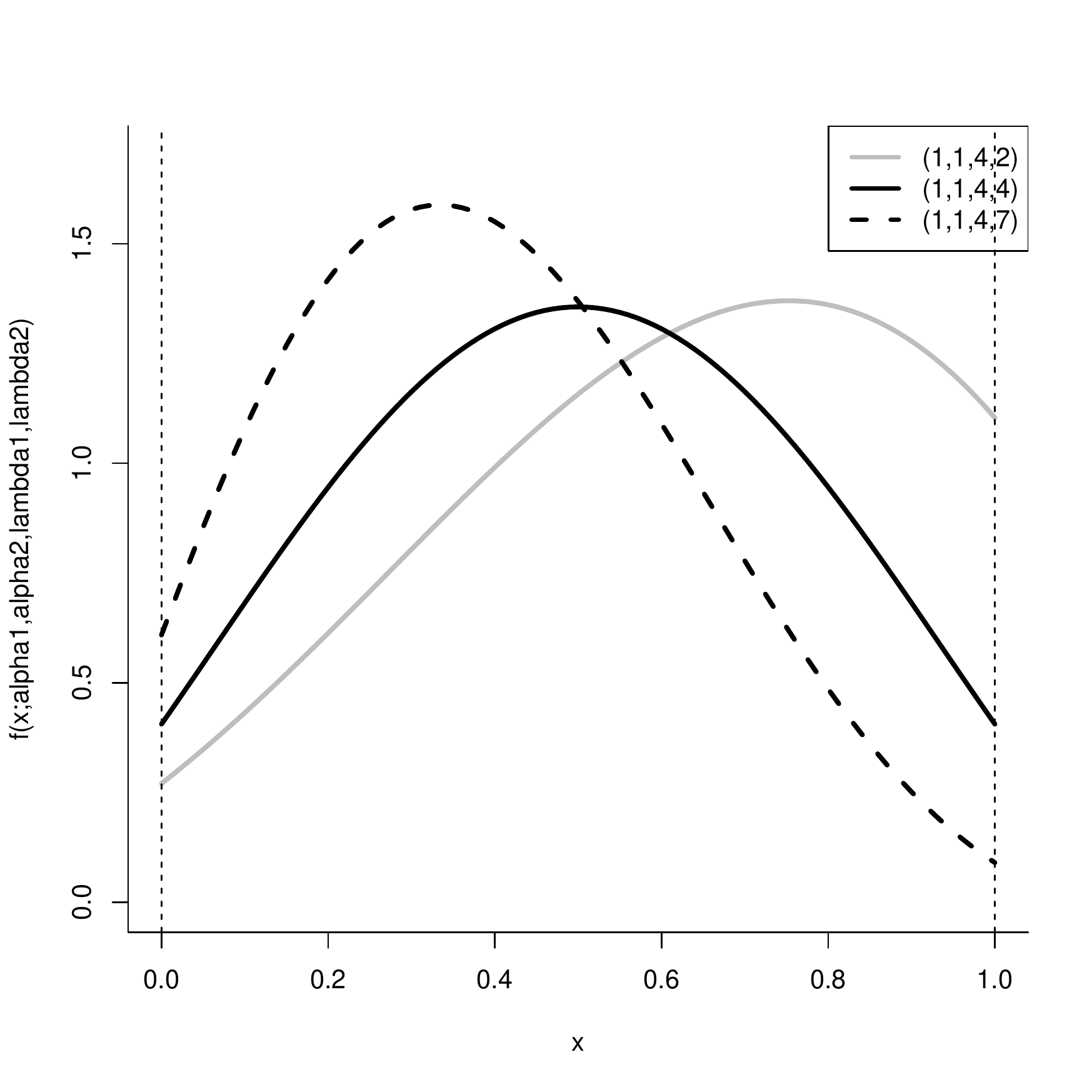}}
 \caption{Plots of the density of $X' \sim \mbox{\normalfont{B}}''\left(\alpha_1,\alpha_2,\lambda_1,\lambda_2\right)$ for $\alpha_1=\alpha_2=1$ and selected values of $\lambda_1,\lambda_2$.}
 \label{fig:DNCB2}
 \end{figure}

\begin{proposition}[Limits at $0$ and $1$ of the $\mbox{B}''$ density when $\alpha_1=\alpha_2=1$]
\label{propo:dens.beta.dnc.lims}
Let $X'$ $\sim$ $\mbox{\normalfont{B}}''\left(\alpha_1,\alpha_2,\lambda_1,\lambda_2\right)$. Then, the limits at $0$ and $1$ of the density $f_{X'}$ of $X'$ when $\alpha_1=\alpha_2=1$ are: 
\begin{equation}
\lim_{x \rightarrow 0^+}f_{X'}\left(x; 1,1,\lambda_1,\lambda_2\right)=e^{-\frac{\lambda_1}{2}}\left(\frac{\lambda_2}{2}+1\right),
\label{eq:dens.beta.dnc.11.0}
\end{equation}
\begin{equation}
\lim_{x \rightarrow 1^-}f_{X'}\left(x; 1,1,\lambda_1,\lambda_2\right)=e^{-\frac{\lambda_2}{2}}\left(\frac{\lambda_1}{2}+1\right).
\label{eq:dens.beta.dnc.11.1}
\end{equation}
\end{proposition}
\begin{proof}
For the proof see~\ref{proof:dens.beta.dnc.lims} in the Appendix.
\end{proof}

Thanks to this essential characteristic, the $\mbox{B}''$ distribution enables to properly model the portions of data having values next to the endpoints of the real interval $\left(0,1\right)$. In this regard, the $\mbox{B}''$ applicative potential will be highlighted in Section~\ref{subsec:dnc.beta.examples} through the analysis of real data.

Finally, by carrying out the same lines as the proof of Proposition~\ref{propo:dens.beta.dnc.lims} or, roughly speaking, by taking $\lambda_2=0$ and renaming $\lambda_1$ with $\lambda$ in Eqs.~(\ref{eq:dens.beta.dnc.11.0}), (\ref{eq:dens.beta.dnc.11.1}), one can obtain the limits of the $\mbox{B}'_1$ density when $\alpha_1=\alpha_2=1$. In view of Property~\ref{prope:relat.ncb1.ncb2}, the limits of the $\mbox{B}'_2$ density can be simply stated by reversing the $\mbox{B}'_1$ ones. Following are their expressions.

\begin{proposition}[Limits at $0$ and $1$ of the $\mbox{B}'_1$ and $\mbox{B}'_2$ densities when $\alpha_1=\alpha_2=1$]
\label{propo:dens.beta.nc12.lims}
Let $X'_1$ $\sim$ $\mbox{\normalfont{B}}'_1\left(\alpha_1,\alpha_2,\lambda\right)$ and $X'_2$ $\sim$ $\mbox{\normalfont{B}}'_2\left(\alpha_1,\alpha_2,\lambda\right)$. Then, the limits at $0$ and $1$ of the density $f_{X'_1}$ of $X'_1$ when $\alpha_1=\alpha_2=1$ are: 
\begin{equation}
\lim_{x_1 \rightarrow 0^+}f_{X'_1}\left(x_1; 1,1,\lambda\right)=e^{-\frac{\lambda}{2}}, \quad
\lim_{x_1 \rightarrow 1^-}f_{X'_1}\left(x_1; 1,1,\lambda\right)=\frac{\lambda}{2}+1,
\label{eq:dens.beta.nc1.11.01}
\end{equation}
while the limits at $0$ and $1$ of the density $f_{X'_2}$ of $X'_2$ when $\alpha_1=\alpha_2=1$ are: 
\begin{equation}
\lim_{x_2 \rightarrow 0^+}f_{X'_2}\left(x_2; 1,1,\lambda\right)=\frac{\lambda}{2}+1, \quad
\lim_{x_2 \rightarrow 1^-}f_{X'_2}\left(x_2; 1,1,\lambda\right)=e^{-\frac{\lambda}{2}}.
\label{eq:dens.beta.nc2.11.01}
\end{equation}
\end{proposition}

\subsection{Patnaik's approximation}
\label{subsec:dnc.beta.dens.pat.approx}

A simple and reliable approximation for the doubly non-central beta distribution can be easily derived by applying the Patnaik's approximation for the non-central chi-squared one \cite{Pat49}.

Indeed, hereafter we prove that the $\mbox{B}''$ model can be approximated by the three-parameter generalization of the beta one introduced by Libby and Novick \cite{LibNov82}.

\begin{proposition}[Patnaik's approximation for B$''$]
\label{propo:dnc.beta.pat.approx}
Let $X'$ have a $\mbox{\normalfont{B}}''$ distribution with shape parameters $\alpha_r$ and non-centrality parameters $\lambda_r$, $r=1,2$. Furthermore, let $Y_r$ be independent $\chi^2_{\nu_r}$ random variables, with:
\begin{equation}
\nu_r=\frac{\left(2\alpha_r+\lambda_r\right)^2}{2\left(\alpha_r+\lambda_r\right)}.
\label{eq:dnc.beta.dens.pat.approx.par1}
\end{equation}
By taking:
\begin{equation}
\beta_r=\frac{\nu_r}{2}, \quad \rho_r=\frac{2\left(\alpha_r+\lambda_r\right)}{2\alpha_r+\lambda_r}, \quad r=1,2, \qquad \quad \gamma=\frac{\rho_2}{\rho_1},
\label{eq:dnc.beta.dens.pat.approx.par2}
\end{equation}
one can approximate $X' \stackrel{d}{\approx} X'_P$, where $X'_P=\frac{\rho_1 \, Y_1}{\rho_1 \, Y_1+\rho_2 \, Y_2} \sim \mbox{\normalfont{G3B}}\left(\beta_1,\beta_2,\gamma\right)$
and:
\begin{equation}
\mbox{\normalfont{G3B}}\left(x';\beta_1,\beta_2,\gamma\right)=\mbox{\normalfont{Beta}}\left(x';\beta_1,\beta_2\right) \,  \frac{\gamma^{\beta_1}}{\left[1-\left(1-\gamma\right)x'\right]^{\beta_1+\beta_2}}, \quad 0<x'<1
\label{eq:g3b.dens}
\end{equation}
is the probability density function of the Libby and Novick's generalized beta distribution \cite{LibNov82}.
\end{proposition}
\begin{proof}
For the proof see~\ref{proof:dnc.beta.pat.approx} in the Appendix.
\end{proof}

Observe that as $\lambda_r$ tends to $0^+$, $r=1,2$, $\nu_r$ tends to $2 \alpha_r$ and $\rho_r$ tends to 1; therefore, the distributions of both $X'$ and $X'_P$ tend to the $\mbox{Beta}\left(\alpha_1,\alpha_2\right)$ one.

A graphic comparison between the $\mbox{\normalfont{B}}''\left(\alpha_1,\alpha_2,\lambda_1,\lambda_2\right)$ density and its approximation herein derived is shown in Figures~\ref{fig:PAT1}, \ref{fig:PAT2}, \ref{fig:PAT3} for selected values of the shape and the non-centrality parameters. Note that in all the cases depicted the plots of the two densities are very similar, except for more or less slight differences on the tails. In this regard, the approximation results particularly unsatisfactory on the tails for $\alpha_1=\alpha_2=1$, due to its inability to replicate the behavior of the $\mbox{\normalfont{B}}''$ density at the unit interval endpoints (see the right-hand panel of Figure~\ref{fig:PAT2}).

\begin{figure}[ht]
 \centering
 \subfigure
   {\includegraphics[width=6.3cm]{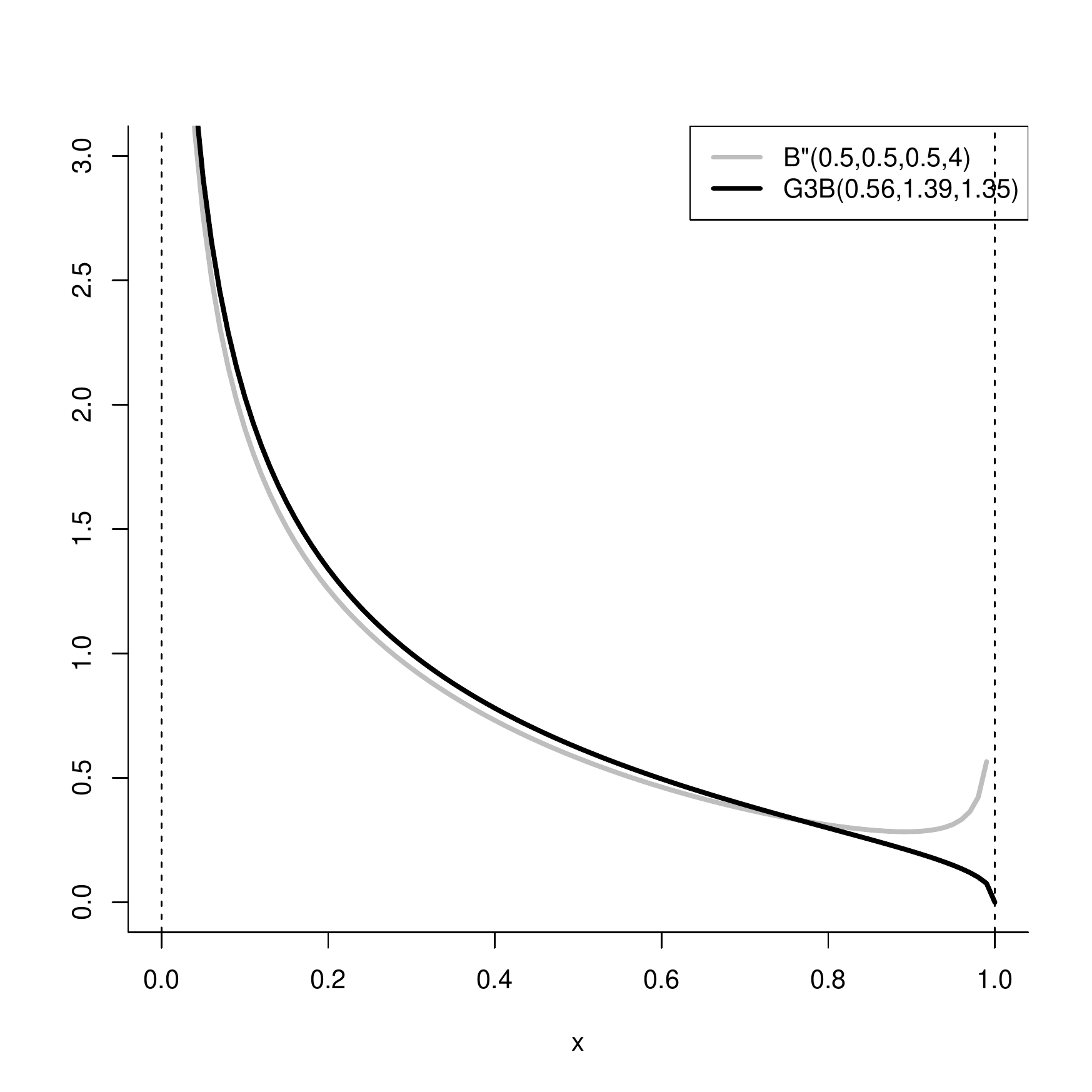}}
 \hspace{5mm}
 \subfigure
   {\includegraphics[width=6.3cm]{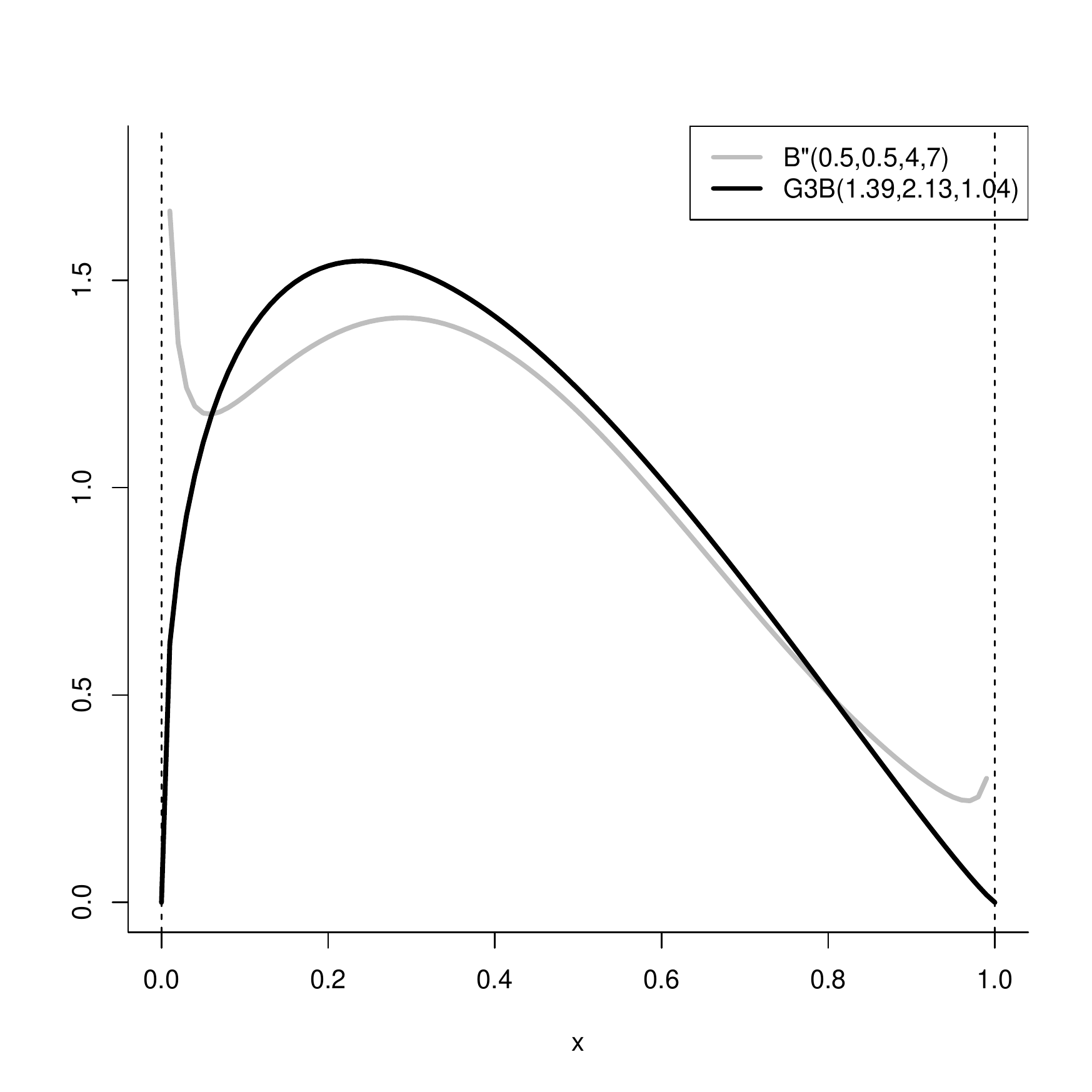}}
  \caption{Plots of the densities of $X' \sim \mbox{\normalfont{B}}''\left(\alpha_1,\alpha_2,\lambda_1,\lambda_2\right)$ and $X'_P \sim \mbox{\normalfont{G3B}}\left(\beta_1,\beta_2,\gamma\right)$ for selected values of $\alpha_1,\alpha_2,\lambda_1,\lambda_2$, with $\beta_1,\beta_2,\gamma$ defined as in Eqs.~(\ref{eq:dnc.beta.dens.pat.approx.par1}),~(\ref{eq:dnc.beta.dens.pat.approx.par2}).}
 \label{fig:PAT1}
\end{figure}

\begin{figure}[ht]
 \centering
 \subfigure
   {\includegraphics[width=6.3cm]{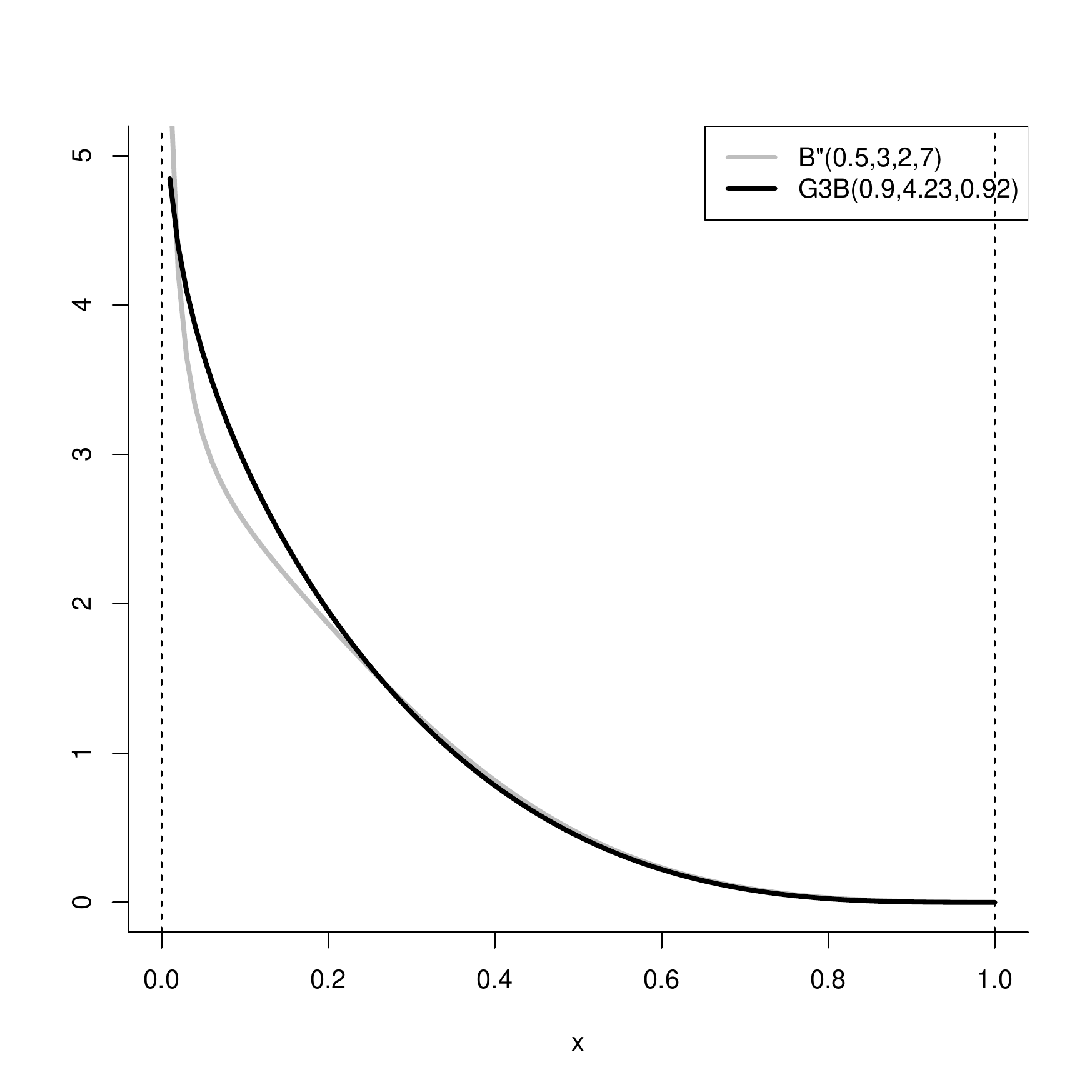}}
 \hspace{5mm}
 \subfigure
   {\includegraphics[width=6.3cm]{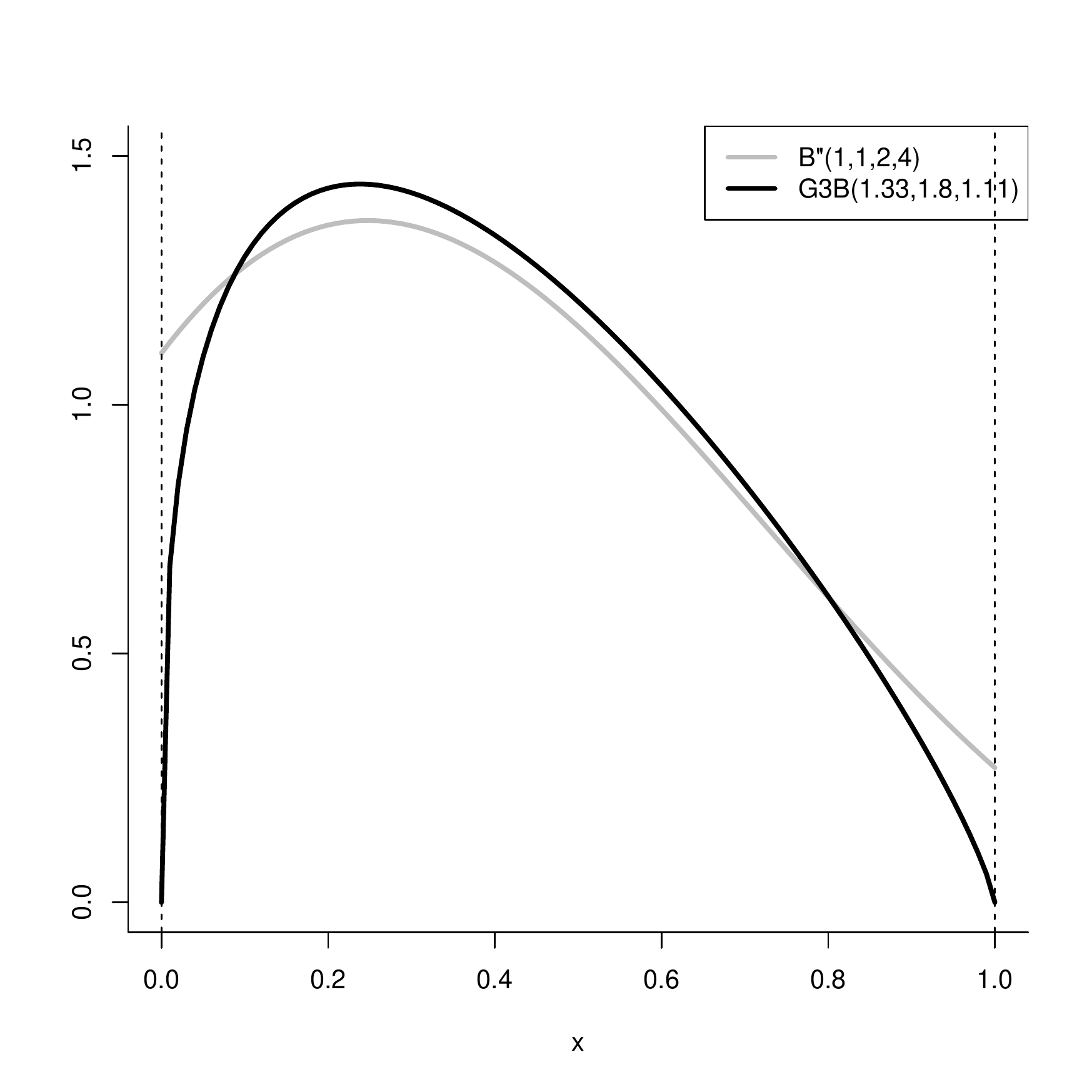}}
 \caption{Plots of the densities of $X' \sim \mbox{\normalfont{B}}''\left(\alpha_1,\alpha_2,\lambda_1,\lambda_2\right)$ and $X'_P \sim \mbox{\normalfont{G3B}}\left(\beta_1,\beta_2,\gamma\right)$ for selected values of $\alpha_1,\alpha_2,\lambda_1,\lambda_2$, with $\beta_1,\beta_2,\gamma$ defined as in Eqs.~(\ref{eq:dnc.beta.dens.pat.approx.par1}),~(\ref{eq:dnc.beta.dens.pat.approx.par2}).}
 \label{fig:PAT2}
\end{figure}  

\begin{figure}[ht]
 \centering 
  \subfigure
   {\includegraphics[width=6.3cm]{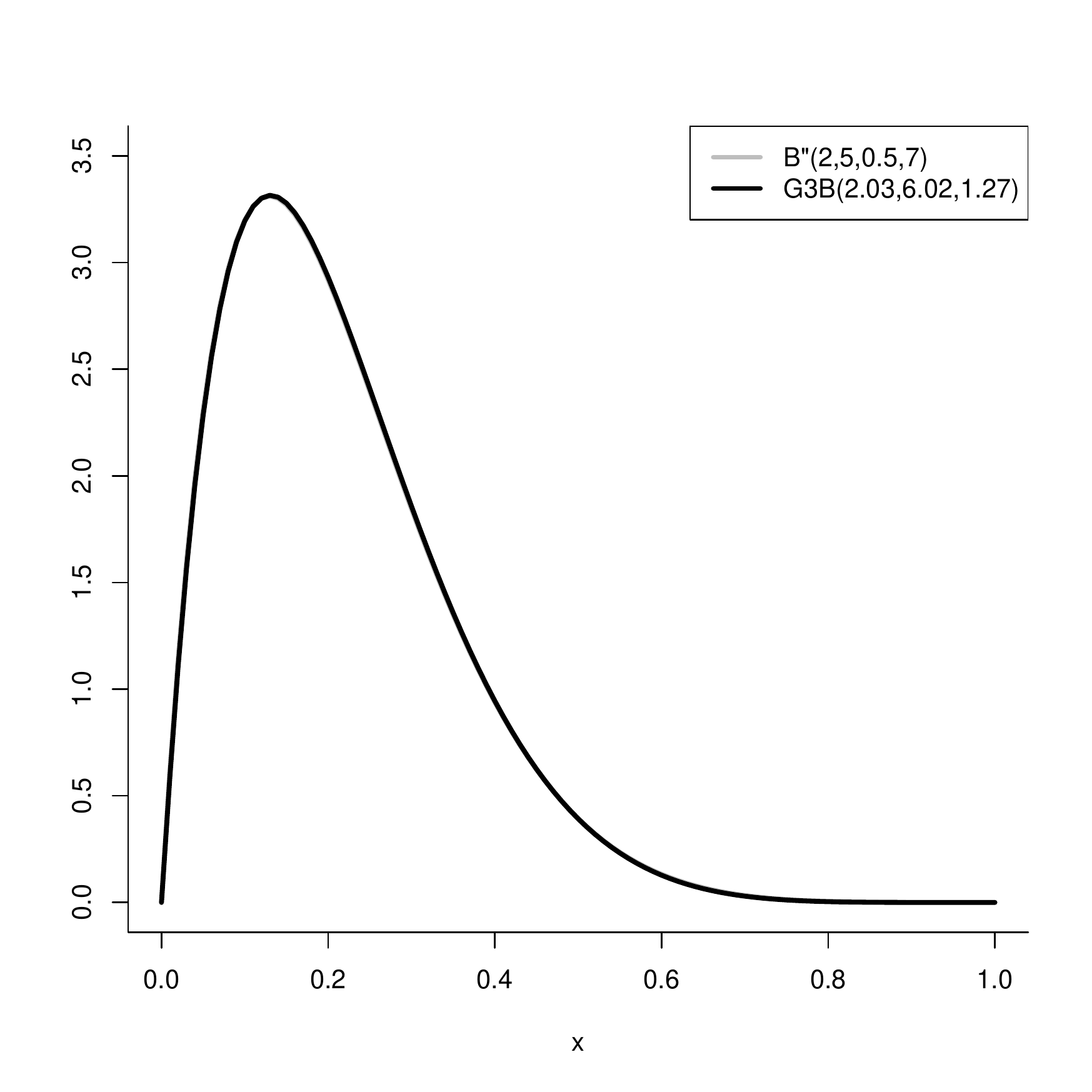}}
 \hspace{5mm}
 \subfigure
   {\includegraphics[width=6.3cm]{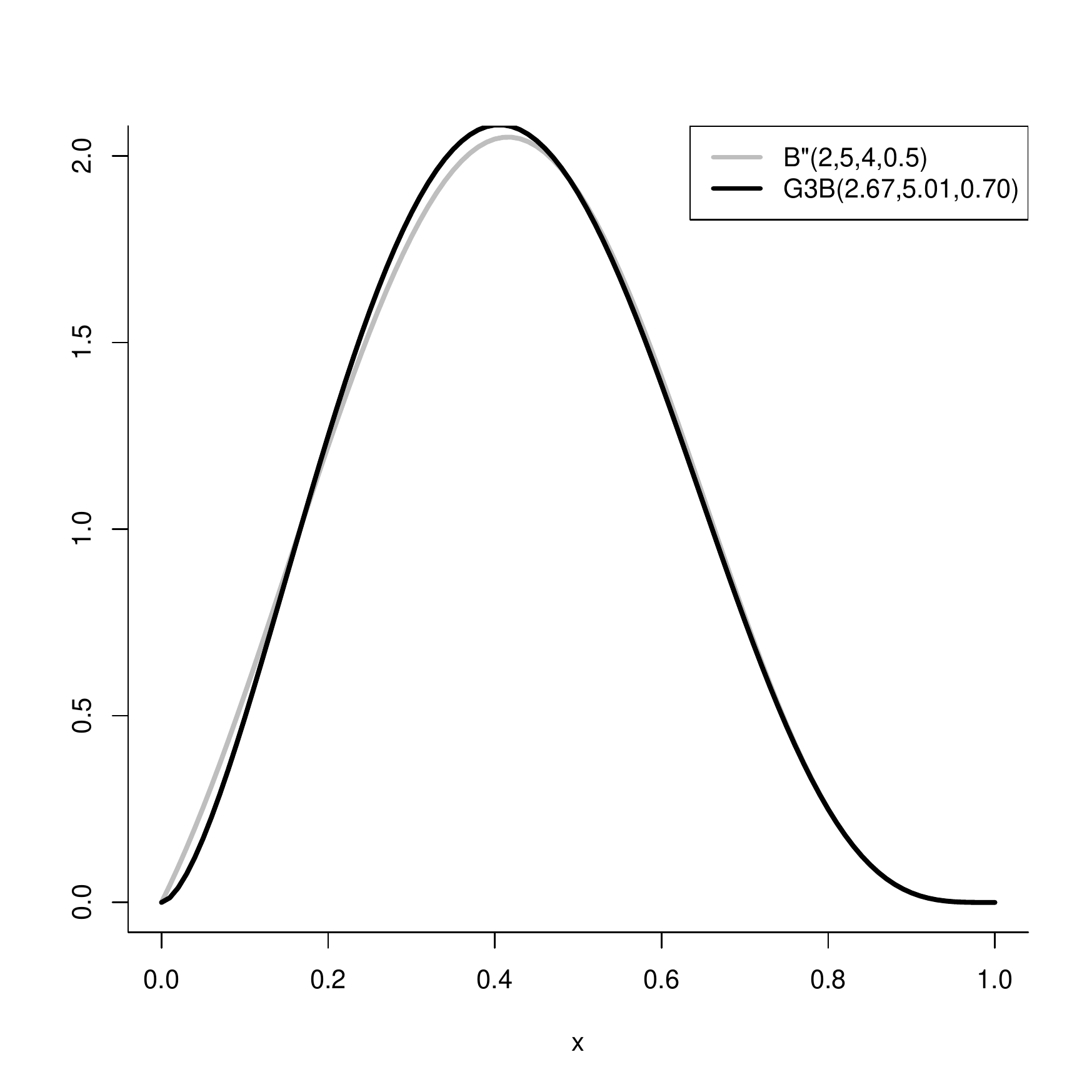}}
  \caption{Plots of the densities of $X' \sim \mbox{\normalfont{B}}''\left(\alpha_1,\alpha_2,\lambda_1,\lambda_2\right)$ and $X'_P \sim \mbox{\normalfont{G3B}}\left(\beta_1,\beta_2,\gamma\right)$ for selected values of $\alpha_1,\alpha_2,\lambda_1,\lambda_2$, with $\beta_1,\beta_2,\gamma$ defined as in Eqs.~(\ref{eq:dnc.beta.dens.pat.approx.par1}),~(\ref{eq:dnc.beta.dens.pat.approx.par2}).}
 \label{fig:PAT3}
\end{figure}

A three-parameter generalized beta random variable, thanks to its relationship with the beta, has distribution function that takes a really simple form. Indeed, from Eq.~(\ref{eq:appdncbetapat4}) in the proof of Proposition~\ref{propo:dnc.beta.pat.approx}, it's immediate to see that if $X'_P \sim \mbox{G3B}\left(\beta_1,\beta_2,\gamma\right)$ then $X=\frac{\gamma X'_P}{\gamma X'_P+1-X'_P} \sim \mbox{Beta}\left(\beta_1,\beta_2\right)$. By exploiting the latter and by denoting the distribution functions of $X'_P$ and $X$ with $F_{X'_P}$ and $F_{X}$ respectively, for every $x' \in (0,1)$ we have accordingly:
\begin{eqnarray}
F_{X'_P}\left(x'\right) & = & \Pr\left(X'_P \le x'\right)=\Pr\left(\frac{X}{X+\gamma \, \left(1-X\right)} \le x'\right)= \nonumber \\
& = & \Pr\left(X \le \frac{\gamma x'}{\gamma x'+1-x'}\right)=F_{X}\left(\frac{\gamma x'}{\gamma x'+1-x'}\right)=\frac{B\left(\frac{\gamma  x'}{\gamma x'+1-x'};\beta_1,\beta_2\right)}{B\left(\beta_1,\beta_2\right)}. \nonumber \\
\label{eq:distr.beta.dnc.approx}
\end{eqnarray}

In view of the foregoing arguments, the latter can be used to approximate the $\mbox{B}''$ distribution function. In this regard, a graphic analysis was performed in order to investigate the goodness of approximation of Eq.~(\ref{eq:distr.beta.dnc.approx}). The results are displayed in Figures~\ref{fig:PATD1}, \ref{fig:PATD2}, \ref{fig:PATD3}, that show a deep reliability of the G3B distribution function as an approximation of the $\mbox{B}''$ one. \texttt{R} codes for the density and the distribution function of the Libby and Novick's generalized beta model are proposed in \ref{funct:dlng3beta}, \ref{funct:plng3beta}.

\begin{figure}[ht]
 \centering
 \subfigure
   {\includegraphics[width=6.3cm]{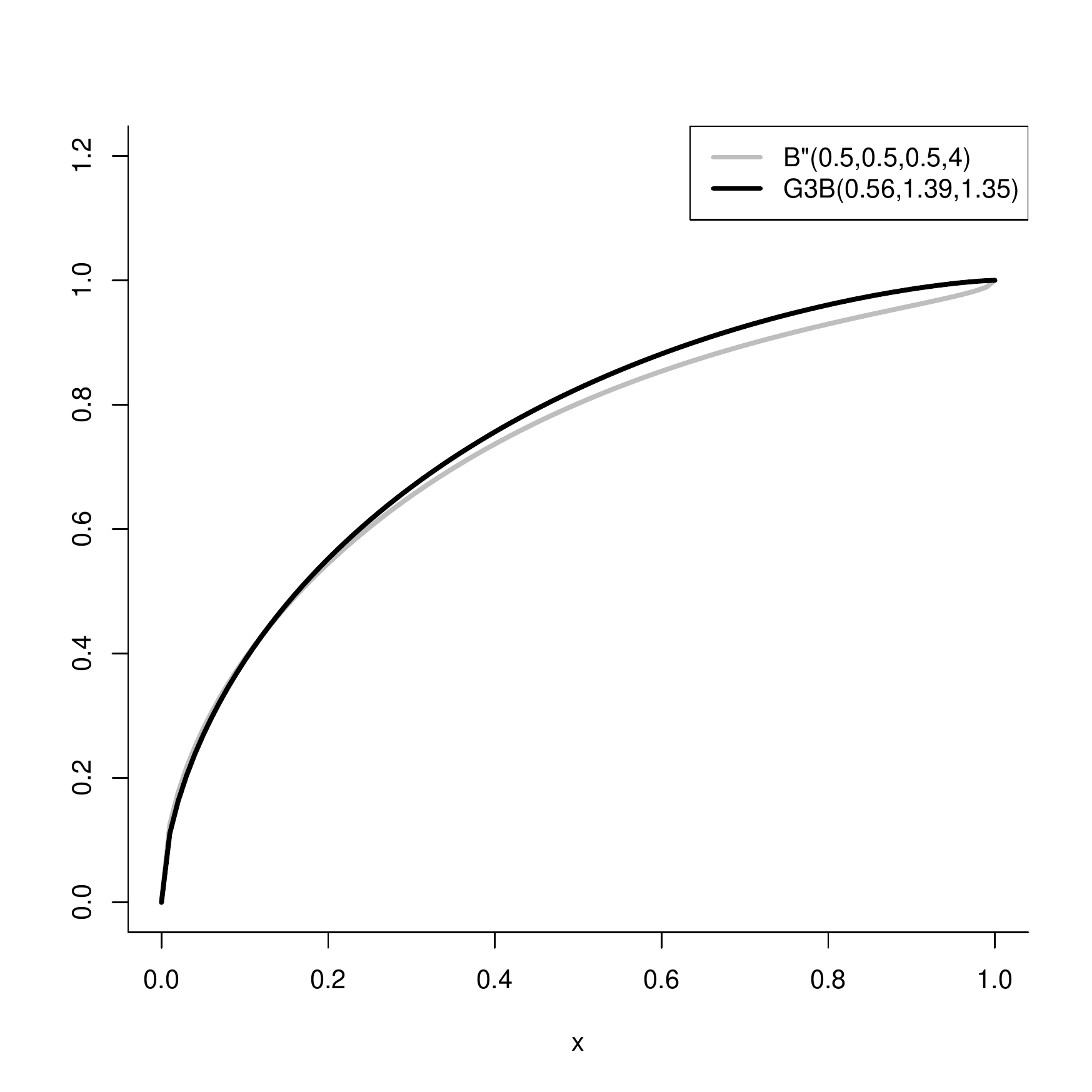}}
 \hspace{5mm}
 \subfigure
   {\includegraphics[width=6.3cm]{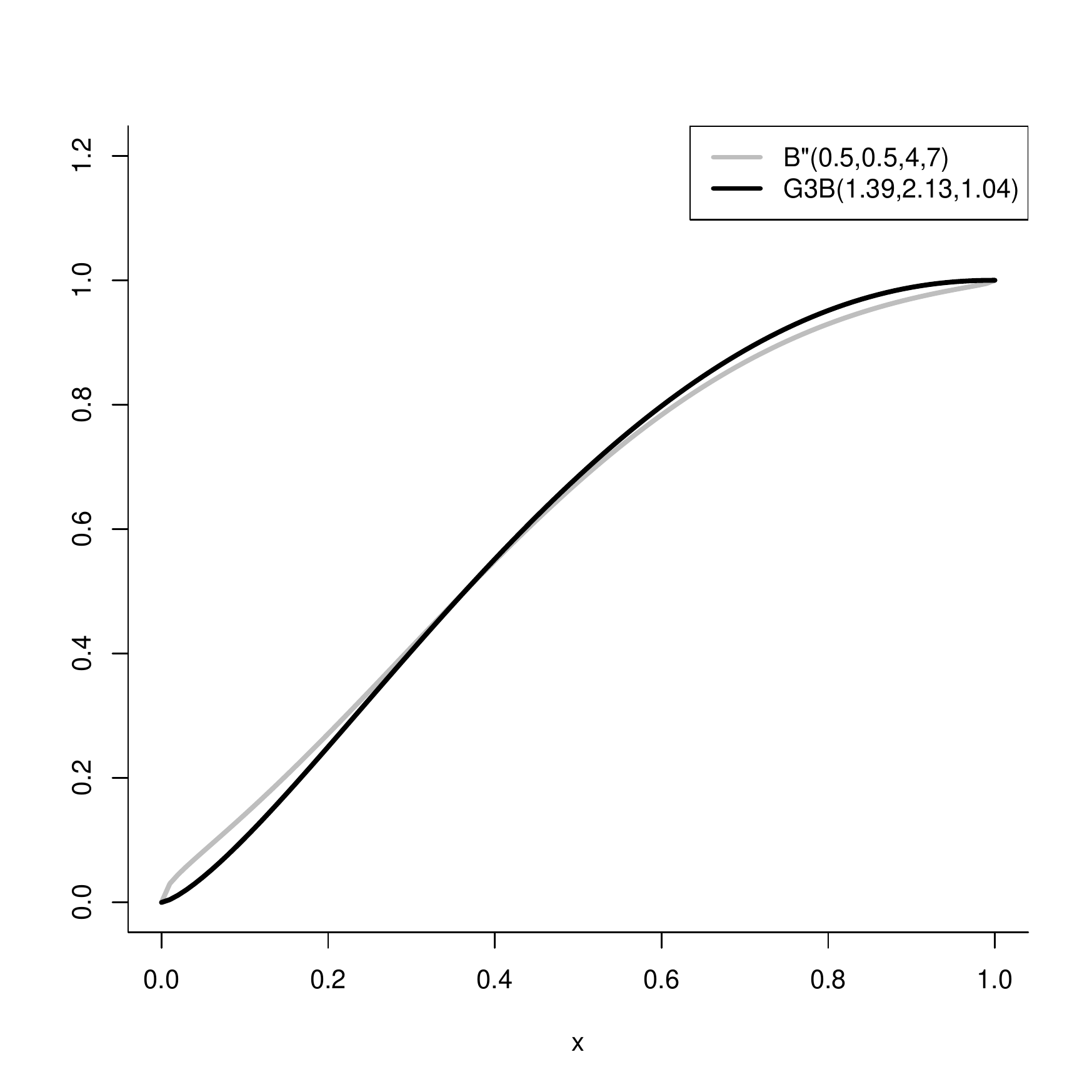}}
  \caption{Plots of the distribution functions of $X' \sim \mbox{\normalfont{B}}''\left(\alpha_1,\alpha_2,\lambda_1,\lambda_2\right)$ and $X'_P \sim \mbox{\normalfont{G3B}}\left(\beta_1,\beta_2,\gamma\right)$ for selected values of $\alpha_1,\alpha_2,\lambda_1,\lambda_2$, with $\beta_1,\beta_2,\gamma$ defined as in Eqs.~(\ref{eq:dnc.beta.dens.pat.approx.par1}),~(\ref{eq:dnc.beta.dens.pat.approx.par2}).}
 \label{fig:PATD1}
\end{figure}

\subsection{Moments}
\label{subsec:dnc.beta.mom}

By analogy with the form of the $\mbox{B}''$ density in Eq.~(\ref{eq:dens.beta.dnc}) and the $\mbox{B}''$ distribution function in Eq.~(\ref{eq:distr.beta.dnc}), the $r$-th moment about zero of the doubly non-central beta distribution can be stated as the double series of the $r$-th moments about zero of the Beta$(\alpha_1+j,\alpha_2+k)$ distributions, $j,k \in \mathbb{N} \cup \{0\}$, weighted by the joint probabilities of $\left(M_1,M_2\right)$, where $M_i$, $i=1,2$, are independent with $\mbox{Poisson}\left(\lambda_i/2\right)$ distributions. As far as we know, the latter is the only analytical form available in the literature for the moments of the $\mbox{B}''$ distribution.

\begin{figure}[ht]
 \centering
 \subfigure
   {\includegraphics[width=6.3cm]{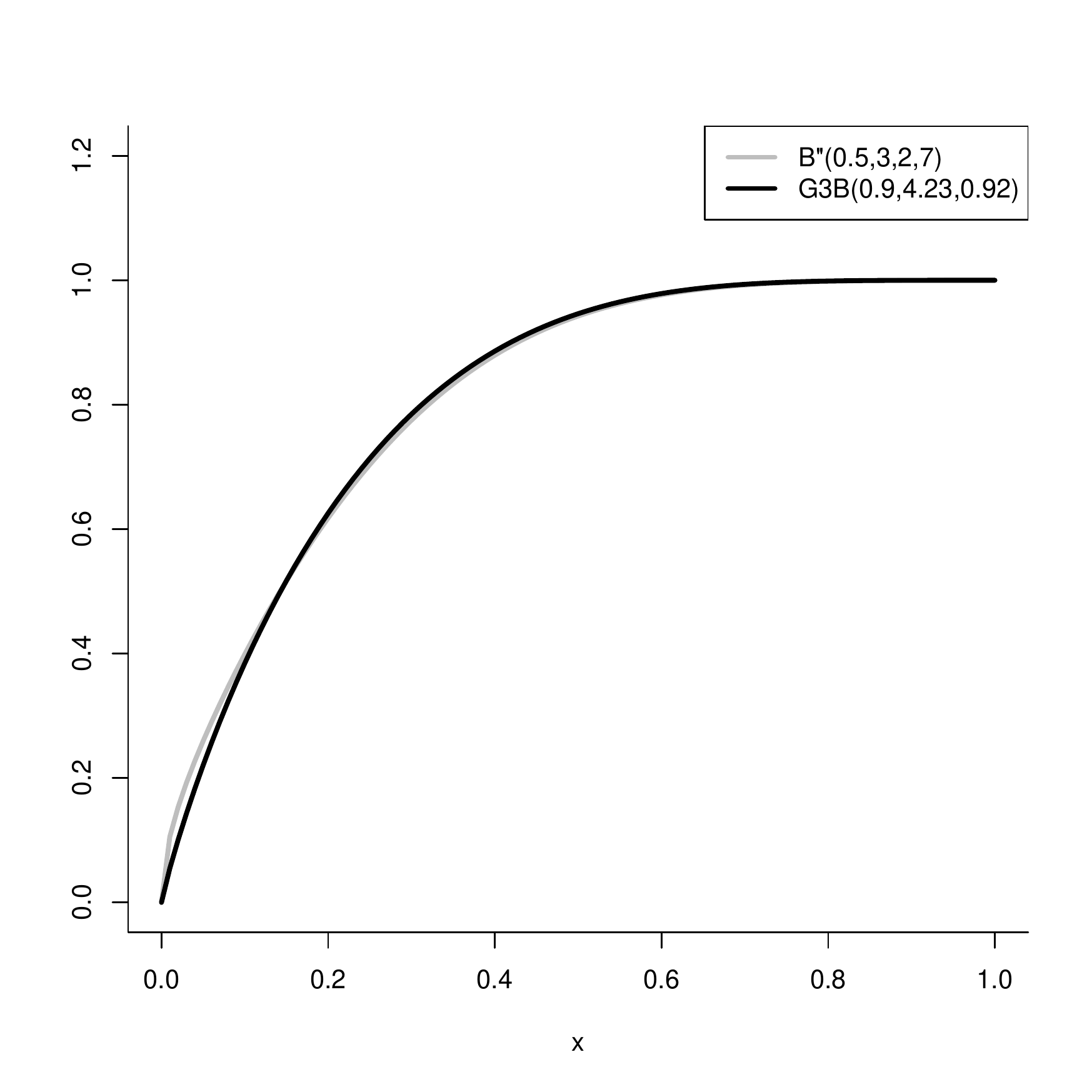}}
 \hspace{5mm}
 \subfigure
   {\includegraphics[width=6.3cm]{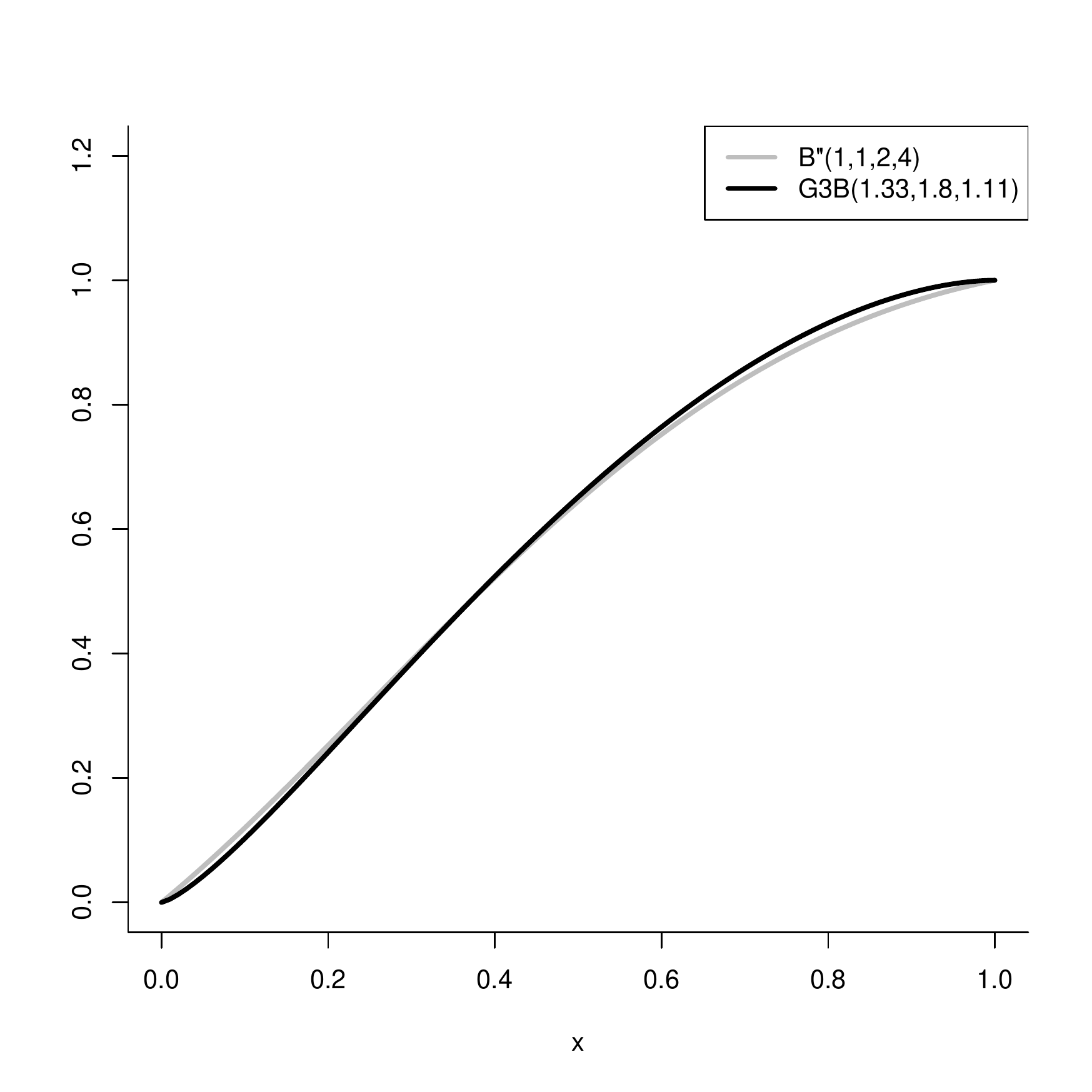}}
 \caption{Plots of the distribution functions of $X' \sim \mbox{\normalfont{B}}''\left(\alpha_1,\alpha_2,\lambda_1,\lambda_2\right)$ and $X'_P \sim \mbox{\normalfont{G3B}}\left(\beta_1,\beta_2,\gamma\right)$ for selected values of $\alpha_1,\alpha_2,\lambda_1,\lambda_2$, with $\beta_1,\beta_2,\gamma$ defined as in Eqs.~(\ref{eq:dnc.beta.dens.pat.approx.par1}),~(\ref{eq:dnc.beta.dens.pat.approx.par2}).}
 \label{fig:PATD2}
\end{figure}  

\begin{figure}[ht]
 \centering 
  \subfigure
   {\includegraphics[width=6.3cm]{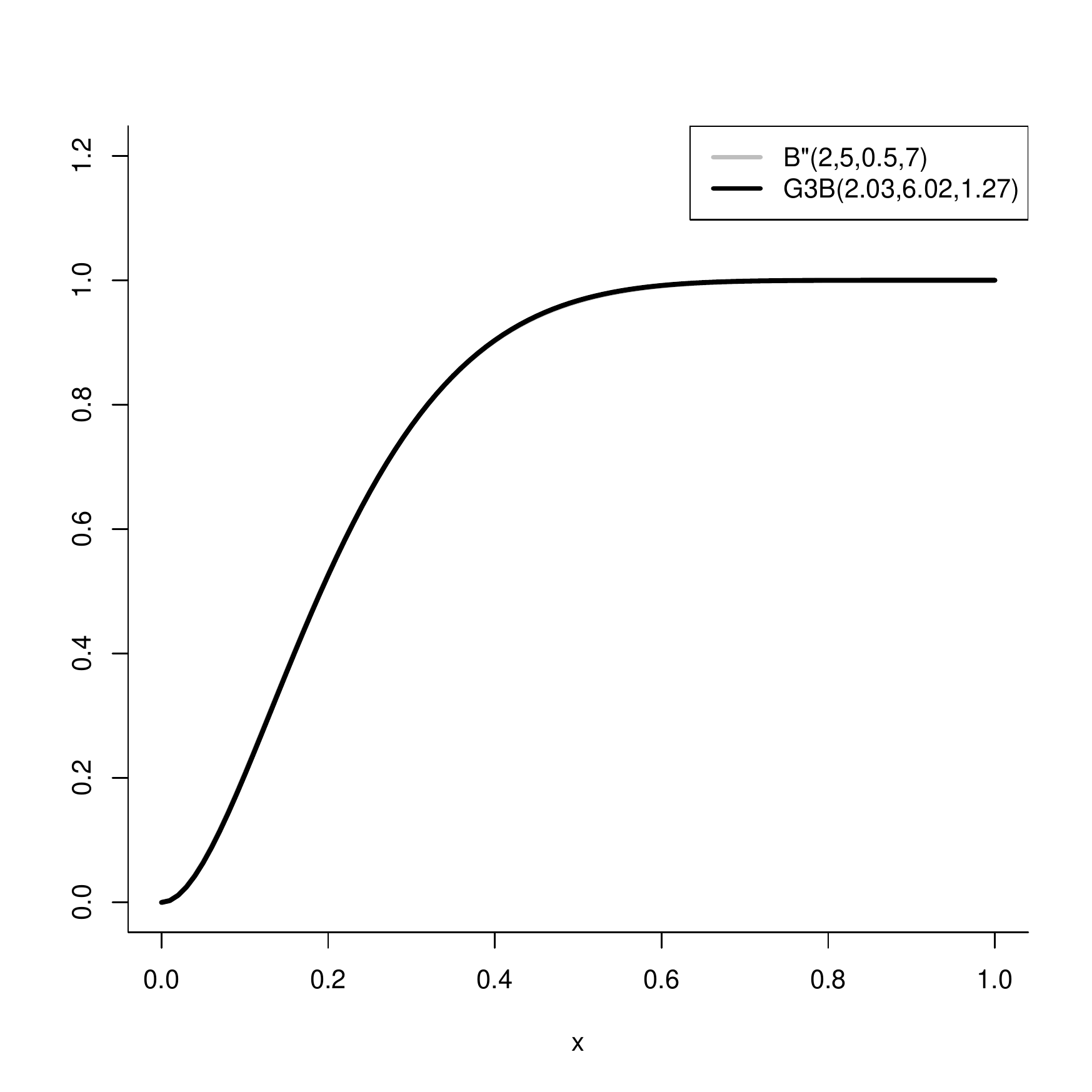}}
 \hspace{5mm}
 \subfigure
   {\includegraphics[width=6.3cm]{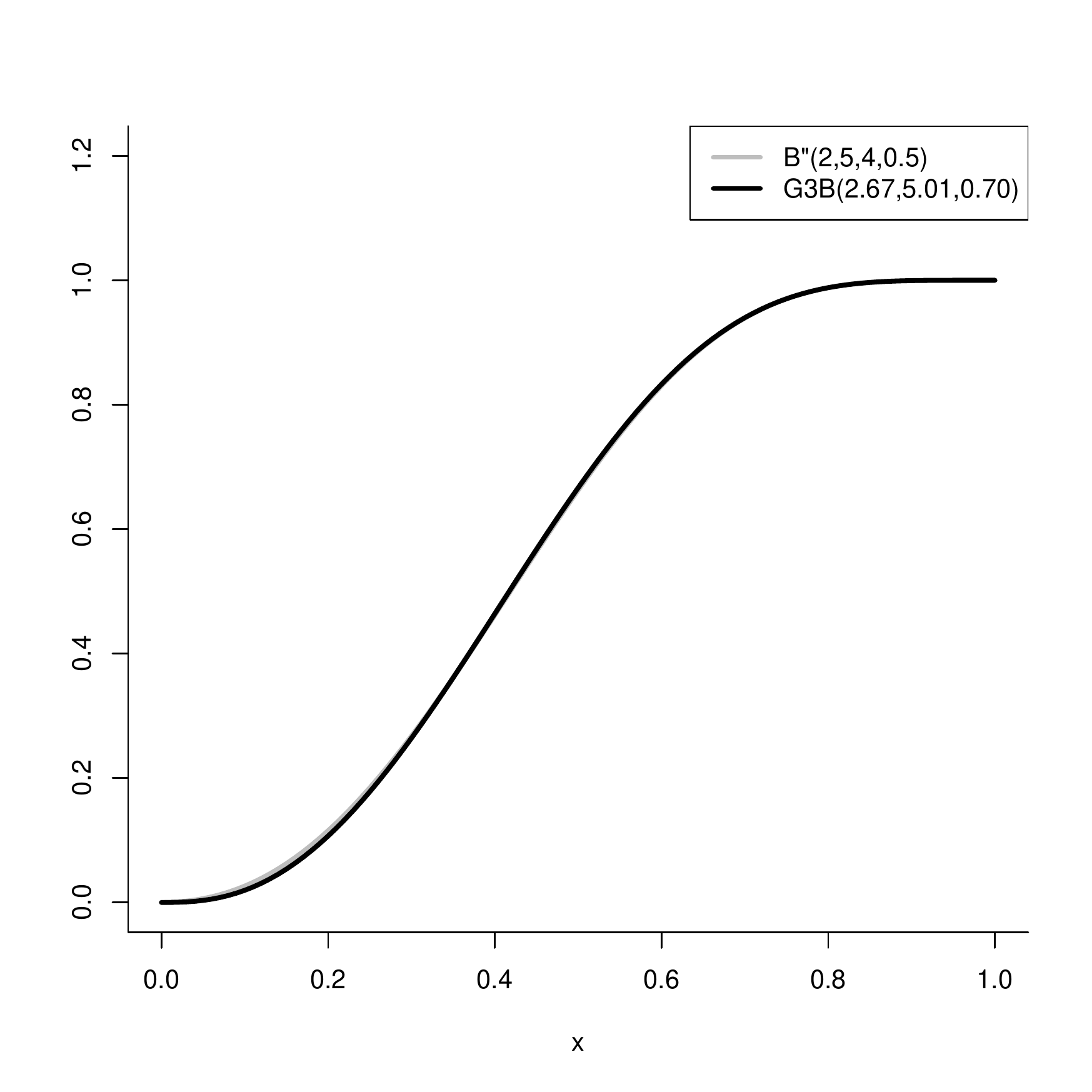}}
  \caption{Plots of the distribution functions of $X' \sim \mbox{\normalfont{B}}''\left(\alpha_1,\alpha_2,\lambda_1,\lambda_2\right)$ and $X'_P \sim \mbox{\normalfont{G3B}}\left(\beta_1,\beta_2,\gamma\right)$ for selected values of $\alpha_1,\alpha_2,\lambda_1,\lambda_2$, with $\beta_1,\beta_2,\gamma$ defined as in Eqs.~(\ref{eq:dnc.beta.dens.pat.approx.par1}),~(\ref{eq:dnc.beta.dens.pat.approx.par2}).}
 \label{fig:PATD3}
\end{figure}

That said, in the present Section a new general formula for the moments of such distribution is provided. This formula allows the computation of moments to be reduced from a double series to a single one. According to the latter, in fact, the $r$-th moment can be evaluated in terms of a perturbation of the corresponding moment of the beta distribution through a weighted sum of Kummer's confluent hypergeometric functions. More specifically, the present result extends and completes Proposition 7 in \cite{OngOrs15} and concludes that the $r$-th moment of the $\mbox{B}''$ distribution can be expressed as follows.

\begin{proposition}[Moments about zero of $\mbox{\normalfont{B}}''$ distribution]
\label{propo:momr.beta.dnc.rapp}
Let $X' \sim \mbox{\normalfont{B}}''\left(\alpha_1,\alpha_2,\lambda_1,\lambda_2\right)$ and $\alpha^+=\alpha_1+\alpha_2$, $\lambda^+=\lambda_1+\lambda_2$. Let $M_j$, $j=1,2$, be independent Poisson random variables with means $\lambda_j/2$ and $M^+=M_1+M_2$. Then, for every $r \in \mathbb{N}$, the $r$-th moment about zero of $X'$ admits the following expression:
\begin{equation}
\mathbb{E}\left[\left(X'\right)^r\right]=\frac{\left(\alpha_1\right)_r}{\left(\alpha^+\right)_r} \, e^{-\frac{\lambda^+}{2}}\sum_{i=0}^{r} \frac{{r \choose i}\left(\alpha^+\right)_i\left(\frac{\lambda_1}{2}\right)^i}{\left(\alpha_1\right)_i \left(\alpha^++r\right)_i} \, _1F_1\left(\alpha^++i;\alpha^++r+i;\frac{\lambda^+}{2}\right).
\label{eq:momr.beta.dnc}
\end{equation}
\end{proposition}
\begin{proof}
For the proof see~\ref{proof:momr.beta.dnc.rapp} in the Appendix.
\end{proof}

The first two moments of the $\mbox{B}''$ distribution can thus be computed as special cases of Eq.~(\ref{eq:momr.beta.dnc}) by taking $r=1$ and $r=2$ as follows:
\begin{equation}
\mathbb{E}\left(X'\right)=\frac{\alpha_1}{\alpha^+}\, e^{-\frac{\lambda^+}{2}} \left[_1F_1\left(\alpha^+;\alpha^++1;\frac{\lambda^+}{2}\right)+\frac{\alpha^+ \, \frac{\lambda_1}{2}}{\alpha_1 \left(\alpha^++1\right)} \, _1F_1\left(\alpha^++1;\alpha^++2;\frac{\lambda^+}{2}\right)\right],
\label{eq:mom1.beta.dnc}
\end{equation}
\begin{eqnarray}
\lefteqn{\mathbb{E}\left[\left(X'\right)^2\right]=\frac{\left(\alpha_1\right)_2}{\left(\alpha^+\right)_2} \, e^{-\frac{\lambda^+}{2}} \left[_1F_1\left(\alpha^+;\alpha^++2;\frac{\lambda^+}{2}\right)+\frac{\alpha^+ \, \lambda_1}{\alpha_1 \left(\alpha^++2\right)} \cdot \right.}\nonumber\\
& \cdot &  \left. _1F_1\left(\alpha^++1;\alpha^++3;\frac{\lambda^+}{2}\right)+\frac{\left(\alpha^+\right)_2 \, \left(\frac{\lambda_1}{2}\right)^2}{\left(\alpha_1\right)_2 \left(\alpha^++2\right)_2} \,  _1F_1\left(\alpha^++2;\alpha^++4;\frac{\lambda^+}{2}\right)\right].
\label{eq:mom2.beta.dnc}
\end{eqnarray}
An implementation in \texttt{R} language of the moments formula in Eq.~(\ref{eq:momr.beta.dnc}) is proposed in \ref{funct:mdncbeta}.

Now we should like to make a few comments on the moments of the type 1 and the type 2 non-central beta distributions. More specifically, by making use of the definition of the $r$-th moment about zero of a random variable, the following formula holds for the moments of $X'_1$ $\sim$ $\mbox{\normalfont{B}}'_1\left(\alpha_1,\alpha_2,\lambda\right)$:
\begin{equation}
\mathbb{E}\left[\left(X'_1\right)^r\right]=\frac{\left(\alpha_1\right)_r}{\left(\alpha^+\right)_r} \, e^{-\frac{\lambda}{2}} \, _2F_2\left(\alpha_1+r,\alpha^+;\alpha_1, \alpha^++r;\frac{\lambda}{2}\right),
\label{eq:momr.beta.nc1.def}
\end{equation}
where $_2F_2\left(a_1,a_2;b_1,b_2;x\right)=\sum_{k=0}^{+\infty}\frac{(a_1)_k \, (a_2)_k}{(b_1)_k \, (b_2)_k}	\frac{x^k}{k!}$ is the generalized hypergeometric function $_pF_q$ with $p=2$ and $q=2$ coefficients respectively at numerator and denominator \cite{SriKar85}.

That said, a new general formula for the moments about zero of the $\mbox{\normalfont{B}}'_1$ distribution can be derived regardless of Eq.~(\ref{eq:momr.beta.nc1.def}) in light of Eq.~(\ref{eq:momr.beta.dnc}). Indeed, by taking $\lambda_2=0$ and renaming $\lambda_1$ with $\lambda$ in Eq.~(\ref{eq:momr.beta.dnc}), the following holds true.

\begin{proposition}[Moments about zero of $\mbox{\normalfont{B}}'_1$ distribution]
\label{propo:momr.beta.nc1}
Let $X'_1$ $\sim$ $\mbox{\normalfont{B}}'_1\left(\alpha_1,\alpha_2,\lambda\right)$ and $\alpha^+=\alpha_1+\alpha_2$. Then, for every $r \in \mathbb{N}$, the $r$-th moment about zero of $X_1'$ admits the following expression:
\begin{equation}
\mathbb{E}\left[\left(X'_1\right)^r\right]=\frac{\left(\alpha_1\right)_r}{\left(\alpha^+\right)_r} \, e^{-\frac{\lambda}{2}}\sum_{i=0}^{r} \frac{{r \choose i}\left(\alpha^+\right)_i\left(\frac{\lambda}{2}\right)^i}{\left(\alpha_1\right)_i \left(\alpha^++r\right)_i} \, _1F_1\left(\alpha^++i;\alpha^++r+i;\frac{\lambda}{2}\right).
\label{eq:momr.beta.nc1}
\end{equation}
\end{proposition}

As a side effect, by comparing Eqs.~(\ref{eq:momr.beta.nc1.def}),~(\ref{eq:momr.beta.nc1}), the following identity between the aforementioned hypergeometric functions holds true:
$$_2F_2\left(\alpha_1+r,\alpha^+;\alpha_1, \alpha^++r;\frac{\lambda}{2}\right)=\sum_{i=0}^{r} \frac{{r \choose i}\left(\alpha^+\right)_i\left(\frac{\lambda}{2}\right)^i}{\left(\alpha_1\right)_i \left(\alpha^++r\right)_i} \, _1F_1\left(\alpha^++i;\alpha^++r+i;\frac{\lambda}{2}\right).$$

Finally, the general formula for the moments about zero of the type 2 non-central beta distribution can be stated by taking $\lambda_1=0$ and renaming $\lambda_2$ with $\lambda$ in Eq.~(\ref{eq:momr.beta.dnc}). The latter can be also derived by making use of the definition of the $r$-th moment about zero of a random variable, obtaining as follows:
\begin{equation}
\mathbb{E}\left[\left(X'_2\right)^r\right]=\frac{\left(\alpha_1\right)_r}{\left(\alpha^+\right)_r} \, e^{-\frac{\lambda}{2}} \, _1F_1\left(\alpha^+;\alpha^++r;\frac{\lambda}{2}\right).
\label{eq:momr.beta.nc2}
\end{equation}

An interesting relationship applies among the means of the three non-central beta distributions recalled herein. More precisely, the mean of the $\mbox{B}''$ distribution with shape parameters $\alpha_1$, $\alpha_2$ and non-centrality parameters $\lambda_1$, $\lambda_2$ can be expressed as a convex linear combination of the means of the $\mbox{B}'_1$ and $\mbox{B}'_2$ distributions with shape parameters $\alpha_1$, $\alpha_2$ and non-centrality parameter $\lambda^+=\lambda_1+\lambda_2$.   

\begin{proposition}[Relationship among the means of $\mbox{\normalfont{B}}''$, $\mbox{\normalfont{B}}'_1$, $\mbox{\normalfont{B}}'_2$ distributions]
\label{propo:relat.means.nc.beta}
Let $X' \;$ $\sim \;$ $\mbox{\normalfont{B}}''$ $\left(\alpha_1,\alpha_2,\lambda_1,\lambda_2\right)$, $X'_1$ $\sim$ $\mbox{\normalfont{B}}'_1\left(\alpha_1,\alpha_2,\lambda^+\right)$ and $X'_2$ $\sim$ $\mbox{\normalfont{B}}'_2\left(\alpha_1,\alpha_2,\lambda^+\right)$, where $\lambda^+=\lambda_1+\lambda_2$. Then:
\begin{equation}
\mathbb{E}\left(X'\right)=\frac{\lambda_1}{\lambda^+} \, \mathbb{E}\left(X'_1\right)+\frac{\lambda_2}{\lambda^+} \, \mathbb{E}\left(X'_2\right).
\label{eq:relat.means.nc.beta}
\end{equation}
\end{proposition}
\begin{proof}
For the proof see~\ref{proof:relat.means.nc.beta} in the Appendix.
\end{proof}

Moreover, by resorting to Proposition~\ref{propo:rappr.clc.beta.dnc}, we can obtain an alternative and interesting expression for the mean of the doubly non-central beta distribution in terms of a convex linear combination of the mean of the beta distribution and a compositional ratio of the non-centrality parameters.
\begin{proposition}[Alternative expression for the mean of $\mbox{\normalfont{B}}''$ distribution]
\label{propo:alter.exp.mean.dnc.beta}
Let $X' \;$ $\sim \;$ $\mbox{\normalfont{B}}''$ $\left(\alpha_1,\alpha_2,\lambda_1,\lambda_2\right)$ and $\alpha^+=\alpha_1+\alpha_2$, $\lambda^+=\lambda_1+\lambda_2$. Then:
\begin{equation}
\mathbb{E}\left(X'\right)=\frac{\alpha_1}{\alpha^+}\left[e^{-\frac{\lambda^+}{2}}  \, _1F_1\left(\alpha^+;\alpha^++1;\frac{\lambda^+}{2}\right)\right]+\frac{\lambda_1}{\lambda^+}\left[1-e^{-\frac{\lambda^+}{2}}  \, _1F_1\left(\alpha^+;\alpha^++1;\frac{\lambda^+}{2}\right)\right].
\label{eq:alter.exp.mean.dnc.beta}
\end{equation}
\end{proposition}
\begin{proof}
For the proof see~\ref{proof:alter.exp.mean.dnc.beta} in the Appendix.
\end{proof}

We conclude the present Section by further investigating the moments of the $\mbox{\normalfont{B}}''$ distribution when $\alpha_1=\alpha_2=1$. In the latter case, the mean and the variance interestingly take on the following simple forms.

\begin{proposition}[Mean and variance of $\mbox{\normalfont{B}}''$ distribution when $\alpha_1=\alpha_2=1$]
\label{propo:exp.mean.var.dnc11.beta}
Let $X' \;$ $\sim \;$ $\mbox{\normalfont{B}}''$ $\left(\alpha_1,\alpha_2,\lambda_1,\lambda_2\right)$ with $\alpha_1=\alpha_2=1$. Then:
\begin{equation}
\mathbb{E}\left(X'\right)=\frac{1}{2}+\frac{\lambda_1-\lambda_2}{2\left(\lambda^+\right)^3}\left[\left(\lambda^+\right)^2-4 \, \lambda^++8-8 \, e^{-\frac{\lambda^+}{2}}\right],
\label{eq:exp.mean.dnc11.beta}
\end{equation}
\begin{eqnarray}
\mbox{\normalfont{Var}}\left(X'\right) & = & \frac{4}{\left(\lambda^+\right)^2}+\frac{4 \, \lambda_1 \, \lambda_2}{\left(\lambda^+\right)^5}\left[\left(\lambda^+-2\right)  \left(\lambda^+-8\right)-2 \, e^{-\frac{\lambda^+}{2}} 	\left(\lambda^++8\right)\right]+\nonumber\\
& + & \frac{8 \left(\lambda_1-\lambda_2\right)^2}{\left(\lambda^+\right)^6}\left[\left(\lambda^+\right)^2 \, e^{-\frac{\lambda^+}{2}} -2\left(1-e^{-\frac{\lambda^+}{2}}\right)\left(\lambda^++1-e^{-\frac{\lambda^+}{2}}\right)\right].
\label{eq:exp.var.dnc11.beta}
\end{eqnarray}
\end{proposition}
\begin{proof}
For the proof see~\ref{proof:exp.mean.var.dnc11.beta} in the Appendix.
\end{proof}

Observe that $\mathbb{E}(X')>\frac{1}{2}$ when $\lambda_1>\lambda_2$; in fact, in view of Eq.~(\ref{eq:exp.mean.dnc11.beta}), $\left(\lambda^+\right)^2-4 \, \lambda^++8>8 \, e^{-\frac{\lambda^+}{2}}$ for every $\lambda^+>0$. 

Note that Eqs.~(\ref{eq:exp.mean.dnc11.beta}),~(\ref{eq:exp.var.dnc11.beta}) become considerably simplified by assuming that the non-centrality parameters are equal. Indeed, by taking $\lambda_1=\lambda_2=\lambda$ we have accordingly:
$$\mathbb{E}\left(X'\right)=\frac{1}{2}, \qquad \mbox{Var}\left(X'\right)=\frac{1}{\lambda^2}+\frac{1}{2 \, \lambda^3}\left[\left(\lambda-1\right)\left(\lambda-4\right)-\left(\lambda+4\right) \, e^{-\lambda}\right].$$

Finally, by carrying out the same lines as the proof of Proposition~\ref{propo:exp.mean.var.dnc11.beta} or, roughly speaking, by taking $\lambda_2=0$ and renaming $\lambda_1$ with $\lambda$ in Eqs.~(\ref{eq:exp.mean.dnc11.beta}), (\ref{eq:exp.var.dnc11.beta}), one can obtain simple expressions for the mean and the variance of the $\mbox{B}'_1$ distribution when $\alpha_1=\alpha_2=1$. Similarly, simple expressions can be derived also for the mean and the variance of the $\mbox{B}'_2$ distribution when $\alpha_1=\alpha_2=1$ by taking $\lambda_1=0$ and renaming $\lambda_2$ with $\lambda$ in Eqs.~(\ref{eq:exp.mean.dnc11.beta}), (\ref{eq:exp.var.dnc11.beta}). Following are their expressions.

\begin{proposition}[Mean and variance of $\mbox{\normalfont{B}}'_1$ and $\mbox{\normalfont{B}}'_2$ distributions when $\alpha_1=\alpha_2=1$]
\label{propo:exp.mean.var.nc12.11.beta}
Let $X'_1 \sim \mbox{\normalfont{B}}'_1\left(\alpha_1,\alpha_2,\lambda\right)$ and $X'_2 \sim \mbox{\normalfont{B}}'_2\left(\alpha_1,\alpha_2,\lambda\right)$ with $\alpha_1=\alpha_2=1$. Then:
\begin{equation}
\mathbb{E}\left(X'_1\right)=\frac{1}{2}+\frac{1}{2 \, \lambda^2}\left(\lambda^2-4 \, \lambda+8-8 \, e^{-\frac{\lambda}{2}}\right),
\label{eq:exp.mean.nc1.11.beta}
\end{equation}
\begin{equation}
\mathbb{E}\left(X'_2\right)=\frac{1}{2}-\frac{1}{2 \, \lambda^2}\left(\lambda^2-4 \, \lambda+8-8 \, e^{-\frac{\lambda}{2}}\right),
\label{eq:exp.mean.nc2.11.beta}
\end{equation}
\begin{equation}
\mbox{\normalfont{Var}}\left(X'_1\right)=\mbox{\normalfont{Var}}\left(X'_2\right)=\frac{4}{\lambda^2}+\frac{8}{\lambda^4}\left[\lambda^2 \, e^{-\frac{\lambda}{2}} -2\left(1-e^{-\frac{\lambda}{2}}\right)\left(\lambda+1-e^{-\frac{\lambda}{2}}\right)\right].
\label{eq:exp.var.nc12.11.beta}
\end{equation}
\end{proposition}

In view of the foregoing arguments, note that $\mathbb{E}(X'_1)>\frac{1}{2}$ and $\mathbb{E}(X'_2)<\frac{1}{2}$ for every $\lambda>0$.

\subsection{Applications}
\label{subsec:dnc.beta.examples}

The applicative potential of the $\mbox{B}''$ model is now highlighted through the analysis of real data. To this end, we first turned our attention to three significant examples arisen respectively from the sectors of geology, economics and psychology.

More specifically, we focused on the proportion of sand in 21 sediment specimens, the proportion of males involved in agriculture as occupation for 47 French-speaking provinces of Switzerland at about 1888 and the subjective diagnostic probability of calculus deficiency assigned by 15 statisticians. The first data are available in \cite{Ait03} (p. 380) and details about the geologic interpretation can be found in \cite{McC75}. The second data are taken from the ``swiss'' data set, which is included in the \texttt{R} ``datasets'' package (\href{http://stat.ethz.ch/R-manual/R-devel/library/datasets/html/swiss.html}{link}); details about the reference frame are available in the \texttt{R} on-line Documentation and in the references quoted therein. Finally, the third data are again available in \cite{Ait03} (p. 375).

In a comparative perspective, five distributions were fitted to the above mentioned three data sets. The first is the standard beta, that is one of the most frequently employed to model proportions. The second is the doubly non-central beta, which is the subject of interest in the present work.

The third model is the three-parameter generalization of the beta distribution proposed by Libby and Novick \cite{LibNov82}, the density function of which has been previously reported. In the notation of Eq.~(\ref{eq:g3b.dens}), the parameter $\gamma>0$ allows the G3B density to take a much wider variety of shapes than the beta one. In particular, when $\beta_1=\beta_2=1$, the latter shows a more flexible behavior at the unit interval endpoints than the beta. In fact, its limits at 0 and 1 have the following expressions:
$$
\lim_{x \rightarrow 0^+} \mbox{\normalfont{G3B}}\left(x;1,1,\gamma\right)=\gamma, \quad \lim_{x \rightarrow 1^-} \mbox{\normalfont{G3B}}\left(x;1,1,\gamma\right)=\frac{1}{\gamma},
$$
which are clearly subject to the strong constraint of being mutual to each other. Furthermore, we recall that the $r$-th moment about zero of $X \sim \mbox{\normalfont{G3B}}\left(\beta_1,\beta_2,\gamma\right)$ is:
\begin{equation}
\mathbb{E}\left(X\right)^r=\frac{\left(\beta_1\right)_r}{\left(\beta_1+\beta_2\right)_r} \, \gamma^{\beta_1} \, _2F_1\left(\beta_1+r,\beta_1+\beta_2;\beta_1+\beta_2+r;1-\gamma\right),
\label{eq:momr.g3b.def}
\end{equation}
where $_2F_1=\sum_{k=0}^{+\infty} \frac{\left(a\right)_k \left(b\right)_k}{\left(c\right)_k}  \frac{x^k}{k!}$, $\left|x\right|<1$ is the Gauss hypergeometric function. Notice that the $_2F_1$ function in Eq.~(\ref{eq:momr.g3b.def}), despite its representation in terms of infinite series converges only for $\left|1-\gamma\right|<1$, can be computed for any $\gamma>0$ by using suitably one of the Euler transformation formulas \cite{SriKar85}:
\begin{eqnarray*}
_2F_1\left(a,b;c;x\right) & = & \left(1-x\right)^{-a} \, _2F_1\left(a,c-b;c;\frac{x}{x-1}\right)\\
& = & \left(1-x\right)^{-b} \, _2F_1\left(c-a,b;c;\frac{x}{x-1}\right)\\
& = & \left(1-x\right)^{c-a-b} \, _2F_1\left(c-a,c-b;c;x\right),
\end{eqnarray*}
that enable to rewrite such function to have absolute values of the argument less than one.

Then, we considered the Gauss hypergeometric model \cite{ArmBay94}. A random variable $X$ is said to have a Gauss hypergeometric distribution with shape parameters $a>0$, $b>0$ and additional parameters $\lambda \in \mathbb{R}$, $z>-1$, denoted by $\mbox{GH}\left(a,b,\lambda,z\right)$, if its probability density function is:
$$
\mbox{\normalfont{GH}}\left(x;a,b,\lambda,z\right)=\frac{\mbox{\normalfont{Beta}}\left(x;a,b\right)}{\left(1+z \, x\right)^{\lambda} \, _2F_1\left(\lambda,a;a+b;-z\right) }, \quad 0<x<1.
$$
Note that the case $z=0$ corresponds to the beta distribution. When $a=b=1$ its limiting values are given by the following functions of $\lambda$ and $z$:
$$
\lim_{x \rightarrow 0^+} \mbox{\normalfont{GH}}\left(x;1,1,\lambda,z\right)=\frac{1}{_2F_1\left(\lambda,1;2;-z\right)},
$$
$$
\lim_{x \rightarrow 1^-} \mbox{\normalfont{GH}}\left(x;1,1,\lambda,z\right)=\frac{1}{\left(1+z\right)^{\lambda} \, _2F_1\left(\lambda,1;2;-z\right)};
$$
the latter are analitically hard, not so easily interpretable and not particularly simplified by using any of the transformation formulas of the $_2F_1$ function that are valid in case of specific values for its arguments. In this regard, see \cite{SriKar85}. Moreover, the $r$-th moment about zero of $X \sim \mbox{\normalfont{GH}}\left(a,b,\lambda,z\right)$ is:
$$
\mathbb{E}\left(X\right)^r=\frac{\left(a\right)_r}{\left(a+b\right)_r} \frac{_2F_1\left(\lambda,a+r;a+b+r;-z\right)}{_2F_1\left(\lambda,a;a+b;-z\right)}.
$$

Finally, we used the confluent hypergeometric model, proposed by Gordy \cite{Gor98}, who applied it to the auction theory. A random variable $X$ is said to have a confluent hypergeometric distribution with shape parameters $c>0$, $d>0$ and additional parameter $\delta \in \mathbb{R}$, denoted by $\mbox{CH}\left(c,d,\delta\right)$, if its probability density function is:
$$
\mbox{\normalfont{CH}}\left(x;c,d,\delta\right)=\frac{\mbox{\normalfont{Beta}}\left(x;c,d\right)  \, e^{-\delta \, x}}{ _1F_1\left(c;c+d;-\delta\right)} , \quad 0<x<1.
$$
Note that the case $\delta=0$ corresponds to the beta distribution. In this regard, by taking $a=1$ and $z=-\delta$ in the following formula:
\begin{equation}
_1F_1\left(a;a+1;z\right)=a\left(-z\right)^{-a}\left[\Gamma\left(a\right)-\Gamma\left(a,-z\right)\right]
\label{eq:form2.1f1}
\end{equation}
(\href{http://functions.wolfram.com/HypergeometricFunctions/Hypergeometric1F1/03/01/02/}{link}), where $\Gamma\left(a,-z\right)=\int_{-z}^{+\infty}t^{a-1} \, e^{-t} \, dt$ is the incomplete gamma function, we obtain:
$$
_1F_1\left(1;2;-\delta\right)=\frac{1}{\delta}\left(1-\int_{\delta}^{+\infty}e^{-t} \, dt\right)=\frac{e^{\delta}-1}{\delta \, e^{\delta}}.
$$
Hence, when $c=d=1$, the CH density function takes on the following form:
$$
\mbox{\normalfont{CH}}\left(x;1,1,\delta\right)=\frac{\delta \, e^{\delta\left(1-x\right)}}{e^{\delta}-1}, \quad 0<x<1;
$$
moreover, as $x$ tends to the endpoints of $\left(0,1\right)$, the latter tends to:
$$
\lim_{x \rightarrow 0^+} \mbox{\normalfont{CH}}\left(x;1,1,\delta\right)=\frac{\delta \, e^{\delta}}{e^{\delta}-1}, \quad \lim_{x \rightarrow 1^-} \mbox{\normalfont{CH}}\left(x;1,1,\delta\right)=\frac{\delta}{e^{\delta}-1}.
$$
The $r$-th moment about zero of $X \sim \mbox{\normalfont{CH}}\left(c,d,\delta\right)$ is:
$$
\mathbb{E}\left(X\right)^r=\frac{\left(c\right)_r}{\left(c+d\right)_r} \frac{ _1F_1\left(c+r;c+d+r;-\delta\right)}{_1F_1\left(c;c+d;-\delta\right)}.
$$
The latter can be easily obtained by using the definition of $r$-th moment of a random variable and the integral representation of the Kummer's confluent hypergeometric function \cite{SriKar85}, that is:
$$
\frac{\Gamma\left(b-a\right) \, \Gamma\left(a\right)}{\Gamma\left(b\right)} \, _1F_1\left(a;b;z\right)=\int_{0}^{1}t^{a-1} \left(1-t\right)^{b-1} e^{z t} \, dt, \qquad b>a>0.
$$

That said, the method of moments was applied in order to obtain the estimates for the parameters of each model. The shape parameters were assigned unitary values in all the models except obviously for the beta one. In case of two parameters to be estimated, the mean and the variance of the model were set simultaneously equal to the values of the corresponding sample statistics; in case of one parameter, instead, only the mean was considered. In particular, the formulas in Eqs.~(\ref{eq:exp.mean.dnc11.beta}),~(\ref{eq:exp.var.dnc11.beta}) were used as the expressions for the mean and the variance of the $\mbox{B}''$ model. Clearly, none of the estimates of interest admits an explicit expression (except for the beta one), due to the hard analytical formulas of the moments. Therefore, the aforementioned systems of equations were solved numerically by means of the built-in function ``FindRoot'' of \textit{Mathematica} language.

The data histogram together with the estimated densities of the five models considered are shown in Figure~\ref{fig:COMP1} for the first analysis setting, in Figure~\ref{fig:COMP2} for the second one and in Figure~\ref{fig:COMP3} for the third one.

Notice that the beta and the doubly non-central beta produce a fairly accurate fit for all the proportions/probabilities considered (left-hand panels), exhibiting a substantially better fit than the G3B, the GH and the CH, the inadequate performances of which are evident (right-hand panels).

In particular, it is remarkable that the $\mbox{B}''$ distribution allows for the tails of the data histograms to be captured and modeled and appears to be more helpful in data interpretation, as it recognizes the presence of values next to zero and one by showing positive and finite limits. On the contrary, the beta distribution cannot display such ability. Furthermore, none of the three alternative models considered enables to capture this peculiarity of the data pattern. Therefore, for such models, good fitting and having positive and finite limits would sound irreconcilable features.

\begin{figure}[ht]
 \centering
 \subfigure
   {\includegraphics[width=6cm]{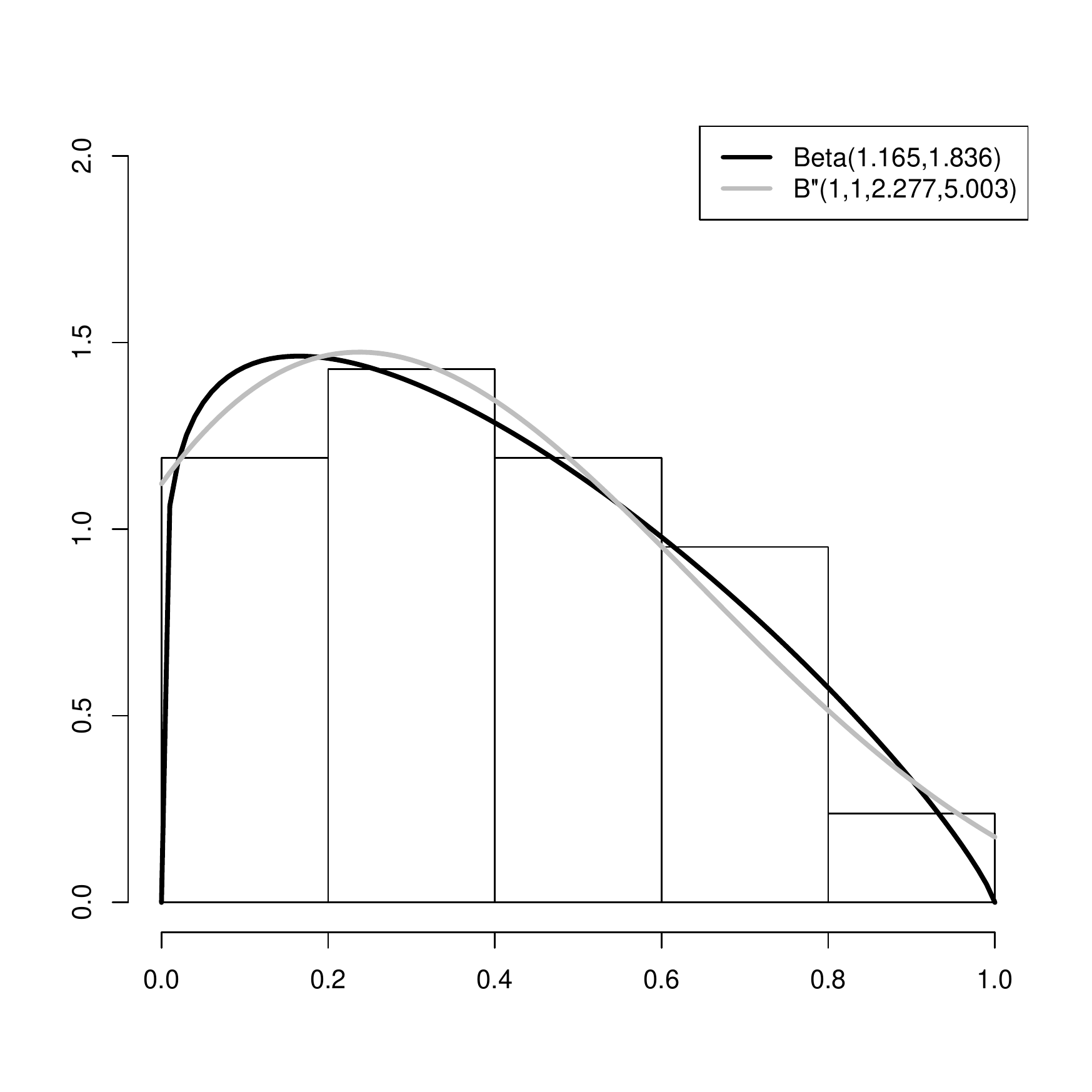}}
 \hspace{5mm}
 \subfigure
   {\includegraphics[width=6cm]{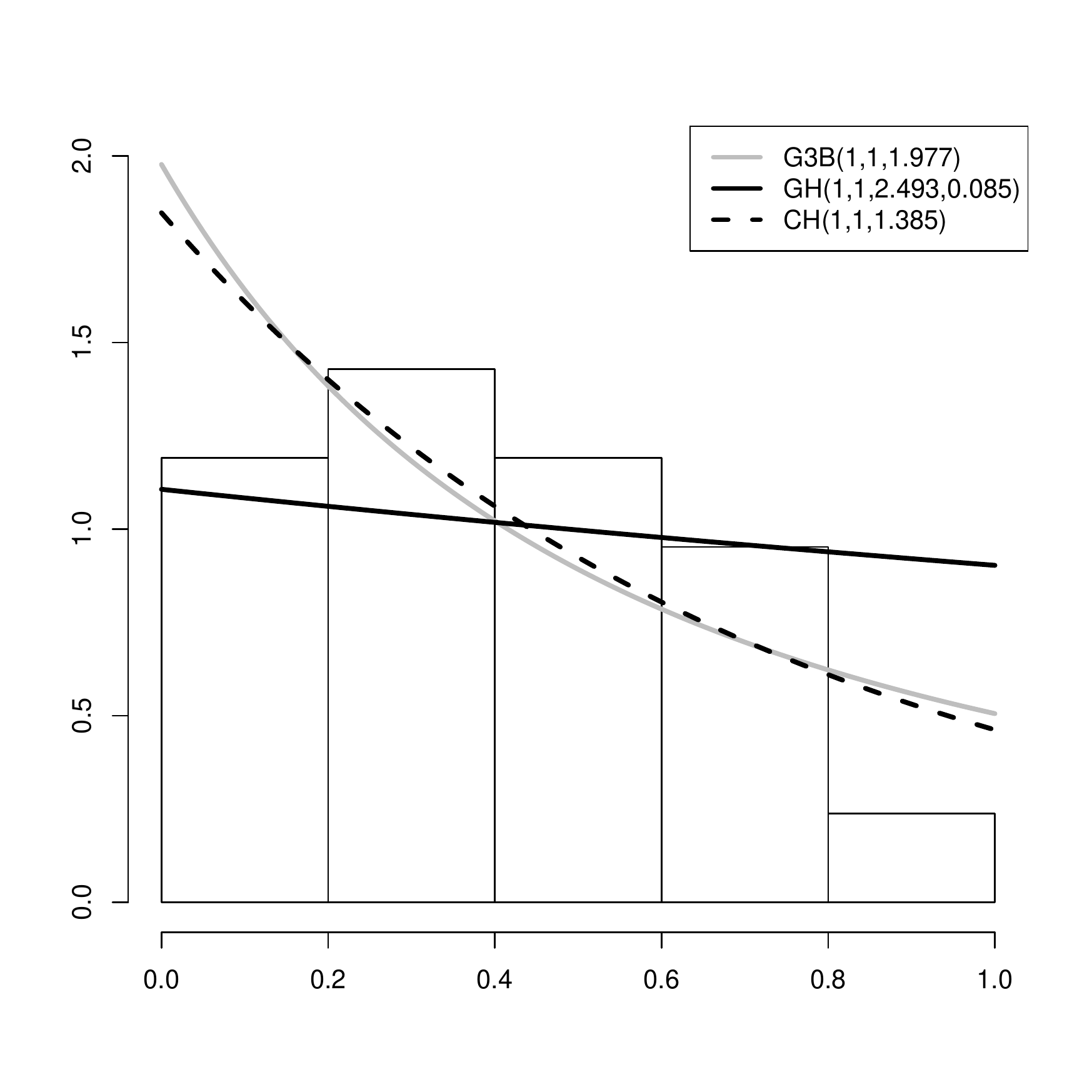}}
 \caption{Histogram of the proportion of sand in 21 sediment specimens superimposed with the estimated densities of the beta model, the doubly non-central beta model with unitary shape parameters (left-hand panel) and the G3B, the GH, the CH models with unitary shape parameters (right-hand panel).}
\label{fig:COMP1}
 \end{figure}

\begin{figure}[ht]
 \centering 
 \subfigure
   {\includegraphics[width=6cm]{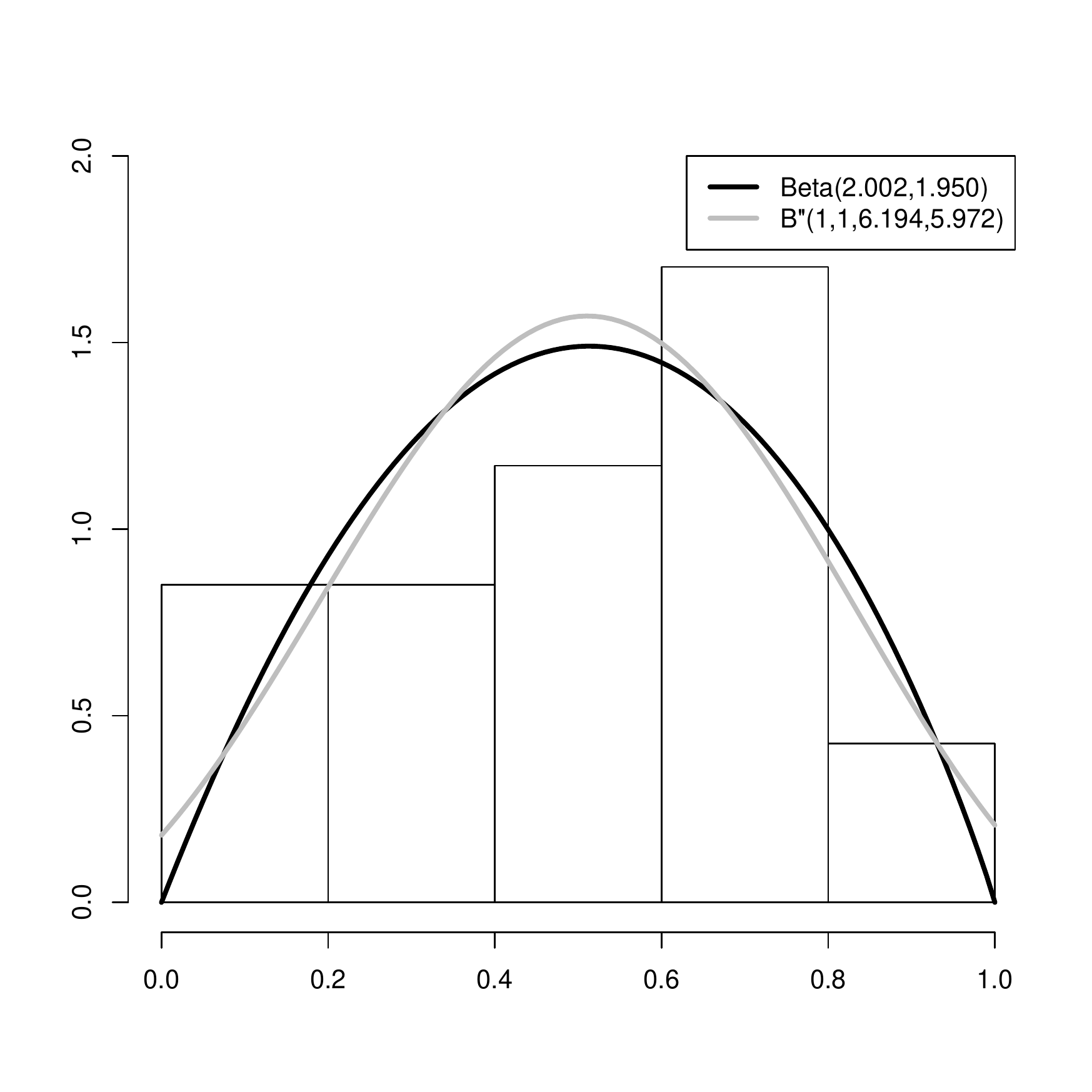}}
 \hspace{5mm}
 \subfigure
   {\includegraphics[width=6cm]{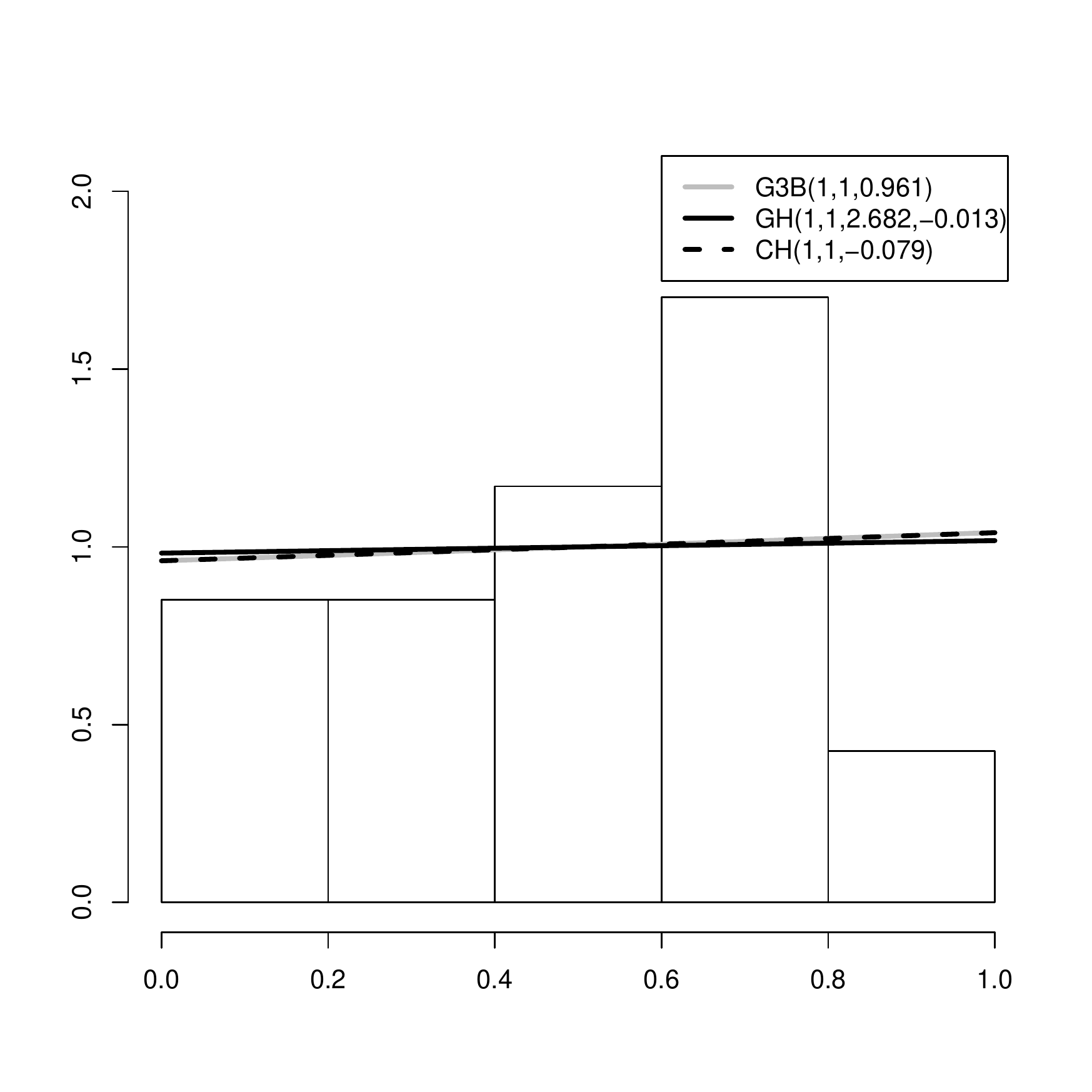}}
\caption{Histogram of the proportion of males involved in agriculture as occupation for 47 French-speaking provinces of Switzerland at about 1888 superimposed with the estimated densities of the beta model, the doubly non-central beta model with unitary shape parameters (left-hand panel) and the G3B, the GH, the CH models with unitary shape parameters (right-hand panel).}
\label{fig:COMP2}
\end{figure}

\begin{figure}[ht]
 \centering 
 \subfigure
   {\includegraphics[width=6cm]{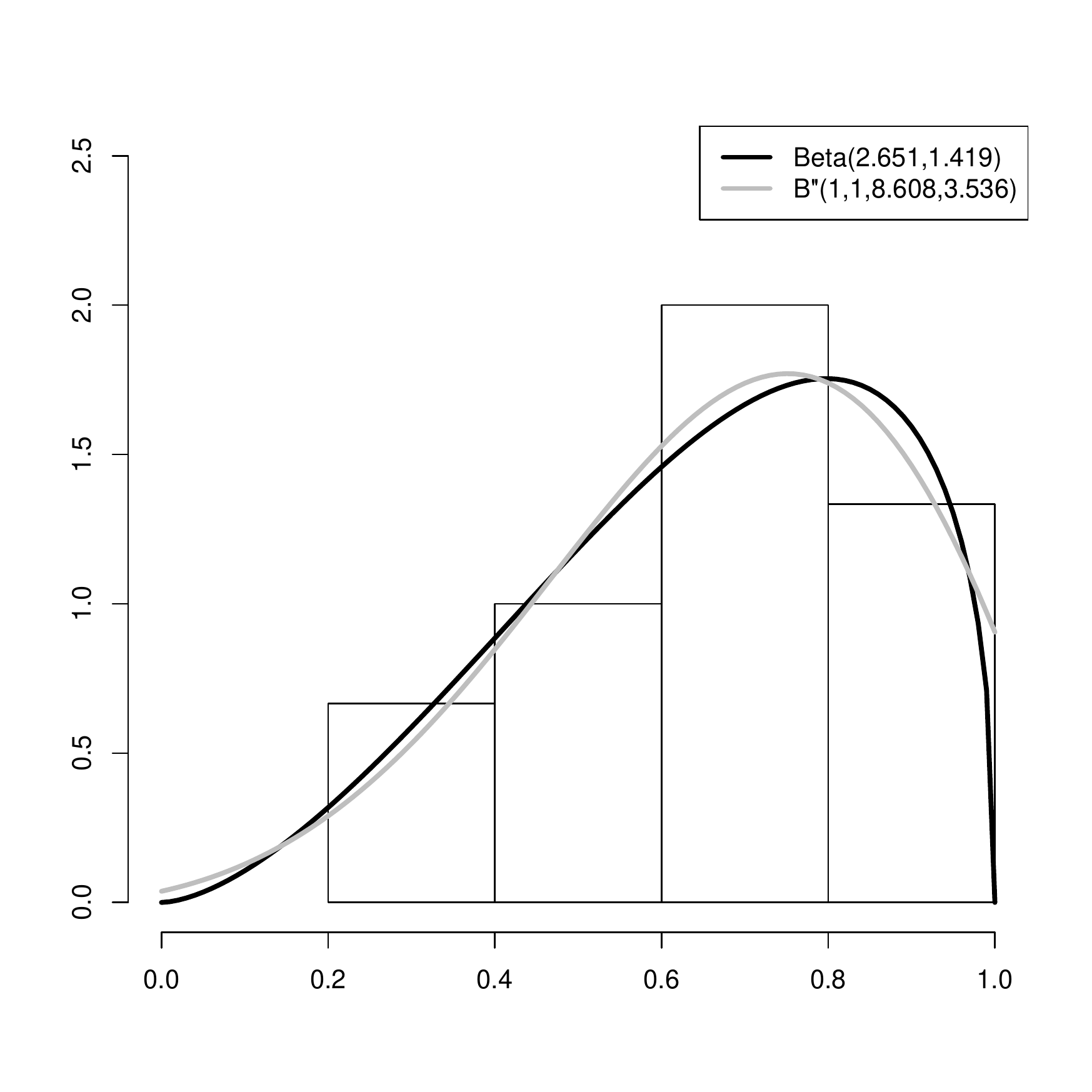}}
 \hspace{5mm}
 \subfigure
   {\includegraphics[width=6cm]{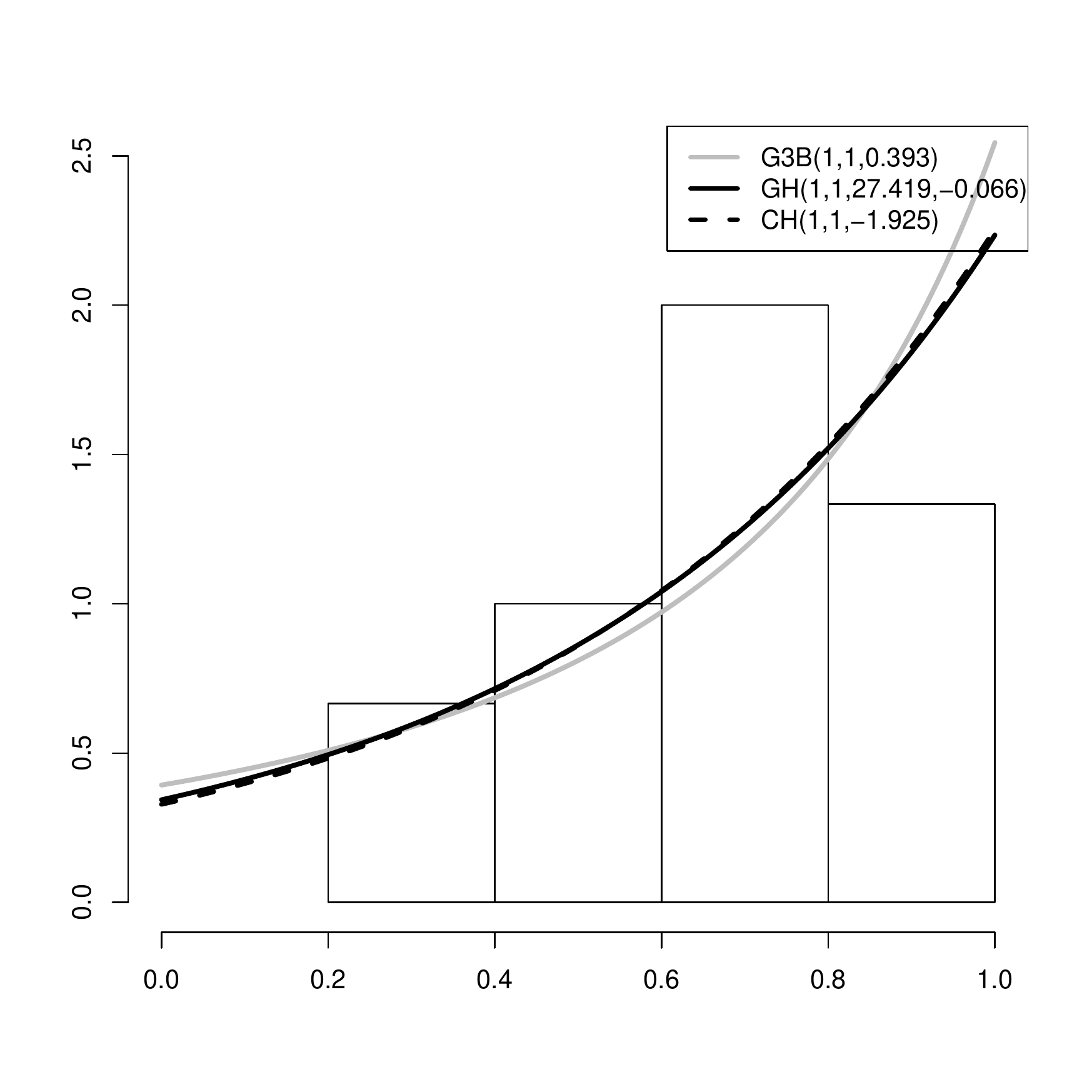}}
\caption{Histogram of the subjective diagnostic probability of calculus deficiency assigned by 15 statisticians superimposed with the estimated densities of the beta model, the doubly non-central beta model with unitary shape parameters (left-hand panel) and the G3B, the GH, the CH models with unitary shape parameters (right-hand panel).}
\label{fig:COMP3}
\end{figure}

Finally, it should be stressed that there might be situations in which the method of moments results in negative estimates for one of the two non-centrality parameters of the $\mbox{B}''$ distribution. This may be explained by incompatibilities between the data concerned and the range of shapes the $\mbox{B}''$ density can take on for $\alpha_1=\alpha_2=1$ by varying $\lambda_1$, $\lambda_2$. To avoid this inconvenience, the method of moments could be adopted to estimate all four parameters of the $\mbox{B}''$ model without fixing the shape parameters to one. However, in this case the moments formulas to be used are special cases of Eq.~(\ref{eq:momr.beta.dnc}) and therefore are computationally heavy. Moreover, no real solutions might exist. Then, one can fall back to using the type 1 or the type 2 non-central beta distributions. More specifically, in the case of negative estimates for $\lambda_1$ and $\lambda_2$, the type 2 and the type 1 models should be respectively fitted to data.

By way of example, we considered the proportion by weight of cornite in 25 specimens of kongite. The latter data are once again available in \cite{Ait03} (p. 356).

In order to derive the method of moments estimates for the $\mbox{B}''\left(1,1,\lambda_1,\lambda_2\right)$ distribution, the mean and the variance of such model were set simultaneously equal to the values of the corresponding sample statistics, obtaining the following solutions $\tilde{\lambda}_1=-0.530932$, $\tilde{\lambda}_2=7.36794$ for the non-centrality parameters. As $\tilde{\lambda}_1<0$, we are lead to believe that the present situation is not well suited for being modeled by the doubly non-central beta distribution while the type 2 model may be more appropriate in this respect. The latter was thus fitted to the aforementioned data in place of the $\mbox{B}''$ and the simple formula in Eq.~(\ref{eq:exp.mean.nc2.11.beta}) was used as the expression for the mean.

That said, Figure~\ref{fig:COMP4} shows the data histogram and the estimated densities of the five models considered.

In the previous examples the differences in fitting between the $\mbox{B}''$ and the three alternative models were unquestionably clear. In the present case, instead, their performances are indeed comparable and all satisfactory, with the small exception of the G3B model, the density of which shows a much higher limiting value at zero (right-hand panel), due to the strong constraint existing	between its limits. This means that, contrarily to the $\mbox{B}''$ distribution, the potential of the $\mbox{B}'_2$ in capturing the observations with low and high values does not exceed the one of the other three models and similar conclusions might be drawn with reference to the $\mbox{B}'_1$ where appropriate.

\begin{figure}[ht]
 \centering 
 \subfigure
   {\includegraphics[width=6cm]{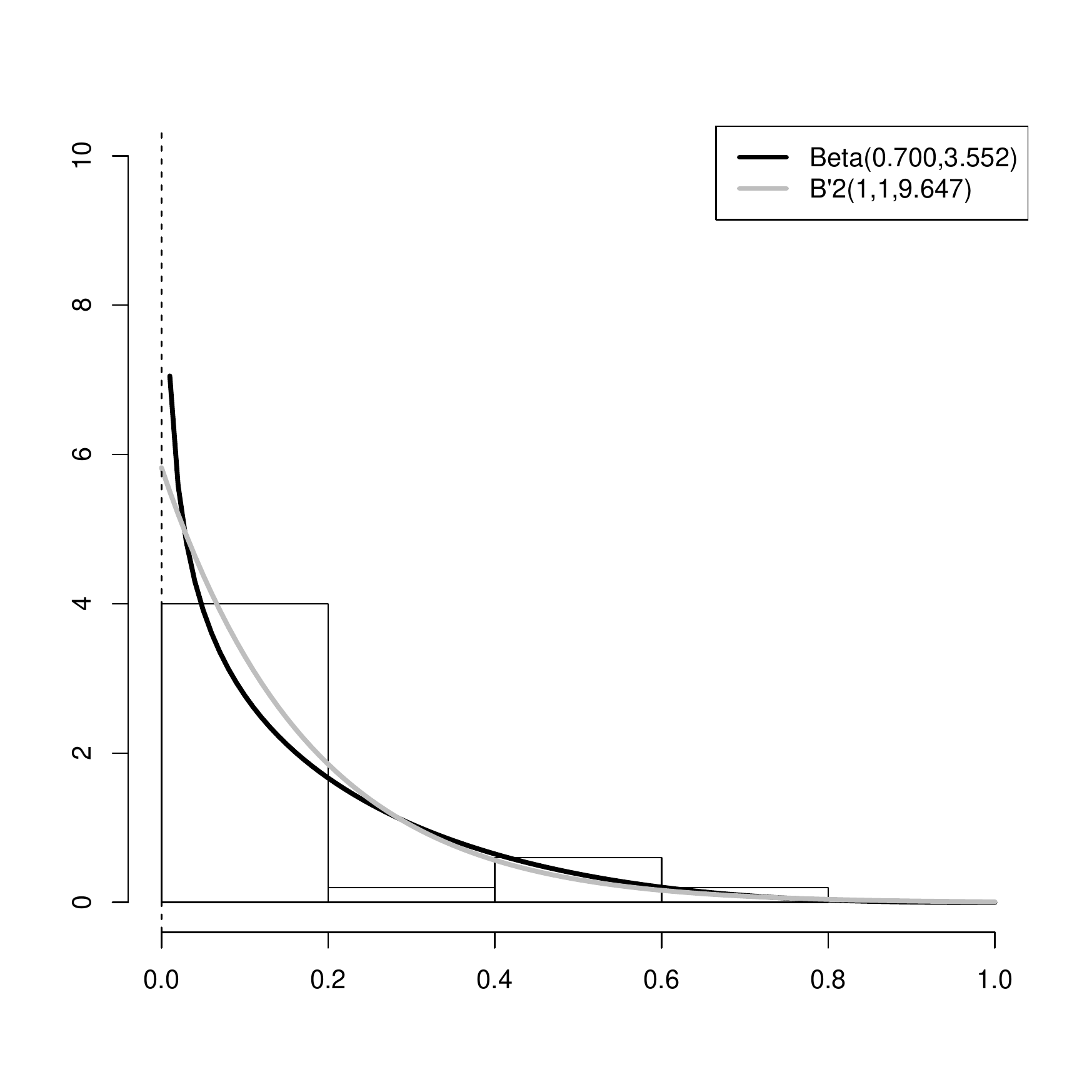}}
 \hspace{5mm}
 \subfigure
   {\includegraphics[width=6cm]{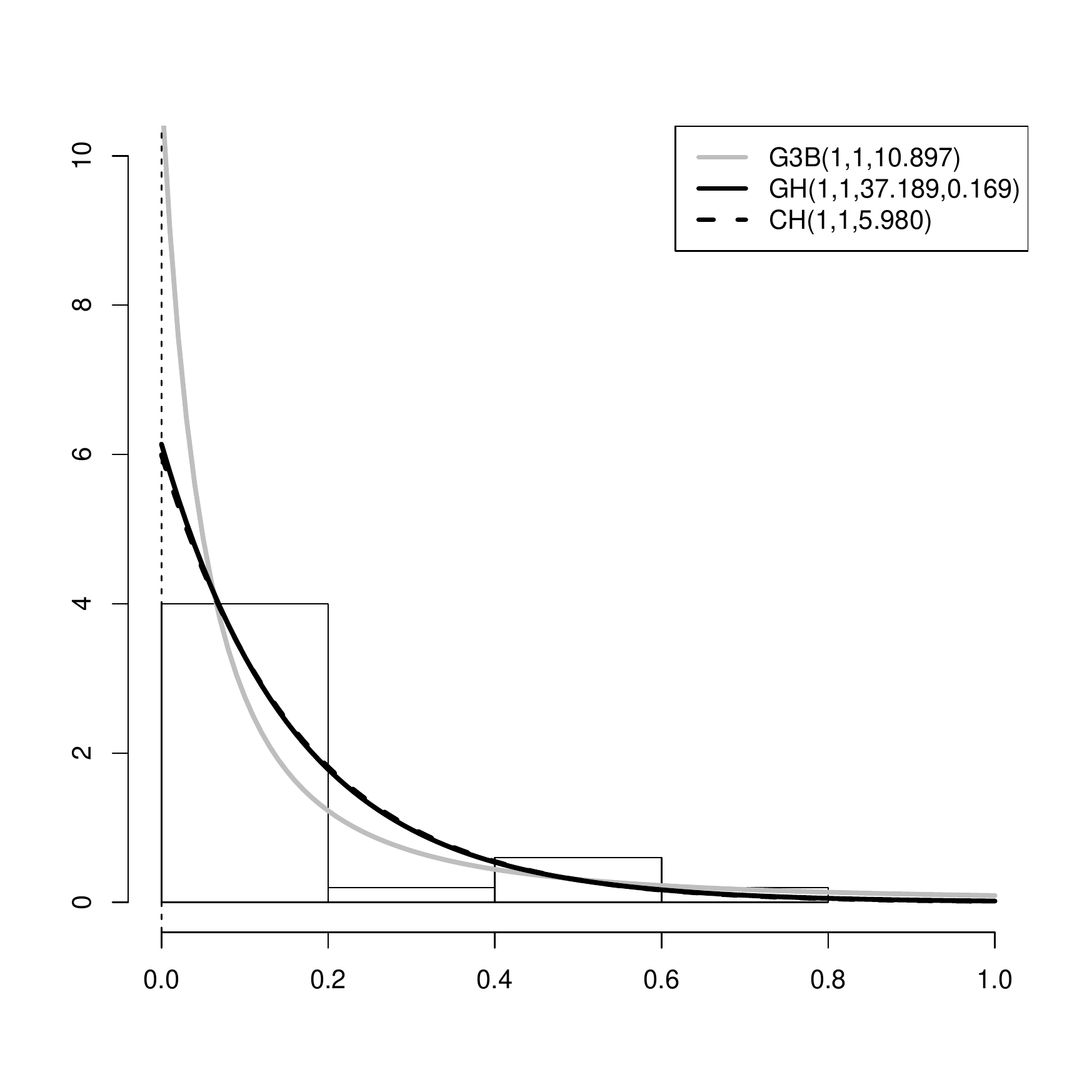}}
\caption{Histogram of the proportion by weight of cornite in 25 specimens of kongite superimposed with the estimated densities of the beta model, the type 2 non-central beta model with unitary shape parameters (left-hand panel) and the G3B, the GH, the CH models with unitary shape parameters (right-hand panel).}
\label{fig:COMP4}
\end{figure}
 
The foregoing conclusions lead us to compare the performances of the beta and the doubly non-central beta distributions more deeply. For this purpose, we resorted to the Akaike information criterion \cite{Aka74}, given by $\mbox{AIC}=-2 \, l(\hat{\theta})+2 \, p$, where $l=l(\theta)$ is the log-likelihood function for the $p$-dimensional vector $\theta$ of the model parameters and $\hat{\theta}$ is the maximum-likelihood estimate of $\theta$. The distribution with the smallest value for this criterion is taken as the one that gives the best description of the data. 

In this regard, let $X_1,\ldots,X_n$ be independent random variables with identical distribution depending on an unknown parameter vector $\theta$ to be estimated.

Suppose first that $X_i \sim \mbox{Beta}\left(\alpha_1,\alpha_2\right)$, $i=1,\ldots,n$. It is well known that, given the observed sample $x=\left(x_1,\ldots,x_n\right)$, the log-likelihood function for the vector $\left(\alpha_1,\alpha_2\right) \in \mathbb{R}^{2+}$ of the shape parameters is:
$$
l\left(\alpha_1,\alpha_2;x\right)=\left(\alpha_1-1\right) \sum_{i=1}^{n}\log x_i+\left(\alpha_2-1\right)\sum_{i=1}^{n}\log \left(1-x_i\right)-n\log\left[B\left(\alpha_1,\alpha_2\right)\right].
$$
Secondly, let $X_i \sim \mbox{B}''\left(1,1,\lambda_1,\lambda_2\right)$, $i=1,\ldots,n$. Then, by resorting to Eq.~(\ref{eq:c.ncb}), it's easy to see that the log-likelihood function for the vector $\left(\lambda_1,\lambda_2\right) \in \mathbb{R}^{2+}$ of the non-centrality parameters is given by:
\begin{equation}
l\left(\lambda_1,\lambda_2;x\right)=-\frac{n}{2} \left(\lambda_1+\lambda_2\right) +\sum_{i=1}^{n}\log \Psi_2\left[2;1,1;\frac{\lambda_1}{2}x_i,\frac{\lambda_2 }{2}\left(1-x_i\right)\right].
\label{eq:loglik.dnc.beta}
\end{equation}

The number of parameters of both the above models is $p=2$; therefore, in view of the AIC definition, the maximized value of the log-likelihood is the only discriminating criterion between them.

In the present setup the log-likelihoods are to be maximized numerically. This procedure can be easily accomplished by using for example the ``optim'' built-in-function from the \texttt{R} statistical package or, alternatively, the ``FindMaximum'' one from \textit{Mathematica} software. These routines are able to locate the maximum of the log-likelihood surface for a wide range of starting values. However, to ease computations, it is useful to have reasonable starting values, such as, for example, the method of moments estimates.

That said, the aforementioned algorithms were applied on the four data sets subject to the analyses previously carried out in this Section.

The standard errors of such estimates can be evaluated by recalling that, under suitable regularity conditions, the maximum-likelihood estimator $\hat{\Theta}$ of $\theta$ is asymptotically distributed according to a multivariate normal with mean vector $\theta$ and asymptotic covariance matrix that can be approximated by the inverse of the observed information matrix $I(\hat{\theta};x)=\left\{-\frac{\partial^2l\left(\theta;x\right)}{\partial \theta \, \partial \theta^{T}}\right\}_{\theta=\hat{\theta}}$. The required second-order derivatives can be computed numerically by means of the \texttt{R} ``optim'' function.

Table~\ref{tab:ml.estim.aic} lists the maximum-likelihood estimates, their standard errors and the AIC statistics of the two models of interest for each of the above cases, labelled as ``sand'', ``male'', ``calculus'' and ``cornite''. By the comparison of the present results with the first ones, it's immediate to see that the parameters estimates are very similar for both methods. In particular, it's to be noted that the considerations previously drawn with regards to the ``cornite'' data are now confirmed by applying the maximum-likelihood approach. Indeed, as $\hat{\lambda}_1=0$, we are inclined to think that the $\mbox{B}''$ distribution results overparametrized to model such data while its special case $\mbox{B}'_2$ is enough to this end. Furthermore, the results indicate that the $\mbox{B}''$ model has the smallest value for the AIC statistic in half the cases (in bold in Table~\ref{tab:ml.estim.aic}). So, in these cases the latter could be chosen as the most suitable model.

\begin{table}[ht]
\centering
\small
\caption{Maximum-likelihood estimates, standard errors (SE) and AIC statistics for the beta model and the doubly non-central beta model with unitary shape parameters in the four case studies ``sand'', ``male'', ``calculus'' and ``cornite'' (the values mentioned in the text are written in bold).}\label{tab:ml.estim.aic}
\begin{tabular}{c|c|c|c|c|c|c}
Data & \multicolumn{3}{c| }{$\mbox{\normalfont{Beta}}\left(\alpha_1,\alpha_2\right)$} & \multicolumn{3}{c}{$\mbox{\normalfont{B}}''\left(1,1,\lambda_1,\lambda_2\right)$}\\
\cline{2-7} & \multicolumn{1}{c}{$\hat{\alpha}_1$ (SE)} & \multicolumn{1}{c}{$\hat{\alpha}_2$ (SE)} & \multicolumn{1}{c| }{AIC} & \multicolumn{1}{c}{$\hat{\lambda}_1$ (SE)} & \multicolumn{1}{c}{$\hat{\lambda}_2$ (SE)} & \multicolumn{1}{c}{AIC} \\ \hline
sand & 1.089(0.304) & 1.757(0.530) & \textbf{-0.55} & 2.096(1.882) & 4.754(2.939) & -0.352\\
male & 1.854(0.362) & 1.898(0.372) & -5.964 & 6.257(1.936) & 6.066(1.889) & \textbf{-6.684}\\
calculus & 2.757(0.998) & 1.479(0.497) & \textbf{-3.044} & 8.893(4.785) & 3.691(2.570) & -2.668\\
cornite & 0.950(0.236) & 4.647(1.408) & -35.450 & 0(1.611) & 9.646(7.376) & \textbf{-36.294}
\end{tabular}
\end{table}

Before concluding, we want to further illustrate the flexibility of the $\mbox{B}''$ distribution. To this end, we used four data sets taken from \cite{Ait03} and related to the subjects of petrology and geology. The first data set consists of the proportions of magnesium oxide in 23 specimens of aphyric Skye lavas (p. 360); more details for a petrological interpretation of the latter can be found in \cite{ThoEssDun72}. The second one deals with the proportions by weight of albite in 25 specimens of kongite (p. 356). Finally, we considered the proportion of clay in 39 sediment samples at different water depths in an Arctic Lake (p. 359), adapted from \cite{CoaRus68} (Table 1) and the proportion of abies in 30 specimens of fossil pollen from three different locations (p. 389). 

As it did before, the fit of the $\mbox{B}''$ distribution with unitary shape parameters was compared with the beta one for each of the above cases. As criteria for comparing the fits, we used the AIC statistic, based on the maximum-likelihood estimates for the model parameters. For both distributions, the latter were evalutated numerically by means of the \texttt{R} function ``optim'' using the method of moments estimates as starting values. The results of fitting are shown in Table~\ref{tab:ml.estim.aic.2}, where the aforementioned data are labelled as ``oxide'', ``albite'', ``clay'' and ``abies''. Moreover, Figures~\ref{fig:Ml1}, \ref{fig:Ml2} show the data histograms superimposed with the fitted probability density functions of the two models.

\begin{table}[ht]
\centering
\small
\caption{Maximum-likelihood estimates, standard errors (SE) and AIC statistics for the beta model and the doubly non-central beta model with unitary shape parameters in the four case studies ``oxide'', ``albite'', ``clay'' and ``abies''.}\label{tab:ml.estim.aic.2}
\begin{tabular}{c|c|c|c|c|c|c}
Data & \multicolumn{3}{c| }{$\mbox{\normalfont{Beta}}\left(\alpha_1,\alpha_2\right)$} & \multicolumn{3}{c}{$\mbox{\normalfont{B}}''\left(1,1,\lambda_1,\lambda_2\right)$}\\
\cline{2-7} & \multicolumn{1}{c}{$\hat{\alpha}_1$ (SE)} & \multicolumn{1}{c}{$\hat{\alpha}_2$ (SE)} & \multicolumn{1}{c| }{AIC} & \multicolumn{1}{c}{$\hat{\lambda}_1$ (SE)} & \multicolumn{1}{c}{$\hat{\lambda}_2$ (SE)} & \multicolumn{1}{c}{AIC} \\ \hline
oxide & 2.766(0.772) & 11.555(3.459) & -39.408 & 9.590(3.585) & 46.277(14.922) & -39.64\\
albite & 16.105(4.517) & 20.761(5.843) & -51.022 & 65.481(19.104) & 84.582(24.509) & -51.768\\
clay & 1.227(0.251) & 3.074(0.706) & -23.24 & 3.738(1.490) & 11.718(3.480) & -24.6\\
abies & 16.852(4.317) & 19.958(5.125) & -61.694 & 66.644(17.738) & 79.197(20.980) & -61.814
\end{tabular}
\end{table}

\begin{figure}[ht]
 \centering 
 \subfigure
   {\includegraphics[width=6cm]{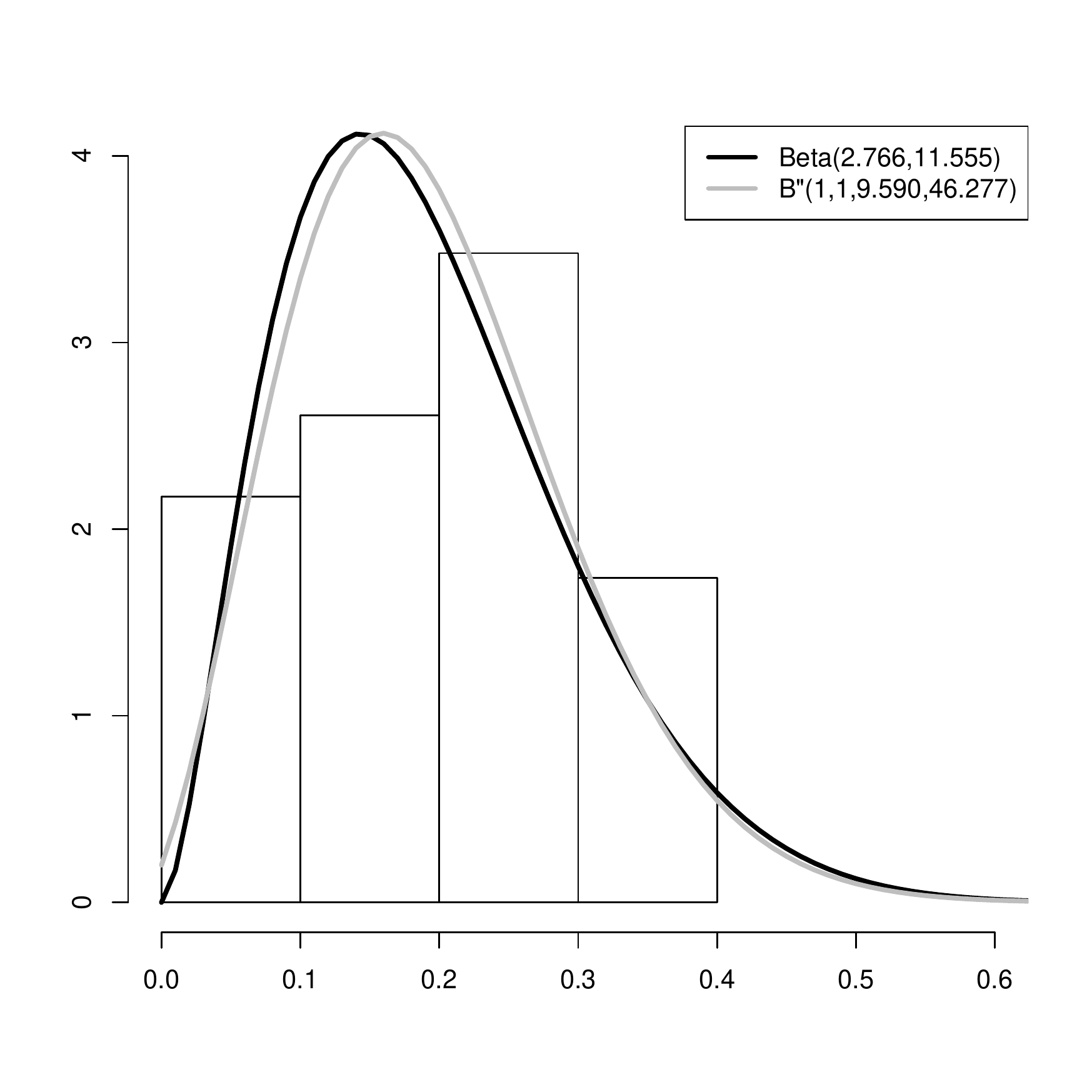}}
 \hspace{5mm}
 \subfigure
   {\includegraphics[width=6cm]{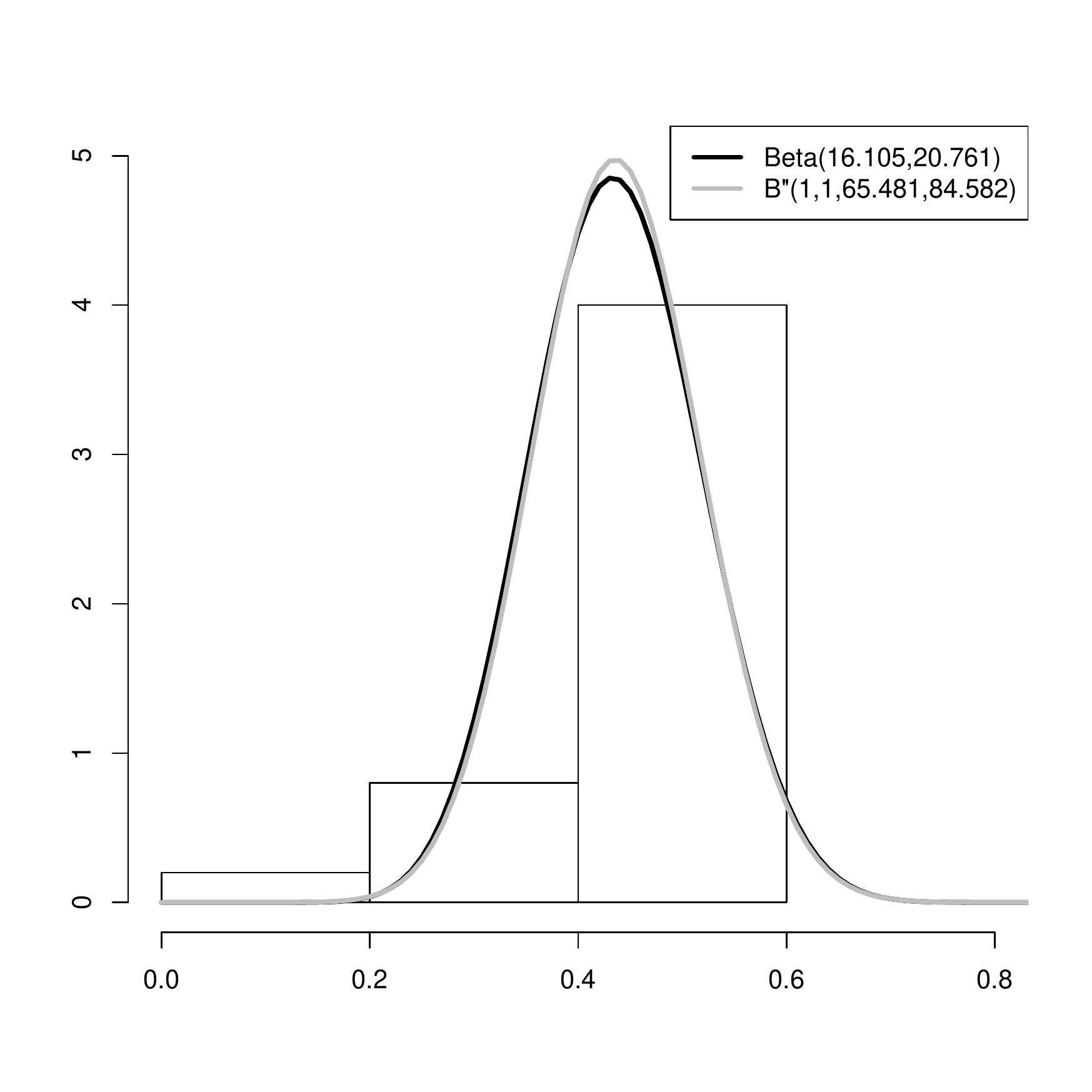}}
\caption{Histograms of the proportion of magnesium oxide in 23 specimens of aphyric Skye lavas (left-hand panel) and of the proportion by weight of albite in 25 specimens of kongite (right-hand panel) superimposed with the estimated densities of the beta model and the doubly non-central beta model with unitary shape parameters.}
\label{fig:Ml1}
\end{figure}

\begin{figure}[ht]
 \centering 
 \subfigure
   {\includegraphics[width=6cm]{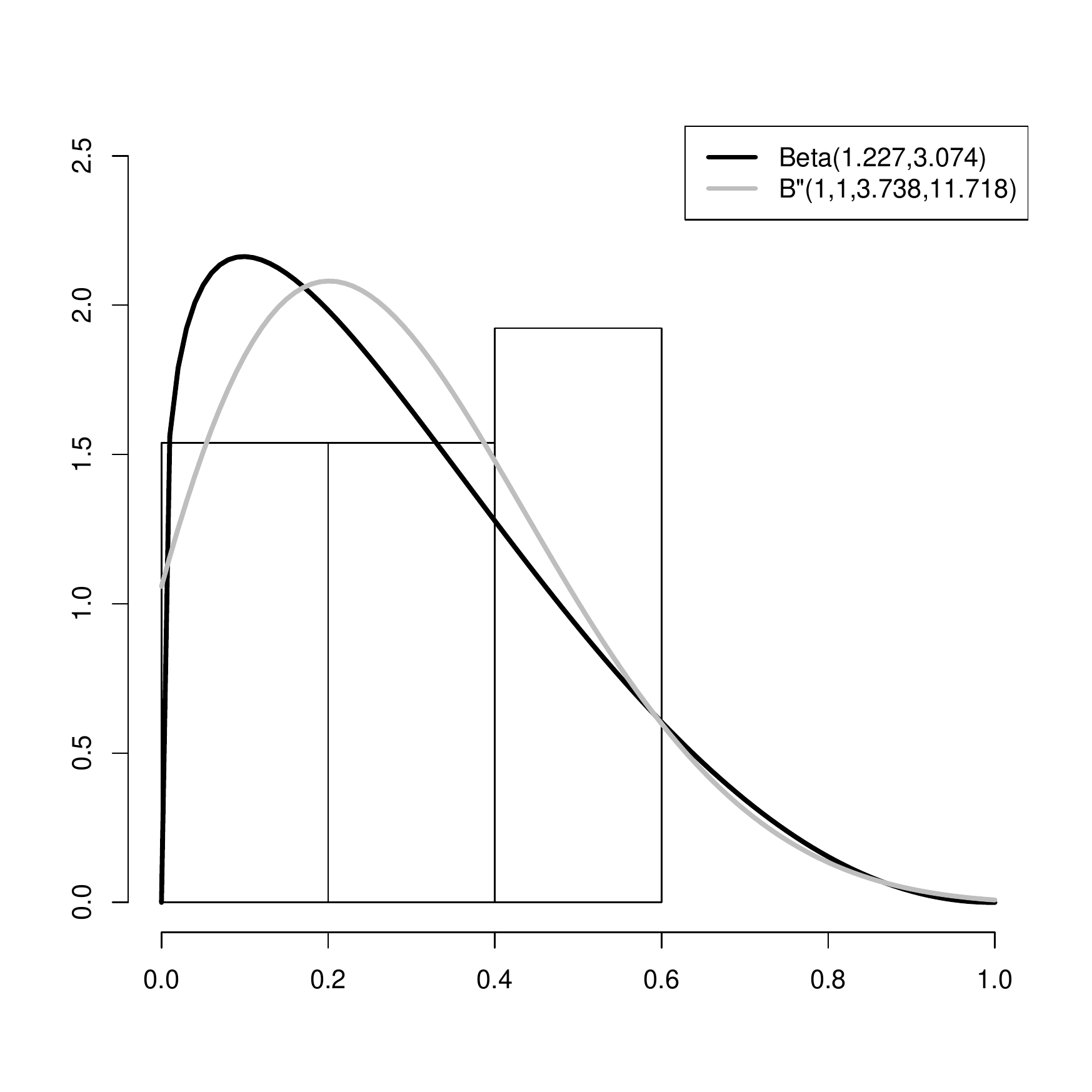}}
 \hspace{5mm}
 \subfigure
   {\includegraphics[width=6cm]{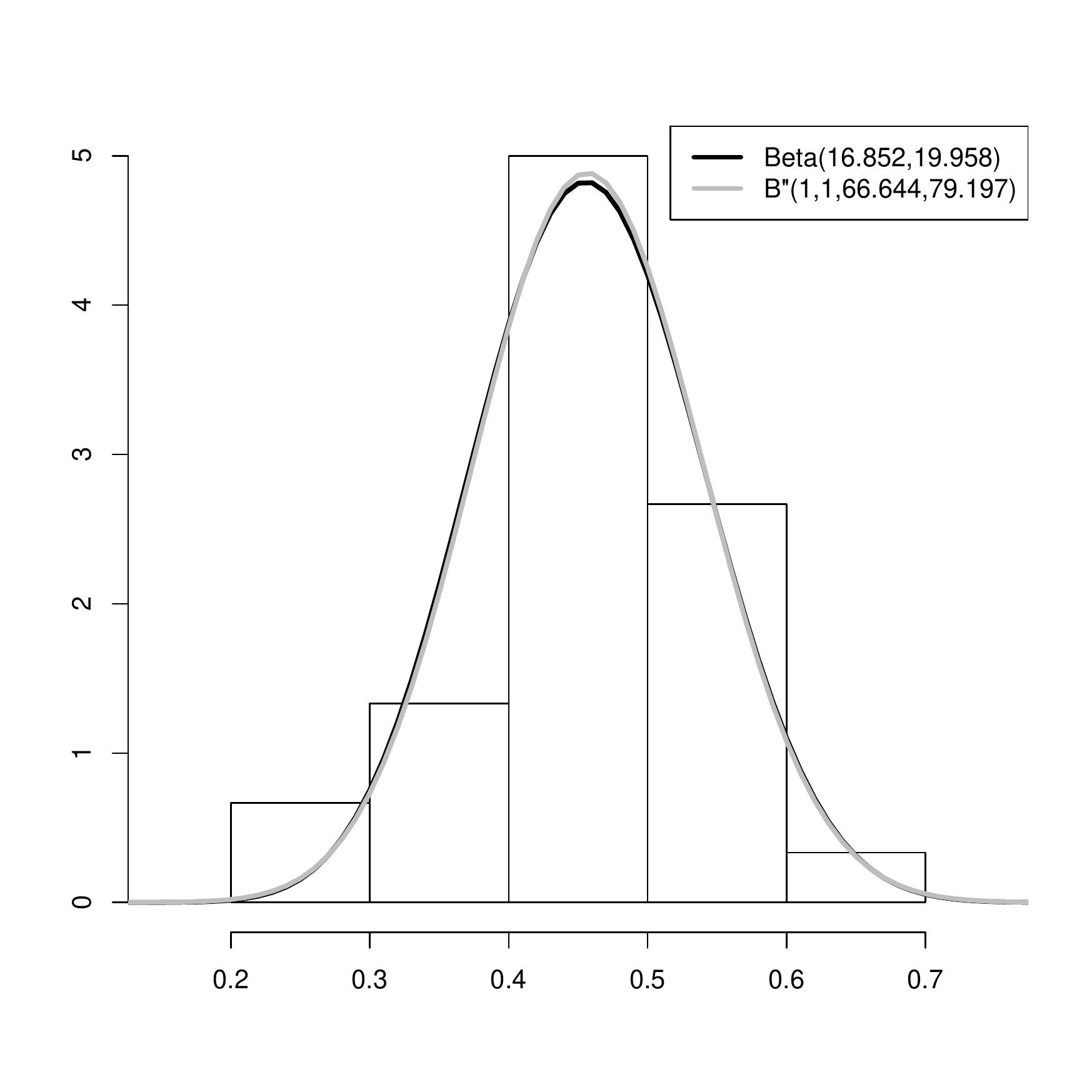}}
\caption{Histograms of the proportion of clay in 39 sediment samples at different water depths in an Arctic Lake (left-hand panel) and of the proportion of abies in 30 specimens of fossil pollen (right-hand panel) superimposed with the estimated densities of the beta model and the doubly non-central beta model with unitary shape parameters.}
\label{fig:Ml2}
\end{figure}

Note that in all these cases the $\mbox{B}''$ model has the lowest AIC and so it could be chosen as the best one. Moreover, the present analysis shows that the higher are the non-centrality parameters of the doubly non-central beta distribution, the better is the ability of the latter to model even data with no values next to zero and one (right-hand panels).

\section{Conclusions}
\label{sec:concl}

New insights into the class of the non-central beta distributions were provided in this paper. More specifically, new representations and moments expressions were derived for the doubly non-central beta distribution despite its uneasy analytical tractability. A particularly relevant advantage of this model over alternative ones on the real interval $(0,1)$, such as the beta, the Libby and Novick's generalized beta, the Gauss hypergeometric and the confluent hypergeometric models, is its ability to properly capture the tails of data by allowing its density to take on finite and positive limits. Indeed, various applications using real data proved the superior performance of the doubly non-central beta distribution over the others in terms of fitting. In particular, in many cases the doubly non-central beta showed lower values of the Akaike information criterion than the beta one, which is well known to be the most frequently employed to model proportions. That is why we hope this model may attract wider applications in statistics.

An investigation of its multivariate generalization is clearly needed. However, the poor analytical simplicity characterizing the unidimensional case can only get worse as the dimensionality increases.

In this regard, a first analysis of a more easily tractable and interpretable analogue of the doubly non-central beta distribution was carried out in \cite{OngOrs15}. An in-depth study of such distribution seems to be desirable; in fact, the latter preserves the applicative potential of the standard doubly non-central beta model and its straightforward form can make possible its extension to the multidimensional setting in a simple manner. Ultimately, it seems worthwhile to continue developing this research line by tackling it in future work.

\appendix
\section{Appendix. Proofs}
\label{sec:app.proofs}

\begin{myproof}[\textbf{Proposition~\ref{propo:expans.poch.symb.binom}}]
\label{proof:expans.poch.symb.binom}
The proof of Eq.~(\ref{eq:expans.poch.symb.binom}) follows from Eq.~(\ref{eq:poch.symb.sum}) by noting that $l=i+\left(l-i\right)$, $\forall i=0,\ldots,l$ and by doing multiplications in a way such that the left-hand side of Eq.~(\ref{eq:expans.poch.symb.binom}) can be written in the following form:
$$\left(a+b\right)_l=\sum_{i=0}^{l} \left[P_i\left(a\right)\right] \, b^i,$$
where $P_i\left(a\right)$ is a polynomial in the variable $a$ with degree equal to $l-i$, $\forall i=0,\ldots,l$. It's easy to see that $P_i\left(a\right)$ can be written in the same form as in the right-hand side of Eq.~(\ref{eq:expans.poch.symb.binom}).
\end{myproof}

\begin{myproof}[\textbf{Proposition~\ref{propo:mom.ncchisq}}]
\label{proof:mom.ncchisq}
By virtue of the law of iterated expectations, one has $\mathbb{E}\left[\left(Y' \right)^r \right]=\mathbb{E}_M\left\{\mathbb{E}\left[\left.\left(Y' \right)^r\right|M \right]\right\}$, where, in the notation of Property~\ref{prope:mixrepres.ncchisq}, $M$ is a Poisson random variable with mean $\lambda/2$ and, conditionally on $M$, $Y'$ has a $\chi^2_{g+2M}$ distribution. In view of the general formula for the moments about zero of the gamma distribution \cite{JohKotBal94}, one obtains $\mathbb{E}\left[\left.\left(Y' \right)^r\right|M \right]=2^r \left(h+M\right)_r$, where $h=g/2$; therefore:
\begin{equation}
\mathbb{E}\left[\left(Y' \right)^r \right]=2^r \, \mathbb{E}\left[\left(h+M\right)_r\right].
\label{eq:mom.ncchisq.dim1}
\end{equation}
By virtue of Proposition~\ref{propo:expans.poch.symb.binom}, Eq.~(\ref{eq:mom.ncchisq.dim1}) can be restated as follows:$$\mathbb{E}\left[\left(Y' \right)^r \right]=2^r \, \sum_{i=0}^{r} \frac{1}{i!}\left[\frac{d^i}{d h^i} \left(h\right)_r\right] \mathbb{E}\left(M^i\right),$$where, in view of the general formula for the moments about zero of Poisson distribution \cite{JohKemKot05}:
\begin{equation}
\mathbb{E}\left(M^i\right)=\sum_{j=0}^{i} \mathcal{S}\left(i,j\right) \left(\frac{\lambda}{2}\right)^j, \qquad i \in \mathbb{N},
\label{eq:mom.pois}
\end{equation}
$\mathcal{S}\left(i,j\right)$ being a Stirling number of the second kind. Thus, Eq.~(\ref{eq:mom.ncchisq}) is established.
\end{myproof}

\begin{myproof}[\textbf{Proposition~\ref{propo:mom.ncchisq.zero}}]
\label{proof:mom.ncchisq.zero}
By taking $h=0$ in Eq.~(\ref{eq:mom.ncchisq.dim1}), one has $\mathbb{E}\left[\left(Y'_{pnc} \right)^r \right]=2^r \, \mathbb{E}\left[\left(M\right)_r\right]$, where $M \sim \mbox{Poisson}\left(\lambda/2\right)$. By bearing in mind that $(M)_r=\sum_{i=0}^r \left|s\left(r,i\right)\right| M^i$ and by virtue of Eq.~(\ref{eq:mom.pois}), Eq.~(\ref{eq:mom.ncchisq.zero}) is established. 
\end{myproof}

\begin{myproof}[\textbf{Proposition~\ref{propo:ident.mom}}]
\label{proof:ident.mom}
By virtue of Eq.~(\ref{eq:poch.symb2}), one has $\Gamma\left(j+h\right)=\Gamma\left(h\right) \cdot \left(h\right)_j$, $\forall j=0,\ldots,r$, so that Eq.~(\ref{eq:mom.literat.ncchisq}) can be rewritten as follows:
\begin{equation}
\mathbb{E}\left[\left(Y' \right)^r \right]=2^r \sum_{j=0}^{r}{r \choose j} \frac{\left(h\right)_r}{\left(h\right)_j} \left(\frac{\lambda}{2}\right)^j,
\label{eq:mom.literat.ncchisq.mod}
\end{equation}
where $h=g/2$. Furthermore, by noting that $\sum_{i=0}^{r}\sum_{j=0}^{i}a_{ij}=\sum_{j=0}^{r}\sum_{i=j}^{r}a_{ij}$, Eq.~(\ref{eq:mom.ncchisq}) can be restated according to the following form:
\begin{equation}
\mathbb{E}\left[\left(Y' \right)^r \right]=2^r \, \sum_{j=0}^{r}\left(\frac{\lambda}{2}\right)^j \sum_{i=j}^{r} \mathcal{S}\left(i,j\right) \frac{1}{i!}\left[\frac{d^i}{d h^i} \left(h\right)_r\right].
\label{eq:mom.ncchisq.mod}
\end{equation}
Hence, by equating Eq.~(\ref{eq:mom.literat.ncchisq.mod}) and Eq.~(\ref{eq:mom.ncchisq.mod}), Eq.~(\ref{eq:ident.mom}) is established. 
\end{myproof}

\begin{myproof}[\textbf{Proposition~\ref{propo:dncb.cond.ind}}]
\label{proof:dncb.cond.ind}
Observe that Eq.~(\ref{eq:mixrepres.dncb}) can be restated as follows:
\begin{equation}
\left.X'\right|\left(M_1,M^+\right) \sim \mbox{Beta}\left(\alpha_1+M_1,\alpha_2+M^+-M_1\right).
\label{eq:misc.beta.dnc}
\end{equation}
In view of Property~\ref{prope:repr.prop.ncchisq}, $Y'^+$ has a $\chi'^2_{2\alpha^+}\left(\lambda^+\right)$ distribution with $\alpha^+=\alpha_1+\alpha_2$ and $\lambda^+=\lambda_1+\lambda_2$; moreover, by virtue of Property~\ref{prope:mixrepres.ncchisq}, one has that:
\begin{equation}
\left.Y'^+\right|\left(M_1,M^+\right) \stackrel{d}{=} \left.Y'^+\right|M^+ \sim \chi^2_{2\alpha^++2M^+},
\label{eq:misc.chisq.nc}
\end{equation}
where $\stackrel{d}{=}$ stands for ``equal in distribution''. By Property~\ref{prope:char.prop.chisq}, $X'$ and $Y'^+$ are conditionally independent given $\left(M_1,M^+\right)$. Hence, conditionally on $\left(M_1,M^+\right)$, the joint distribution of $\left(X',Y'^+\right)$ factorizes into the marginal distributions of $X'$ and $Y'^+$.

That said, the proof follows by noting that the joint density function of $\left.\left(X',Y'^+\right)\right|M^+$ turns out to factorize into the marginal density functions of $X'|M^+$ and $Y'^+|M^+$. Indeed, under Eq.~(\ref{eq:misc.chisq.nc}) one can obtain:
\begin{eqnarray*}
f_{\left.\left(X',Y'^+\right)\right|M^+}\left(x,y\right) & = & \sum_{i=0}^{M^+}\Pr\left(\left.M_1=i\right|M^+\right) \cdot f_{\left.\left(X',Y'^+\right)\right|\left(M_1,M^+\right)}\left(x,y\right)=\\
& = & f_{\left.Y'^+\right|M^+}\left(y\right) \cdot f_{\left.X'\right|M^+}\left(x\right),
\end{eqnarray*}
where, under Eq.~(\ref{eq:misc.beta.dnc}) and by bearing in mind that $M_1|M^+ \sim \mbox{Binomial}\left(M^+,\frac{\lambda_1}{\lambda^+}\right)$, the density $f_{\left.X'\right|M^+}$ of $X'$ given $M^+$ is of the same form as in Eq.~(\ref{eq:dncb.distr.cond.m}).
\end{myproof}

\begin{myproof}[\textbf{Proposition~\ref{propo:pertrepres.dncb}}]
\label{proof:pertrepres.dncb}
In light of Eq.~(\ref{eq:poch.symb2}) and in the notation of Eq.~(\ref{eq:dens.beta.dnc}), it follows that:
\begin{eqnarray*}
B\left(\alpha_1+j,\alpha_2+k\right) & = & \frac{\Gamma\left(\alpha_1+j\right) \, \Gamma\left(\alpha_2+k\right)}{\Gamma\left(\alpha^++j+k\right)}=\frac{\Gamma\left(\alpha_1\right) \, \left(\alpha_1\right)_j \; \Gamma\left(\alpha_2\right) \, \left(\alpha_2\right)_k}{\Gamma\left(\alpha^+\right) \, \left(\alpha^+\right)_{j+k}}=\\
& = & B\left(\alpha_1,\alpha_2\right) \, \frac{\left(\alpha_1\right)_j \, \left(\alpha_2\right)_k}{\left(\alpha^+\right)_{j+k}}.
\end{eqnarray*}
Hence, Eq.~(\ref{eq:c.ncb}) can be obtained from Eq.~(\ref{eq:dens.beta.dnc}) with simple computations by making use of the latter result and of Eq.~(\ref{eq:perturb.ncb}); indeed we have:
$$f_{X'}\left(x;\alpha_1,\alpha_2,\lambda_1,\lambda_2 \right)=\mbox{Beta}\left(x;\alpha_1,\alpha_2\right) \cdot e^{-\frac{\lambda^+}{2}}  \sum_{j=0}^{+\infty} \sum_{k=0}^{+\infty} \frac{\left(\alpha^+\right)_{j+k}}{\left(\alpha_1\right)_j \left(\alpha_2\right)_k} \frac{\left(\frac{\lambda_1}{2}x\right)^j}{j!} \frac{\left[\frac{\lambda_2}{2}(1-x)\right]^k}{k!}.$$
\end{myproof}

\begin{myproof}[\textbf{Property~\ref{prope:relat.ncb1.ncb2}}]
\label{proof:relat.ncb1.ncb2}
By virtue of Eq.~(\ref{eq:def.nc2.beta}) one has:
$$X'_2=\frac{Y_1}{Y_1+Y'_2}=1-\frac{Y'_2}{Y'_2+Y_1} \; \Leftrightarrow \; \frac{Y'_2}{Y'_2+Y_1}=1-X'_2,$$
where $Y_1 \sim \chi^{2}_{2\alpha_1}$ and $Y'_2 \sim \chi'^{\,2}_{2\alpha_2}\left(\lambda\right)$ independently. In view of Eq.~(\ref{eq:def.nc1.beta}), Eq.~(\ref{eq:relat.ncb1.ncb2}) is established.
\end{myproof}

\begin{myproof}[\textbf{Proposition~\ref{propo:rappr.clc.beta.dnc}}]
\label{proof:rappr.clc.beta.dnc}
Let $Y_r$, $r=1,2$, be independent $\chi^2_{2\alpha_r}$ random variables and $Y^+=Y_1+Y_2 \sim \chi^2_{2\alpha^+}$, with $\alpha^+=\alpha_1+\alpha_2$. In light of Eqs.~(\ref{eq:sumrepres.ncchisq}),~(\ref{eq:def.dnc.beta}) one has: 
\begin{equation}
X'=\frac{Y'_1}{Y'_1+Y'_2}=\frac{Y_1+\sum_{j=1}^{M_1}F_j}{Y^++\sum_{j=1}^{M^+}F_j}=\frac{Y_1}{Y^++\sum_{j=1}^{M^+}F_j}+\frac{\sum_{j=1}^{M_1}F_j}{Y^++\sum_{j=1}^{M^+}F_j}.
\label{eq:rappr.clc.beta.dnc1}
\end{equation}
Observe that the first term on the right-hand side of Eq.~(\ref{eq:rappr.clc.beta.dnc1}) can be restated as:  
$$
\frac{Y_1}{Y^++\sum_{j=1}^{M^+}F_j}=\frac{Y^+}{Y^++\sum_{j=1}^{M^+}F_j} \cdot \frac{Y_1}{Y^+};
$$
similarly, with respect to the second term, we have:
$$
\frac{\sum_{j=1}^{M_1}F_j}{Y^++\sum_{j=1}^{M^+}F_j}=
\frac{\sum_{j=1}^{M^+}F_j}{Y^++\sum_{j=1}^{M^+} F_j} \cdot \frac{\sum_{j=1}^{M_1}F_j}{\sum_{j=1}^{M^+} F_j},
$$
which is meaningful provided we set that:
$$
\frac{\sum_{j=1}^{M_1}F_j}{\sum_{j=1}^{M^+} F_j}=0 \qquad \mbox{if }M_r=0, \quad \forall r=1,2.
$$
Finally, by setting:
\begin{equation}
X'_2=\frac{Y^+}{Y^++\sum_{j=1}^{M^+}F_j}, \qquad X=\frac{Y_1}{Y^+}, \qquad X'_{pnc}=\frac{\sum_{j=1}^{M_1}F_j}{\sum_{j=1}^{M^+} F_j},
\label{eq:rappr.clc.beta.dnc1.primopezzo.nomi}
\end{equation}
the decomposition in Eq.~(\ref{eq:rappr.clc.beta.dnc}) is established.

Now consider the random vector $\left(X,X'_2,X'_{pnc}\right)$. In light of Eq.~(\ref{eq:rappr.clc.beta.dnc1.primopezzo.nomi}), the marginal random vector $\left(X'_2,X'_{pnc}\right)$ is a function of $\left(Y^+,M_1,M_2,\left\{F_j\right\}\right)$; moreover, the latter is independent of $X$: in fact, $Y^+$ is independent of $X$ by virtue of Property~\ref{prope:char.prop.chisq}. Finally, $X$ and $\left(X'_2,X'_{pnc}\right)$ are mutually independent and, in view of Eq.~(\ref{eq:def.beta}), $X \sim \mbox{\normalfont{Beta}}\left(\alpha_1,\alpha_2\right)$: result \textit{i)} is thus proved.

In order to prove result \textit{ii)}, observe first that:
$$
\left.X'_2\right|\left(M_1,M_2\right)=\left.\frac{Y^+}{Y^++\sum_{j=1}^{M^+}F_j}\right|\left(M_1,M_2\right) \sim \mbox{Beta}\left(\alpha^+,M^+\right)
$$
and:
$$
\left.X'_{pnc}\right|\left(M_1,M_2\right)=\left.\frac{\sum_{j=1}^{M_1}F_j}{\sum_{j=1}^{M^+} F_j}\right|\left(M_1,M_2\right) \sim \mbox{Beta}\left(M_1,M_2\right);
$$
moreover, $X'_2$ and $X'_{pnc}$ are conditionally independent given $\left(M_1,M_2\right)$. That said, the proof of \textit{ii)} follows by noting that the joint density function of $\left(X'_2,X'_{pnc}\right) \, |M^+$ turns out to factorize into the marginal distributions of $\left.X'_2\right|M^+$ and $\left.X'_{pnc}\right|M^+$. Indeed, by bearing in mind that $\left.M_1\right|M^+ \sim \mbox{Binomial}\left(M^+,\frac{\lambda_1}{\lambda^+}\right)$, one can obtain:
\begin{eqnarray*}
\lefteqn{f_{\left.\left(X'_2,X'_{pnc}\right)\right|M^+}\left(x_2,x\right)=}\\
& = & \sum_{i=0}^{M^+} f_{\left.\left(X'_2,X'_{pnc}\right)\right|\left(M_1,M_2\right)}\left(x_2,x\right) \cdot \Pr\left(\left.M_1=i\right|M^+\right)=\\
& = & \sum_{i=0}^{M^+} f_{\left.X'_2\right|\left(M_1,M_2\right)}\left(x_2\right) \cdot f_{\left.X'_{pnc}\right|\left(M_1,M_2\right)}\left(x\right) \cdot \Pr\left(\left.M_1=i\right|M^+\right)=\\
& = & \sum_{i=0}^{M^+} \mbox{Beta}\left(x_2;\alpha^+,M^+\right) \cdot \mbox{Beta}\left(x;i,M^+-i\right) \cdot \Pr\left(\left.M_1=i\right|M^+\right)=\\
& = & \mbox{Beta}\left(x_2;\alpha^+,M^+\right) \cdot \sum_{i=0}^{M^+} \Pr\left(\left.M_1=i\right|M^+\right) \cdot \mbox{Beta}\left(x;i,M^+-i\right)=\\
& = & f_{\left.X'_2\right|M^+}\left(x_2\right) \cdot f_{\left.X'_{pnc}\right|M^+}\left(x\right).
\end{eqnarray*}
Finally, result \textit{iii)} follows from Eq.~(\ref{eq:mixrepres.dncb}) and Eq.~(\ref{eq:mixrepres.ncb2}).
\end{myproof}

\begin{myproof}[\textbf{Property~\ref{prope:relat.dncb1}}]
\label{proof:relat.dncb1}
Let $Y'_r$, $r=1,2$, be independent $\chi'^{\,2}_{2\alpha_r}\left(\lambda_r\right)$ random variables. The proof follows from Eq.~(\ref{eq:def.dnc.beta}) by noting that:
$$X'=\frac{Y'_1}{Y'_1+Y'_2}=1-\frac{Y'_2}{Y'_2+Y'_1} \; \Leftrightarrow \; 1-X'=\frac{Y'_2}{Y'_2+Y'_1} \sim \mbox{\normalfont{B}}''\left(\alpha_2,\alpha_1,\lambda_2,\lambda_1\right).$$
\end{myproof}

\begin{myproof}[\textbf{Proposition~\ref{propo:dens.beta.dnc.lims}}]
\label{proof:dens.beta.dnc.lims}
By taking $\alpha_1=\alpha_2=1$ in Eq.~(\ref{eq:dens.beta.dnc}), one has:
\begin{equation}
f_{X'}\left(x;1,1,\lambda_1,\lambda_2\right)=\sum_{j=0}^{+\infty} \sum_{k=0}^{+\infty} \frac{e^{-\frac{\lambda_1}{2}} \left(\frac{\lambda_1}{2}\right)^j}{j!} \frac{e^{-\frac{\lambda_2}{2}} \left(\frac{\lambda_2}{2}\right)^k}{k!} \frac{x^j \left(1-x\right)^k}{B\left(1+j,1+k\right)},\quad 0<x<1.
\label{eq:dens.beta.dnc.11}
\end{equation}
Hence, by taking the limit of both sides of Eq.~(\ref{eq:dens.beta.dnc.11}) as $x$ tends to $0^+$, the outcome is:
\begin{eqnarray*}
\lefteqn{\lim_{x \rightarrow 0^+} f_{X'}\left(x;1,1,\lambda_1,\lambda_2\right)=}\\
& = & e^{-\frac{\lambda_1}{2}} \, \sum_{k=0}^{+\infty} \frac{e^{-\frac{\lambda_2}{2}} \left(\frac{\lambda_2}{2}\right)^k}{k!} \frac{1}{B\left(1,1+k\right)}=e^{-\frac{\lambda_1}{2}} \, \sum_{k=0}^{+\infty} (k+1) \, \frac{e^{-\frac{\lambda_2}{2}} \left(\frac{\lambda_2}{2}\right)^k}{k!}=\\
& = & e^{-\frac{\lambda_1}{2}} \left[\sum_{k=0}^{+\infty} k \, \frac{e^{-\frac{\lambda_2}{2}} \left(\frac{\lambda_2}{2}\right)^k}{k!}+\sum_{k=0}^{+\infty} \frac{e^{-\frac{\lambda_2}{2}} \left(\frac{\lambda_2}{2}\right)^k}{k!}\right]=e^{-\frac{\lambda_1}{2}} \left(\frac{\lambda_2}{2} +1\right),
\end{eqnarray*}
that is Eq.~(\ref{eq:dens.beta.dnc.11.0}).

Similarly the limit at $1$ of Eq.~(\ref{eq:dens.beta.dnc.11}) turns out to be:
\begin{eqnarray*}
\lefteqn{\lim_{x \rightarrow 1^-} f_{X'}\left(x;1,1,\lambda_1,\lambda_2\right)=}\\
& = & e^{-\frac{\lambda_2}{2}} \, \sum_{j=0}^{+\infty} \frac{e^{-\frac{\lambda_1}{2}} \left(\frac{\lambda_1}{2}\right)^j}{j!} \frac{1}{B\left(1+j,1\right)}=e^{-\frac{\lambda_2}{2}} \, \sum_{j=0}^{+\infty} (j+1) \, \frac{e^{-\frac{\lambda_1}{2}} \left(\frac{\lambda_1}{2}\right)^j}{j!}=\\
& = & e^{-\frac{\lambda_2}{2}} \left[\sum_{j=0}^{+\infty} j \, \frac{e^{-\frac{\lambda_1}{2}} \left(\frac{\lambda_1}{2}\right)^j}{j!}+\sum_{j=0}^{+\infty} \frac{e^{-\frac{\lambda_1}{2}} \left(\frac{\lambda_1}{2}\right)^j}{j!}\right]=e^{-\frac{\lambda_2}{2}} \left(\frac{\lambda_1}{2} +1\right),
\end{eqnarray*}
that is Eq.~(\ref{eq:dens.beta.dnc.11.1}).
\end{myproof}

\begin{myproof}[\textbf{Proposition~\ref{propo:dnc.beta.pat.approx}}]
\label{proof:dnc.beta.pat.approx}
By virtue of Eq.~(\ref{eq:def.dnc.beta}) and Property~\ref{prope:pat.approx.ncchisq}, we have:
\begin{equation}
X'=\frac{Y'_1}{Y'_1+Y'_2} \; \stackrel{d}{\approx} \; X'_P=\frac{\rho_1 Y_1}{\rho_1 Y_1+\rho_2 Y_2},
\label{eq:appdncbetapat1}
\end{equation}
where $\rho_r=\frac{2\left(\alpha_r+\lambda_r\right)}{2\alpha_r+\lambda_r}$, $r=1,2$ and $Y_r$ are independent $\chi^2_{\nu_r}$ random variables with $\nu_r=\frac{\left(2\alpha_r+\lambda_r\right)^2}{2\left(\alpha_r+\lambda_r\right)}$; moreover, Eq.~(\ref{eq:appdncbetapat1}) is tantamount to:
\begin{equation}
X' \stackrel{d}{\approx}\frac{1}{\frac{\rho_1 Y_1+\rho_2 Y_2}{\rho_1 \, Y_1}}=\frac{1}{1+\frac{\rho_2}{\rho_1}\frac{Y_2}{Y_1}}.
\label{eq:appdncbetapat2}
\end{equation}
Now let $X$ have a $\mbox{Beta}\left(\frac{\nu_1}{2},\frac{\nu_2}{2}\right)$ distribution. Therefore, we have:
\begin{equation}
X=\frac{Y_1}{Y_1+Y_2} \; \Leftrightarrow \; \frac{Y_2}{Y_1}=\frac{1-X}{X};
\label{eq:appdncbetapat3}
\end{equation}
in light of Eq.~(\ref{eq:appdncbetapat3}), Eq.~(\ref{eq:appdncbetapat2}) can be thus restated as follows:
\begin{equation}
X' \stackrel{d}{\approx} \frac{1}{1+\frac{\rho_2}{\rho_1}\frac{1-X}{X}}=\frac{\rho_1 \, X}{\rho_1 \, X+\rho_2  \left(1-X\right)}=f\left(X\right)=X'_P.
\label{eq:appdncbetapat4}
\end{equation}
By noting that $
X=f^{-1}\left(X'_P\right)=\frac{\rho_2 \, X'_P}{\rho_1 \left(1-X'_P\right)+\rho_2 \, X'_P}$, $\frac{dx}{dx'}=\frac{\rho_1 \, \rho_2}{\left[\rho_1  \left(1-x'\right)+\rho_2 \, x'\right]^2}$, by taking $\beta_r=\frac{\nu_r}{2}$, $r=1,2$ and by denoting the densities of $X$ and $X'_P$ with $f_{X}$ and $f_{X'_P}$, respectively, the proof is straightforward once we observe that:
$$
f_{X'_P}\left(x'\right)=f_X\left(\frac{\rho_2 \, x'}{\rho_1 \left(1-x'\right)+\rho_2 \, x'}\right) \frac{dx}{dx'}=\frac{\left(\frac{\rho_2}{\rho_1}\right)^{\beta_1}}{B\left(\beta_1,\beta_2\right)}\frac{x'^{\beta_1-1} \, \left(1-x'\right)^{\beta_2-1}}{\left[1-\left(1-\frac{\rho_2}{\rho_1}\right) \, x'\right]^{\beta_1+\beta_2}},
$$
with $x' \in (0,1)$.
\end{myproof}

\begin{myproof}[\textbf{Proposition~\ref{propo:momr.beta.dnc.rapp}}]
\label{proof:momr.beta.dnc.rapp}
In the notation of Eq.~(\ref{eq:def.dnc.beta}) and by virtue of Proposition~\ref{propo:dncb.cond.ind}, one has: $\mathbb{E}\left[\left.\left(X'\right)^r\right|M^+\right]=\frac{\mathbb{E}\left[\left.\left(Y'_1\right)^r\right|M^+\right]}{\mathbb{E}\left[\left.\left(Y'^+\right)^r\right|M^+\right]}$; moreover, in view of the general formula for the moments of the gamma distribution \cite{JohKotBal94}, one has:
\begin{eqnarray*}
\lefteqn{\mathbb{E}\left[\left.\left(Y'_1\right)^r\right|M^+\right]=\mathbb{E}\left\{\left.\mathbb{E}\left[\left.\left(Y'_1\right)^r \right|M_1,M^+\right] \right|M^+\right\}=\mathbb{E}\left\{\left.\mathbb{E}\left[\left.\left(Y'_1\right)^r \right|M_1\right]\right|M^+\right\}=} \\
& = & \mathbb{E}\left[\left.2^r \, \left(\alpha_1+M_1\right)_r\right|M^+\right]=2^r \sum_{i=0}^{M^+}\left(\alpha_1+i\right)_r {M^+ \choose i} \left(\frac{\lambda_1}{\lambda^+}\right)^i \left(1-\frac{\lambda_1}{\lambda^+}\right)^{M^+-i}
\end{eqnarray*}
and $\mathbb{E}\left[\left.\left(Y'^+\right)^r\right|M^+\right]=2^r \left(\alpha^++M^+\right)_r$. Therefore:
\begin{equation}
\mathbb{E}\left[\left.\left(X'\right)^r\right|M^+\right]=\frac{1}{\left(\alpha^++M^+\right)_r} \sum_{i=0}^{M^+}\left(\alpha_1+i\right)_r {M^+ \choose i} \left(\frac{\lambda_1}{\lambda^+}\right)^i \left(1-\frac{\lambda_1}{\lambda^+}\right)^{M^+-i}.
\label{eq:momr.beta.dnc.dim6b}
\end{equation}
By letting $L \sim \mbox{Binomial}\left(M^+,\lambda_1/\lambda^+\right)$, Eq.~(\ref{eq:momr.beta.dnc.dim6b}) can be restated as follows:
\begin{equation}
\mathbb{E}\left[\left.\left(X'\right)^r\right|M^+\right]=\frac{\mathbb{E}\left[\left(\alpha_1+L\right)_r\right]}{\left(\alpha^++M^+\right)_r}.
\label{eq:momr.beta.dnc.dim7}
\end{equation}
In this regard, by replacing $a$ and $b$ with, respectively, $\alpha_1$ and $L$ in Eq.~(\ref{eq:expans.poch.symb.binom}), one has:
\begin{equation}
\mathbb{E}\left[\left(\alpha_1+L\right)_r\right]=\sum_{i=0}^{r}\frac{1}{i!}\left[\frac{d^i}{d\alpha_1^i}\left(\alpha_1\right)_r \right]\mathbb{E}\left(L^i\right),
\label{eq:momr.beta.dnc.dim8}
\end{equation}
where, $\forall i \in \mathbb{N} \cup \{0\}$:
\begin{equation}
\mbox{E}\left(L^i\right)=\sum_{k=0}^{i}\mathcal{S}\left(i,k\right) \frac{M^+!}{\left(M^+-k\right)!} \left(\frac{\lambda_1}{\lambda^+}\right)^k
\label{eq:momr.beta.dnc.dim9}
\end{equation}
\cite{JohKemKot05}, $\mathcal{S}\left(i,k\right)$ being a Stirling number of the second kind. By making use of $\frac{M^+!}{\left(M^+-k\right)!}=\sum_{j=0}^{k} s\left(k,j\right) \cdot \left(M^+\right)^j$ \cite{JohKemKot05}, $s\left(k,j\right)$ being a Stirling number of the first kind, Eq.~(\ref{eq:momr.beta.dnc.dim9}) can be rewritten as follows:
\begin{equation}
\mathbb{E}\left(L^i\right)=\sum_{k=0}^i \mathcal{S}\left(i,k\right) \left[\sum_{j=0}^{k} s \left(k,j\right) \cdot \left(M^+\right)^j\right] \, \left(\frac{\lambda_1}{\lambda^+}\right)^k.
\label{eq:momr.beta.dnc.dim11}
\end{equation}
By noting that $\sum_{k=0}^{i}\sum_{j=0}^{k}a_{kj}=\sum_{j=0}^{i}\sum_{k=j}^{i}a_{kj}$ and by letting $\theta_1=\lambda_1/\lambda^+$, Eq.~(\ref{eq:momr.beta.dnc.dim11}) turns out to be tantamount to:
\begin{equation}
\mathbb{E}\left(L^i\right)=\sum_{j=0}^{i}\left[\sum_{k=j}^{i}\mathcal{S}\left(i,k\right) s\left(k,j\right)\theta_1^k\right]\left(M^+\right)^j;
\label{eq:momr.beta.dnc.dim12}
\end{equation}
under Eq.~(\ref{eq:momr.beta.dnc.dim12}), Eq.~(\ref{eq:momr.beta.dnc.dim8}) can be written accordingly in the form of:
\begin{equation}
\mathbb{E}\left[\left(\alpha_1+L\right)_r\right]=\sum_{i=0}^{r}\frac{1}{i!}\left[\frac{d^i}{d \alpha_1^i} \left(\alpha_1\right)_r \right] \cdot	\sum_{j=0}^{i}\left[\sum_{k=j}^{i}\mathcal{S}\left(i,k\right) s\left(k,j\right)\theta_1^k\right]\left(M^+\right)^j
\label{eq:momr.beta.dnc.dim13}
\end{equation}
and finally, under Eq.~(\ref{eq:momr.beta.dnc.dim13}), Eq.~(\ref{eq:momr.beta.dnc.dim7}) can be stated in the following form:
$$
\mathbb{E}\left[\left(X'\right)^r|M^+\right]=\frac{\sum_{i=0}^{r}\frac{1}{i!}\left[\frac{d^i}{d \alpha_1^i} \left(\alpha_1\right)_r \right] \cdot	\sum_{j=0}^{i}\left[\sum_{k=j}^{i}\mathcal{S}\left(i,k\right) s\left(k,j\right)\theta_1^k\right]\left(M^+\right)^j}{\left(\alpha^++M^+\right)_r}.
$$
Therefore, by virtue of the law of iterated expectations, since $M^+ \sim \mbox{Poisson}\left(\lambda^+/2\right)$, the $r$-th moment about zero of the doubly non-central beta distribution turns out to have the following expression:
\begin{equation}
\mathbb{E}\left[\left(X'\right)^r\right]=e^{-\frac{\lambda^+}{2}}\sum_{i=0}^{r}\frac{1}{i!}\left[\frac{d^i}{d \alpha_1^i} \left(\alpha_1\right)_r \right] \cdot	\sum_{j=0}^{i}\left[\sum_{k=j}^{i}\mathcal{S}\left(i,k\right) s\left(k,j\right)\theta_1^k\right] \cdot \sum_{l=0}^{+\infty} \frac{l^j}{\left(\alpha^++l\right)_r} \frac{ \left(\frac{\lambda^+}{2}\right)^l}{l!}.
\label{eq:momr.beta.dnc.dim15}
\end{equation}
Observe that, in view of Eq.~(\ref{eq:poch.symb.sum}), the following holds:
$$
\left(\alpha^+\right)_l \left(\alpha^++l\right)_r=\left(\alpha^+\right)_r \left(\alpha^++r\right)_l \Leftrightarrow \frac{1}{\left(\alpha^++l\right)_r}=\frac{\left(\alpha^+\right)_l}{\left(\alpha^++r\right)_l} \frac{1}{\left(\alpha^+\right)_r}
$$
and Eq.~(\ref{eq:momr.beta.dnc.dim15}) can be rewritten accordingly as follows:
$$
\mathbb{E}\left[\left(X'\right)^r\right]=\frac{e^{-\frac{\lambda^+}{2}}}{\left(\alpha^+\right)_r}\sum_{i=0}^{r}\frac{1}{i!}\left[\frac{d^i}{d \alpha_1^i} \left(\alpha_1\right)_r \right] \cdot	\sum_{j=0}^{i}\left[\sum_{k=j}^{i}S\left(i,k\right) s\left(k,j\right)\theta_1^k\right] \cdot \sum_{l=0}^{+\infty} \frac{l^j \, \left(\alpha^+\right)_l}{\left(\alpha^++r\right)_l} \frac{ \left(\frac{\lambda^+}{2}\right)^l}{l!}.
$$

Now let $M^{+}_*$ be a random variable on the non-negative integers such that:
$$
\Pr\left(M^{+}_*=l\right)=\frac{\frac{\left(\alpha^+\right)_l}{\left(\alpha^++r\right)_l} \frac{\left(\frac{\lambda^+}{2}\right)^l}{l!}}{_1F_1\left(\alpha^+;\alpha^++r;\frac{\lambda^+}{2}\right)}, \quad \forall l \in \mathbb{N} \cup \{0\},
$$
so that, for every $j \in \mathbb{N} \cup \{0\}$:
$$\mathbb{E}\left(M^{+}_*\right)^j=\frac{1}{_1F_1\left(\alpha^+;\alpha^++r;\frac{\lambda^+}{2}\right)}\sum_{l=0}^{+\infty} \frac{l^j \, \left(\alpha^+\right)_l}{\left(\alpha^++r\right)_l} \frac{ \left(\frac{\lambda^+}{2}\right)^l}{l!}$$
and:
\begin{eqnarray}
\lefteqn{\mathbb{E}\left[\left(X'\right)^r\right]=\frac{e^{-\frac{\lambda^+}{2}}}{\left(\alpha^+\right)_r} \, _1F_1\left(\alpha^+;\alpha^++r;\frac{\lambda^+}{2}\right) \cdot}\nonumber\\
& \cdot & \sum_{i=0}^{r}\frac{1}{i!}\left[\frac{d^i}{d \alpha_1^i} \left(\alpha_1\right)_r \right] \cdot	\sum_{j=0}^{i}\left[\sum_{k=j}^{i}S\left(i,k\right) s\left(k,j\right)\theta_1^k\right] \cdot \mathbb{E}\left(M^{+}_*\right)^j.
\label{eq:momr.beta.dnc.dim19}
\end{eqnarray}
Note that the generating function of the descending factorial moments of $M^{+}_*$ has the following expression:
$$
\mathbb{E}\left[\left(1+t\right)^{M^{+*}}\right]=\frac{_1F_1\left(\alpha^+;\alpha^++r;\frac{\left(1+t\right)\lambda^+}{2}\right)}{_1F_1\left(\alpha^+;\alpha^++r;\frac{\lambda^+}{2}\right)}, \quad t \in \mathbb{R}
$$
and its derivative of order $m \in \mathbb{N}$ is:
\begin{equation}
\frac{d^m}{dt^m}\mathbb{E}\left[\left(1+t\right)^{M^{+*}}\right]=\frac{\left(\frac{\lambda^+}{2}\right)^m \left(\alpha^+\right)_m}{\left(\alpha^++r\right)_m} \frac{_1F_1\left(\alpha^++m;\alpha^++r+m;\frac{\left(1+t\right)\lambda^+}{2}\right)}{_1F_1\left(\alpha^+;\alpha^++r;\frac{\lambda^+}{2}\right)};
\label{eq:momr.beta.dnc.dim21}
\end{equation}
hence, by taking $t=0$ in Eq.~(\ref{eq:momr.beta.dnc.dim21}), it follows that the $m$-th descending factorial moment of $M^{+}_*$ turns out to be:
\begin{equation}
\mathbb{E}\left(M^{+}_*\right)_{[m]}=\frac{\left(\frac{\lambda^+}{2}\right)^m \, \left(\alpha^+\right)_m}{\left(\alpha^++r\right)_m}\frac{_1F_1\left(\alpha^++m;\alpha^++r+m;\frac{\lambda^+}{2}\right)}{_1F_1\left(\alpha^+;\alpha^++r;\frac{\lambda^+}{2}\right)}.
\label{eq:momr.beta.dnc.dim22}
\end{equation}
By bearing in mind that $\mathbb{E}\left(M^{+}_*\right)^j=\sum_{m=0}^{j}\mathcal{S}\left(j,m\right)\mathbb{E}\left(M^{+}_*\right)_{[m]}$ \cite{JohKemKot05} and in light of Eq.~(\ref{eq:momr.beta.dnc.dim22}), Eq.~(\ref{eq:momr.beta.dnc.dim19}) can be rewritten as follows:
\begin{eqnarray*}
\lefteqn{\mbox{E}\left[\left(X'\right)^r\right]=}\\
& = & \frac{e^{-\frac{\lambda^+}{2}}}{\left(\alpha^+\right)_r} \sum_{i=0}^{r}\frac{1}{i!}\left[\frac{d^i}{d \alpha_1^i} \left(\alpha_1\right)_r \right] \cdot	\sum_{j=0}^{i}\left[\sum_{k=j}^{i}\mathcal{S}\left(i,k\right) s\left(k,j\right)\theta_1^k\right] \cdot \\
& \cdot &  \sum_{m=0}^{j} \mathcal{S}\left(j,m\right) \frac{\left(\frac{\lambda^+}{2}\right)^m \, \left(\alpha^+\right)_m}{\left(\alpha^++r\right)_m} \, _1F_1\left(\alpha^++m;\alpha^++r+m;\frac{\lambda^+}{2}\right).
\end{eqnarray*}
Furthermore, by virtue of the following properties of the Stirling numbers of the first and the second kinds \cite{JohKemKot05}:
$$s\left(a,0\right)=s\left(0,a\right)=\mathcal{S}\left(a,0\right)=\mathcal{S}\left(0,a\right)=0, \quad \forall a>0,$$
$$\sum_{j=m}^{n}\mathcal{S}\left(n,j\right)s\left(j,m\right)=\sum_{j=m}^{n}s\left(n,j\right)\mathcal{S}\left(j,m\right)=\left\{\begin{array}{ll} 1 & \mbox{if }m=n\\ 0 & \mbox{otherwise} \end{array} \right.,$$one has:
\begin{eqnarray*}
\lefteqn{\mbox{E}\left[\left(X'\right)^r\right]=\frac{e^{-\frac{\lambda^+}{2}}}{\left(\alpha^+\right)_r} \sum_{i=0}^{r} \left[\sum_{j=i}^{r}\frac{1}{j!}\frac{d^j}{d\alpha_1^j}\left(\alpha_1\right)_r \mathcal{S}\left(j,i\right)\right] \cdot}\\
& \cdot & s\left(i,i\right)\mathcal{S}\left(i,i\right)\frac{\left(\alpha^+\right)_i \left(\theta_1 \frac{\lambda^+}{2}\right)^i }{\left(\alpha^++r\right)_i} \, _1F_1\left(\alpha^++i;\alpha^++r+i;\frac{\lambda^+}{2}\right);
\end{eqnarray*}
finally, in view of Eq.~(\ref{eq:ident.mom}) and by bearing in mind that $s\left(a,a\right)=\mathcal{S}\left(a,a\right)=1$, for every $a \geq 0$, Eq.~(\ref{eq:momr.beta.dnc}) is established.
\end{myproof}

\begin{myproof}[\textbf{Proposition~\ref{propo:relat.means.nc.beta}}]
\label{proof:relat.means.nc.beta}
Observe that Eq.~(\ref{eq:mom1.beta.dnc}) can be rewritten as follows:
\begin{eqnarray*}
\lefteqn{\mathbb{E}\left(X'\right)=}\\
& = & \frac{\alpha_1}{\alpha^+}\, e^{-\frac{\lambda^+}{2}} \, _1F_1\left(\alpha^+;\alpha^++1;\frac{\lambda^+}{2}\right)+e^{-\frac{\lambda^+}{2}} \frac{\frac{\lambda_1}{2}}{\alpha^++1} \, _1F_1\left(\alpha^++1;\alpha^++2;\frac{\lambda^+}{2}\right)=\\
& = & \frac{\lambda_1}{\lambda^+} \left[\frac{\alpha_1}{\alpha^+}\, e^{-\frac{\lambda^+}{2}} \, _1F_1\left(\alpha^+;\alpha^++1;\frac{\lambda^+}{2}\right)+e^{-\frac{\lambda^+}{2}} \frac{\frac{\lambda^+}{2}}{\alpha^++1} \, _1F_1\left(\alpha^++1;\alpha^++2;\frac{\lambda^+}{2}\right)\right]+\\
& + & \frac{\lambda_2}{\lambda^+}\left[\frac{\alpha_1}{\alpha^+}\, e^{-\frac{\lambda^+}{2}} \, _1F_1\left(\alpha^+;\alpha^++1;\frac{\lambda^+}{2}\right)\right];
\end{eqnarray*}
Eq.~(\ref{eq:relat.means.nc.beta}) is thus established.
\end{myproof}

\begin{myproof}[\textbf{Proposition~\ref{propo:alter.exp.mean.dnc.beta}}]
\label{proof:alter.exp.mean.dnc.beta}
In the notation of Proposition~\ref{propo:rappr.clc.beta.dnc}, one has:
\begin{equation}
\mathbb{E}\left(X'\right)=\mathbb{E}\left[X'_2 \, X +\left(1-X'_2\right) \,  X'_{pnc}\right];
\label{eq:alter.exp.mean.dnc.beta.dim1}
\end{equation}
in view of result \textit{i)} of the aforementioned Proposition, Eq.~(\ref{eq:alter.exp.mean.dnc.beta.dim1}) can be rewritten as $\mathbb{E}\left(X'\right)=\frac{\alpha_1}{\alpha^+} \, \mathbb{E}\left(X'_2\right)+ \mathbb{E}\left[\left(1-X'_2\right) \,  X'_{pnc}\right]$. As $X'_2 \sim \mbox{\normalfont{B}}'_2\left(\alpha^+,0,\lambda^+\right)$ in light of Eq.~(\ref{eq:momr.beta.nc2}), one has: $\mathbb{E}\left(X'_2\right)=e^{-\frac{\lambda^+}{2}} \, _1F_1\left(\alpha^+;\alpha^++1;\frac{\lambda^+}{2}\right)$. By virtue of the law of iterated expectations and in view of result \textit{ii)} of Proposition~\ref{propo:rappr.clc.beta.dnc}, the following holds true:
\begin{eqnarray}
\lefteqn{\mathbb{E}\left[\left(1-X'_2\right) \,  X'_{pnc}\right]=}\nonumber\\
& = & \mathbb{E}_{M^+}\left\{\mathbb{E}\left[\left.\left(1-X'_2\right) \,  X'_{pnc}\right|M^+\right]\right\}=\mathbb{E}_{M^+}\left\{\mathbb{E}\left[\left.\left(1-X'_2\right)\right|M^+\right] \,  \mathbb{E}\left(\left.X'_{pnc}\right|M^+\right)\right\},\nonumber\\
\label{eq:alter.exp.mean.dnc.beta.dim3}
\end{eqnarray}
where:
\begin{eqnarray*}
\mathbb{E}\left(\left.X'_{pnc}\right|M^+\right) & = & \sum_{i=0}^{M^+} {M^+ \choose i} \left(\frac{\lambda_1}{\lambda^+}\right)^i \left(1-\frac{\lambda_1}{\lambda^+}\right)^{M^+-i} \,  \mathbb{E}\left[\mbox{Beta}\left(x;i,M^+-i\right)\right]=\\
& = & \sum_{i=0}^{M^+} {M^+ \choose i} \left(\frac{\lambda_1}{\lambda^+}\right)^i \left(1-\frac{\lambda_1}{\lambda^+}\right)^{M^+-i} \,  \frac{i}{M^+}=\\
& = & \frac{1}{M^+} \, \sum_{i=0}^{M^+} i \, {M^+ \choose i} \left(\frac{\lambda_1}{\lambda^+}\right)^i \left(1-\frac{\lambda_1}{\lambda^+}\right)^{M^+-i}=\\
& = & \frac{1}{M^+} \cdot M^+ \, \frac{\lambda_1}{\lambda^+}=\frac{\lambda_1}{\lambda^+}.
\end{eqnarray*}
Therefore, Eq.~(\ref{eq:alter.exp.mean.dnc.beta.dim3}) can be restated as follows:
\begin{eqnarray*}
\lefteqn{\mathbb{E}\left[\left(1-X'_2\right) \,  X'_{pnc}\right]=\frac{\lambda_1}{\lambda^+} \, \mathbb{E}_{M^+}\left\{\mathbb{E}\left[\left.\left(1-X'_2\right)\right|M^+\right] \right\}=}\\
& = & \frac{\lambda_1}{\lambda^+} \, \mathbb{E}\left(1-X'_2\right)=\frac{\lambda_1}{\lambda^+} \left[1- \mathbb{E}\left(X'_2\right)\right]=\frac{\lambda_1}{\lambda^+} \left[1- e^{-\frac{\lambda^+}{2}} \, _1F_1\left(\alpha^+;\alpha^++1;\frac{\lambda^+}{2}\right)\right]
\end{eqnarray*}
and Eq.~(\ref{eq:alter.exp.mean.dnc.beta}) is established.
\end{myproof}

\begin{myproof}[\textbf{Proposition~\ref{propo:exp.mean.var.dnc11.beta}}]
\label{proof:exp.mean.var.dnc11.beta}
By taking $\alpha_1=\alpha_2=1$ in Eq.~(\ref{eq:alter.exp.mean.dnc.beta}), we have:
$$
\mathbb{E}\left(X'\right)=\frac{\lambda_1}{\lambda^+}+\left(\frac{1}{2}-\frac{\lambda_1}{\lambda^+}\right) \, e^{-\frac{\lambda^+}{2}}  \, _1F_1\left(2;3;\frac{\lambda^+}{2}\right).
$$
Observe that by taking $a=2$ and $z=\frac{\lambda^+}{2}$ in Eq.~(\ref{eq:form2.1f1}), one has:
\begin{eqnarray*}
_1F_1\left(2;3;\frac{\lambda^+}{2}\right) & =  & 2\left(-\frac{\lambda^+}{2}\right)^{-2}\left[\Gamma\left(2\right)-\Gamma\left(2,-\frac{\lambda^+}{2}\right)\right]=\frac{8}{\left(\lambda^+\right)^2}\left[1-\int_{-\frac{\lambda^+}{2}}^{+\infty}t \, e^{-t} \, dt\right]=\\
& = & \frac{8}{\left(\lambda^+\right)^2}\left[1+\left(\frac{\lambda^+}{2}-1\right)e^{\frac{\lambda^+}{2}}\right];
\end{eqnarray*}
therefore:
\begin{eqnarray*}
\mathbb{E}\left(X'\right) & = & \frac{\lambda_1}{\lambda^+}+\frac{8}{\left(\lambda^+\right)^2}\left(\frac{1}{2}-\frac{\lambda_1}{\lambda^+}\right) \, \left(e^{-\frac{\lambda^+}{2}}+\frac{\lambda^+}{2}-1\right)=\\
& = & \frac{\lambda_1}{\lambda^+} -\frac{4 \, \lambda_1}{\left(\lambda^+\right)^2}+\frac{2}{\lambda^+}+\frac{8 \, \lambda_1}{\left(\lambda^+\right)^3}-\frac{4}{\left(\lambda^+\right)^2}+\frac{4}{\left(\lambda^+\right)^2} \, e^{-\frac{\lambda^+}{2}}-\frac{8 \, \lambda_1}{\left(\lambda^+\right)^3} \, e^{-\frac{\lambda^+}{2}}.
\end{eqnarray*}
Eq.~(\ref{eq:exp.mean.dnc11.beta}) can be thus obtained with simple computations by noting that:
$$
\frac{\lambda_1}{\lambda^+}=\frac{1}{2}+\frac{\lambda_1-\lambda_2}{2 \, \lambda^+}, \qquad -\frac{4 \, \lambda_1}{\left(\lambda^+\right)^2}+\frac{2}{\lambda^+}=-\frac{2 \left(\lambda_1-\lambda_2\right)}{\left(\lambda^+\right)^2},
$$
$$
\frac{8 \, \lambda_1}{\left(\lambda^+\right)^3}-\frac{4}{\left(\lambda^+\right)^2}=\frac{4 \left(\lambda_1-\lambda_2\right)}{\left(\lambda^+\right)^3} , \qquad \frac{4}{\left(\lambda^+\right)^2} \, e^{-\frac{\lambda^+}{2}}-\frac{8 \, \lambda_1}{\left(\lambda^+\right)^3} \, e^{-\frac{\lambda^+}{2}}=-\frac{4 \left(\lambda_1-\lambda_2\right)}{\left(\lambda^+\right)^3}\, e^{-\frac{\lambda^+}{2}}.
$$

We now come to the proof of Eq.~(\ref{eq:exp.var.dnc11.beta}). When $\alpha_1=\alpha_2=1$, Eq.~(\ref{eq:mom2.beta.dnc}) can be rewritten as follows:
\begin{eqnarray*}
\lefteqn{\mathbb{E}\left[\left(X'\right)^2\right]=}\\
& = & \frac{1}{3} \, e^{-\frac{\lambda^+}{2}} \, _1F_1\left(2;4;\frac{\lambda^+}{2}\right)+\frac{\lambda_1}{6} \, e^{-\frac{\lambda^+}{2}} \, _1F_1\left(3;5;\frac{\lambda^+}{2}\right)+\frac{\lambda_1^2}{80} \, e^{-\frac{\lambda^+}{2}} \, _1F_1\left(4;6;\frac{\lambda^+}{2}\right).
\end{eqnarray*}
Observe that by taking $a=2, 3, 4$ and $z=\frac{\lambda^+}{2}$ in the following formula (\href{http://functions.wolfram.com/HypergeometricFunctions/Hypergeometric1F1/03/01/02/}{link}):
\begin{eqnarray}
\lefteqn{_1F_1\left(a;a+2;z\right)=\frac{\left(-z\right)^{-a}}{z} \cdot}\nonumber\\
& \cdot &  \left\{\Gamma\left(a\right) \, a^3+\left(za+a+z\right) \, \Gamma\left(a+1\right)-\left(a+1\right) \left[e^z \left(-z\right)^{a+1}+\left(a+z\right) \, \Gamma\left(a+1,-z\right)\right]\right\}\nonumber\\
\label{eq:f11.form.momsec}
\end{eqnarray}
one has respectively:
\begin{eqnarray*}
\lefteqn{_1F_1\left(2;4;\frac{\lambda^+}{2}\right)=}\\
& = & \frac{8}{\left(\lambda^+\right)^3} \left[12+3\, \lambda^++\frac{3\left(\lambda^+\right)^3}{8} \, e^{\frac{\lambda^+}{2}}-3\left(2+\frac{\lambda^+}{2}\right) \, \int_{-\frac{\lambda^+}{2}}^{+\infty}t^2 \, e^{-t} \, dt\right]=\\
& = & \frac{96}{\left(\lambda^+\right)^3}+\frac{24}{\left(\lambda^+\right)^2} -\frac{96}{\left(\lambda^+\right)^3} \, e^{\frac{\lambda^+}{2}}+\frac{24}{\left(\lambda^+\right)^2} \,e^{\frac{\lambda^+}{2}},
\end{eqnarray*}
\begin{eqnarray*}
\lefteqn{_1F_1\left(3;5;\frac{\lambda^+}{2}\right)=}\\
& = & -\frac{16}{\left(\lambda^+\right)^4} \left[72+12\, \lambda^+-\frac{\left(\lambda^+\right)^4}{4} \, e^{\frac{\lambda^+}{2}}-4\left(3+\frac{\lambda^+}{2}\right) \, \int_{-\frac{\lambda^+}{2}}^{+\infty}t^3 \, e^{-t} \, dt\right]=\\
& = & -\frac{1152}{\left(\lambda^+\right)^4}-\frac{192}{\left(\lambda^+\right)^3} +\frac{1152}{\left(\lambda^+\right)^4} \, e^{\frac{\lambda^+}{2}}-\frac{384}{\left(\lambda^+\right)^3} \,e^{\frac{\lambda^+}{2}}+\frac{48}{\left(\lambda^+\right)^2} \, e^{\frac{\lambda^+}{2}},
\end{eqnarray*}
\begin{eqnarray*}
\lefteqn{_1F_1\left(4;6;\frac{\lambda^+}{2}\right)=}\\
& = & \frac{32}{\left(\lambda^+\right)^5} \left[480+60\, \lambda^++\frac{5 \, \left(\lambda^+\right)^5}{32} \, e^{\frac{\lambda^+}{2}}-5\left(4+\frac{\lambda^+}{2}\right) \, \int_{-\frac{\lambda^+}{2}}^{+\infty}t^4 \, e^{-t} \, dt\right]=\\
& = & \frac{15360}{\left(\lambda^+\right)^5}+\frac{1920}{\left(\lambda^+\right)^4} -\frac{15360}{\left(\lambda^+\right)^5} \, e^{\frac{\lambda^+}{2}}+\frac{5760}{\left(\lambda^+\right)^4} \,e^{\frac{\lambda^+}{2}}-\frac{960}{\left(\lambda^+\right)^3} \, e^{\frac{\lambda^+}{2}}+\frac{80}{\left(\lambda^+\right)^2} \, e^{\frac{\lambda^+}{2}}.
\end{eqnarray*}
Therefore, one has:
\begin{eqnarray*}
\lefteqn{\mathbb{E}\left[\left(X'\right)^2\right]=}\\
& = & -\frac{32}{\left(\lambda^+\right)^3}+\frac{8}{\left(\lambda^+\right)^2} +\frac{32}{\left(\lambda^+\right)^3} \, e^{-\frac{\lambda^+}{2}}+\frac{8}{\left(\lambda^+\right)^2} \,e^{-\frac{\lambda^+}{2}}+\\
& + & \lambda_1 \left[\frac{192}{\left(\lambda^+\right)^4}-\frac{64}{\left(\lambda^+\right)^3} +\frac{8}{\left(\lambda^+\right)^2}- \frac{192}{\left(\lambda^+\right)^4}\, e^{-\frac{\lambda^+}{2}}-\frac{32}{\left(\lambda^+\right)^3} \,e^{-\frac{\lambda^+}{2}}\right]+\\
& + & \lambda_1^2 \left[-\frac{192}{\left(\lambda^+\right)^5}+\frac{72}{\left(\lambda^+\right)^4} -\frac{12}{\left(\lambda^+\right)^3}+\frac{1}{\left(\lambda^+\right)^2}+ \frac{192}{\left(\lambda^+\right)^5}\, e^{-\frac{\lambda^+}{2}}+\frac{24}{\left(\lambda^+\right)^4} \,e^{-\frac{\lambda^+}{2}}\right].
\end{eqnarray*}
Moreover:
\begin{eqnarray*}
\lefteqn{\mathbb{E}\left[\left(X'\right)\right]^2=}\\
& = & \frac{16}{\left(\lambda^+\right)^4}-\frac{16}{\left(\lambda^+\right)^3} +\frac{4}{\left(\lambda^+\right)^2}-\frac{32}{\left(\lambda^+\right)^4} \, e^{-\frac{\lambda^+}{2}}+\frac{16}{\left(\lambda^+\right)^3} \,e^{-\frac{\lambda^+}{2}}+\frac{16}{\left(\lambda^+\right)^4} \,e^{-\lambda^+}+\\
& + & \lambda_1 \left[-\frac{64}{\left(\lambda^+\right)^5}+\frac{64}{\left(\lambda^+\right)^4} -\frac{24}{\left(\lambda^+\right)^3}+ \frac{4}{\left(\lambda^+\right)^2}+\frac{128}{\left(\lambda^+\right)^5} \, e^{-\frac{\lambda^+}{2}}-\frac{64}{\left(\lambda^+\right)^4} \,e^{-\frac{\lambda^+}{2}}+ \right. \\
& + & \left. \frac{8}{\left(\lambda^+\right)^3} \,e^{-\frac{\lambda^+}{2}}-\frac{64}{\left(\lambda^+\right)^5} \,e^{-\lambda^+}\right]+\\
& + & \lambda_1^2 \left[\frac{64}{\left(\lambda^+\right)^6}-\frac{64}{\left(\lambda^+\right)^5} +\frac{32}{\left(\lambda^+\right)^4}-\frac{8}{\left(\lambda^+\right)^3}+ \frac{1}{\left(\lambda^+\right)^2}-\frac{128}{\left(\lambda^+\right)^6}\, e^{-\frac{\lambda^+}{2}}+\frac{64}{\left(\lambda^+\right)^5} \,e^{-\frac{\lambda^+}{2}}+ \right.\\
& - & \left.\frac{16}{\left(\lambda^+\right)^4} \,e^{-\frac{\lambda^+}{2}}+\frac{64}{\left(\lambda^+\right)^6} \,e^{-\lambda^+}\right].
\end{eqnarray*}
Hence:
\begin{eqnarray*}
\lefteqn{\mbox{Var}\left(X'\right)=\mathbb{E}\left[\left(X'\right)^2\right]-\mathbb{E}\left[\left(X'\right)\right]^2=}\\
& = & -\frac{16}{\left(\lambda^+\right)^4}-\frac{16}{\left(\lambda^+\right)^3} +\frac{4}{\left(\lambda^+\right)^2}+\frac{32}{\left(\lambda^+\right)^4} \, e^{-\frac{\lambda^+}{2}}+\frac{16}{\left(\lambda^+\right)^3} \,e^{-\frac{\lambda^+}{2}}+\frac{8}{\left(\lambda^+\right)^2} \,e^{-\frac{\lambda^+}{2}}+\\
& - & \frac{16}{\left(\lambda^+\right)^4} \,e^{-\lambda^+}+\\
& + & \lambda_1 \left[\frac{64}{\left(\lambda^+\right)^5}+\frac{128}{\left(\lambda^+\right)^4} -\frac{40}{\left(\lambda^+\right)^3}+ \frac{4}{\left(\lambda^+\right)^2}-\frac{128}{\left(\lambda^+\right)^5} \, e^{-\frac{\lambda^+}{2}}-\frac{128}{\left(\lambda^+\right)^4} \,e^{-\frac{\lambda^+}{2}}+ \right. \\
& - & \left. \frac{40}{\left(\lambda^+\right)^3} \,e^{-\frac{\lambda^+}{2}}+\frac{64}{\left(\lambda^+\right)^5} \,e^{-\lambda^+}\right]+\\
& + & \lambda_1^2 \left[-\frac{64}{\left(\lambda^+\right)^6}-\frac{128}{\left(\lambda^+\right)^5} +\frac{40}{\left(\lambda^+\right)^4}-\frac{4}{\left(\lambda^+\right)^3}+ \frac{128}{\left(\lambda^+\right)^6}\, e^{-\frac{\lambda^+}{2}}+\frac{128}{\left(\lambda^+\right)^5}\, e^{-\frac{\lambda^+}{2}}+ \right.\\
& + & \left. \frac{40}{\left(\lambda^+\right)^4} \,e^{-\frac{\lambda^+}{2}} - \frac{64}{\left(\lambda^+\right)^6} \,e^{-\lambda^+}\right].
\end{eqnarray*}
At this point, the proof of the variance formula follows from the mere application of some tedious algebra. Upon noting that:
\begin{eqnarray*}
-\frac{16}{\left(\lambda^+\right)^4} \,e^{-\lambda^+}+\frac{32}{\left(\lambda^+\right)^4} \, e^{-\frac{\lambda^+}{2}}-\frac{16}{\left(\lambda^+\right)^4} & = & -\frac{16}{\left(\lambda^+\right)^4} \left(1-e^{-\frac{\lambda^+}{2}}\right)^2,\\
\lambda_1 \left[\frac{64}{\left(\lambda^+\right)^5} \, e^{-\lambda^+}-\frac{128}{\left(\lambda^+\right)^5} \, e^{-\frac{\lambda^+}{2}} +\frac{64}{\left(\lambda^+\right)^5}\right] & = & \frac{64 \, \lambda_1}{\left(\lambda^+\right)^5} \, \left(1-e^{-\frac{\lambda^+}{2}}\right)^2,\\
\lambda_1^2 \left[-\frac{64}{\left(\lambda^+\right)^6} \,e^{-\lambda^+}+ \frac{128}{\left(\lambda^+\right)^6}\, e^{-\frac{\lambda^+}{2}}-\frac{64}{\left(\lambda^+\right)^6}\right] & = & -\frac{64 \, \lambda_1^2}{\left(\lambda^+\right)^6}\, \left(1-e^{-\frac{\lambda^+}{2}}\right)^2,
\end{eqnarray*}
we have:
\begin{eqnarray*}
\lefteqn{-\frac{16 \left(\lambda_1-\lambda_2\right)^2}{\left(\lambda^+\right)^6} \, \left(1-e^{-\frac{\lambda^+}{2}}\right)^2=}\\
& = & -\frac{16}{\left(\lambda^+\right)^4} \left(1-e^{-\frac{\lambda^+}{2}}\right)^2+\frac{64 \, \lambda_1}{\left(\lambda^+\right)^5} \, \left(1-e^{-\frac{\lambda^+}{2}}\right)^2-\frac{64 \, \lambda_1^2}{\left(\lambda^+\right)^6}\, \left(1-e^{-\frac{\lambda^+}{2}}\right)^2;
\end{eqnarray*}
moreover:
\begin{eqnarray*}
\lefteqn{\frac{\lambda_1 \, \lambda_2}{\lambda^+}\left[\frac{128}{\left(\lambda^+\right)^4}-\frac{40}{\left(\lambda^+\right)^3}+\frac{4}{\left(\lambda^+\right)^2}-\frac{128}{\left(\lambda^+\right)^4} \,e^{-\frac{\lambda^+}{2}}-\frac{40}{\left(\lambda^+\right)^3} \,e^{-\frac{\lambda^+}{2}}\right]=}\\
& = & \lambda_1 \, \left[\frac{128}{\left(\lambda^+\right)^4} -\frac{40}{\left(\lambda^+\right)^3}+ \frac{4}{\left(\lambda^+\right)^2}-\frac{128}{\left(\lambda^+\right)^4} \,e^{-\frac{\lambda^+}{2}}-\frac{40}{\left(\lambda^+\right)^3} \,e^{-\frac{\lambda^+}{2}}\right]+\\
& + & \lambda_1^2 \left[-\frac{128}{\left(\lambda^+\right)^5} +\frac{40}{\left(\lambda^+\right)^4}-\frac{4}{\left(\lambda^+\right)^3}+ \frac{128}{\left(\lambda^+\right)^5}\, e^{-\frac{\lambda^+}{2}}+\frac{40}{\left(\lambda^+\right)^4} \,e^{-\frac{\lambda^+}{2}} \right],
\end{eqnarray*}
so that:
\begin{eqnarray*}
\mbox{Var}\left(X'\right) & = & -\frac{16 \left(\lambda_1-\lambda_2\right)^2}{\left(\lambda^+\right)^6} \, \left(1-e^{-\frac{\lambda^+}{2}}\right)^2+\\
& + & \frac{\lambda_1 \, \lambda_2}{\lambda^+}\left[\frac{128}{\left(\lambda^+\right)^4}-\frac{40}{\left(\lambda^+\right)^3}+\frac{4}{\left(\lambda^+\right)^2}-\frac{128}{\left(\lambda^+\right)^4} \,e^{-\frac{\lambda^+}{2}}-\frac{40}{\left(\lambda^+\right)^3} \,e^{-\frac{\lambda^+}{2}}\right]+\\
& - & \frac{16}{\left(\lambda^+\right)^3} +\frac{4}{\left(\lambda^+\right)^2}+\frac{16}{\left(\lambda^+\right)^3} \,e^{-\frac{\lambda^+}{2}}+\frac{8}{\left(\lambda^+\right)^2} \,e^{-\frac{\lambda^+}{2}}.
\end{eqnarray*}
By observing that:
\begin{eqnarray*}
\lefteqn{\frac{16}{\left(\lambda^+\right)^5} \left(1-e^{-\frac{\lambda^+}{2}}\right) \left[4 \, \lambda_1 \, \lambda_2- \left(\lambda_1-\lambda_2\right)^2\right]=}\\
& = & \frac{\lambda_1 \, \lambda_2}{\lambda^+}\left[\frac{128}{\left(\lambda^+\right)^4}-\frac{128}{\left(\lambda^+\right)^4} \,e^{-\frac{\lambda^+}{2}}\right]-\frac{16}{\left(\lambda^+\right)^3} +\frac{16}{\left(\lambda^+\right)^3} \,e^{-\frac{\lambda^+}{2}},
\end{eqnarray*}
we have:
\begin{eqnarray*}
\mbox{Var}\left(X'\right) & = & -\frac{16 \left(\lambda_1-\lambda_2\right)^2}{\left(\lambda^+\right)^6} \, \left(1-e^{-\frac{\lambda^+}{2}}\right)\left(\lambda^++1-e^{-\frac{\lambda^+}{2}}\right)+\\
& + & \frac{\lambda_1 \, \lambda_2}{\lambda^+}\left[\frac{64}{\left(\lambda^+\right)^4}-\frac{40}{\left(\lambda^+\right)^3}+\frac{4}{\left(\lambda^+\right)^2}-\frac{64}{\left(\lambda^+\right)^4} \,e^{-\frac{\lambda^+}{2}}-\frac{40}{\left(\lambda^+\right)^3} \,e^{-\frac{\lambda^+}{2}}\right]+\\
& + & \frac{8}{\left(\lambda^+\right)^2} \,e^{-\frac{\lambda^+}{2}} +\frac{4}{\left(\lambda^+\right)^2}.
\end{eqnarray*}
Note that:
$$
\frac{\lambda_1 \, \lambda_2}{\lambda^+}\left[\frac{64}{\left(\lambda^+\right)^4}-\frac{40}{\left(\lambda^+\right)^3}+\frac{4}{\left(\lambda^+\right)^2}\right]=\frac{4 \, \lambda_1 \, \lambda_2}{\left(\lambda^+\right)^5} \left(\lambda^+-2\right)\left(\lambda^+-8\right)
$$
and:
$$
-\frac{40 \, \lambda_1 \, \lambda_2}{\left(\lambda^+\right)^4}\,e^{-\frac{\lambda^+}{2}} +\frac{8}{\left(\lambda^+\right)^2}\,e^{-\frac{\lambda^+}{2}}=\frac{8}{\left(\lambda^+\right)^4}\,e^{-\frac{\lambda^+}{2}} \left[\left(\lambda_1-\lambda_2\right)^2-\lambda_1 \, \lambda_2\right],
$$
so that:
\begin{eqnarray*}
\lefteqn{\frac{8}{\left(\lambda^+\right)^4}\,e^{-\frac{\lambda^+}{2}} \left(\lambda_1-\lambda_2\right)^2-\frac{16 \left(\lambda_1-\lambda_2\right)^2}{\left(\lambda^+\right)^6} \, \left(1-e^{-\frac{\lambda^+}{2}}\right)\left(\lambda^++1-e^{-\frac{\lambda^+}{2}}\right)=}\\
& = & \frac{8 \, \left(\lambda_1-\lambda_2\right)^2}{\left(\lambda^+\right)^6} \left[\left(\lambda^+\right)^2 \,e^{-\frac{\lambda^+}{2}} -2 \left(1-e^{-\frac{\lambda^+}{2}}\right)\left(\lambda^++1-e^{-\frac{\lambda^+}{2}}\right) \right].
\end{eqnarray*}
Finally:
$$
-\frac{8 \, \lambda_1 \, \lambda_2}{\left(\lambda^+\right)^4} \, e^{-\frac{\lambda^+}{2}}-\frac{64 \, \lambda_1 \, \lambda_2}{\left(\lambda^+\right)^5} \, e^{-\frac{\lambda^+}{2}}=-\frac{8 \, \lambda_1 \, \lambda_2}{\left(\lambda^+\right)^5} \left(\lambda^++8\right) e^{-\frac{\lambda^+}{2}}
$$
and:
\begin{eqnarray*}
\lefteqn{\frac{4 \, \lambda_1 \, \lambda_2}{\left(\lambda^+\right)^5} \left(\lambda^+-2\right)\left(\lambda^+-8\right)-\frac{8 \, \lambda_1 \, \lambda_2}{\left(\lambda^+\right)^5} \left(\lambda^++8\right) e^{-\frac{\lambda^+}{2}}=}\\
& = & \frac{4 \, \lambda_1 \, \lambda_2}{\left(\lambda^+\right)^5} \left[\left(\lambda^+-2\right)\left(\lambda^+-8\right)-2 \,e^{-\frac{\lambda^+}{2}} \left(\lambda^++8\right) \right].
\end{eqnarray*}
Eq.~(\ref{eq:exp.var.dnc11.beta}) is thus established.
\end{myproof}

\section{Appendix. R functions}
\label{sec:app.r.func}

\begin{funct}[Perturbation factor of the beta density in the perturbation representation of the doubly non-central beta density in Eq.~(\ref{eq:c.ncb})]
\label{funct:dperturb}

Arguments:
\begin{itemize}
\item \textit{x}: vector of quantiles
\item \textit{shape1}, \textit{shape2}: shape parameters of the doubly non-central beta distribution
\item \textit{ncp1}, \textit{ncp2}: non-centrality parameters of the doubly non-central beta distribution
\item \textit{tol}: tolerance with zero meaning to iterate until additional terms to not change the partial sum
\item \textit{maxiter}: maximum number of iterations to perform
\item \textit{debug}: Boolean, with TRUE meaning to return debugging information and FALSE meaning to return just the evaluate
\end{itemize}
\noindent \texttt{dperturb<-function(x,shape1,shape2,ncp1,ncp2,tol,maxiter,debug) \hspace{0.1cm} \{ }

\noindent \hspace{0.2cm} \texttt{L<-c(shape1,shape2)}

\noindent \hspace{0.2cm} \texttt{U<-sum(L)}

\noindent \hspace{0.2cm} \texttt{y1<-(ncp1/2)*x}

\noindent \hspace{0.2cm} \texttt{y2<-(ncp2/2)*(1-x)}

\noindent \hspace{0.2cm} \texttt{esp=-((ncp1+ncp2)/2)}

\noindent \hspace{0.2cm} \texttt{coef<-1}

\noindent \hspace{0.2cm} \texttt{temp<-hypergeo::genhypergeo(U=U,L=L[2],z=y2,tol=tol,maxiter=maxiter,}

\noindent \hspace{1.3cm} \texttt{check\_mod=TRUE,polynomial=FALSE,debug=FALSE)}

\noindent \hspace{0.2cm} \texttt{out<-NULL}

\noindent \hspace{0.2cm} \texttt{for(m in seq\_len(maxiter)) \hspace{0.1cm} \{ }

\noindent \hspace{0.5cm} \texttt{coef<-coef*((U/L[1])*y1/m)}

\noindent \hspace{0.5cm} \texttt{fac<-coef*hypergeo::genhypergeo(U=U+1,L=L[2],z=y2,tol=tol,}

\noindent \hspace{1.4cm} \texttt{maxiter=maxiter,check\_mod=TRUE,polynomial=FALSE,debug=FALSE)}

\noindent \hspace{0.5cm} \texttt{series<-temp+fac}

\noindent \hspace{0.5cm} \texttt{if(debug) \hspace{0.1cm} \{ }

\noindent \hspace{0.9cm} \texttt{out<-c(out,fac)}

\noindent \hspace{0.5cm} \texttt{\} }

\noindent \hspace{0.5cm} \texttt{if(hypergeo::isgood(series-temp,tol)) \hspace{0.1cm} \{ }

\noindent \hspace{0.9cm} \texttt{if(debug) \hspace{0.1cm} \{ }

\noindent \hspace{1.3cm} \texttt{return(list(exp(esp)*series,exp(esp)*out))}

\noindent \hspace{0.9cm} \texttt{\} }

\noindent \hspace{0.9cm} \texttt{else \{ }

\noindent \hspace{1.3cm} \texttt{return(exp(esp)*series)}

\noindent \hspace{0.9cm} \texttt{\} }

\noindent \hspace{0.5cm} \texttt{\} }

\noindent \hspace{0.5cm} \texttt{temp<-series}

\noindent \hspace{0.5cm} \texttt{U<-U+1}

\noindent \hspace{0.5cm} \texttt{L[1]<-L[1]+1}

\noindent \hspace{0.2cm} \texttt{\} }

\noindent \hspace{0.2cm} \texttt{if(debug) \hspace{0.1cm} \{}

\noindent \hspace{0.5cm} \texttt{return(list(exp(esp)*series,exp(esp)*out))}

\noindent \hspace{0.2cm} \texttt{\}}

\noindent \texttt{\}}
\end{funct}

\begin{funct}[Perturbation representation of the doubly non-central beta density in Eq.~(\ref{eq:c.ncb})]
\label{funct:ddncbeta}

Arguments:
\begin{itemize}
\item \textit{x}: vector of quantiles
\item \textit{shape1}, \textit{shape2}: shape parameters of the doubly non-central beta distribution
\item \textit{ncp1}, \textit{ncp2}: non-centrality parameters of the doubly non-central beta distribution
\end{itemize}
\noindent \texttt{ddncbeta<-function(x,shape1,shape2,ncp1,ncp2) \hspace{0.1cm} \{}

\noindent \hspace{0.2cm} \texttt{dbeta(x,shape1=shape1,shape2=shape2,ncp=0,log=FALSE)*}

\noindent \hspace{0.2cm} \texttt{dperturb(x=x,shape1=shape1,shape2=shape2,ncp1=ncp1,ncp2=ncp2,tol=0,}

\noindent \hspace{0.2cm} \texttt{maxiter=2000,debug=FALSE)}

\noindent \texttt{\}}

\noindent The cases $\mbox{\textit{ncp1}}=0$ and $\mbox{\textit{ncp2}}=0$ correspond respectively to the densities of the type 2 and the type 1 non-central beta distributions.
\end{funct}

\begin{funct}[Internal series of the doubly non-central beta distribution function in Eq.~(\ref{eq:distr.beta.dnc}) for any fixed value of the index of the external one]
\label{funct:int.pdncbeta}

Arguments:
\begin{itemize}
\item \textit{x}: vector of quantiles
\item \textit{first}: value of the index of the external series
\item \textit{shape2}: second shape parameter of the doubly non-central beta distribution
\item \textit{ncp2}: second non-centrality parameter of the doubly non-central beta distribution
\item \textit{tol}: tolerance with zero meaning to iterate until additional terms to not change the partial sum
\item \textit{maxiter}: maximum number of iterations to perform
\item \textit{debug}: Boolean, with TRUE meaning to return debugging information and FALSE meaning to return just the evaluate
\end{itemize}
\noindent \texttt{int.pdncbeta<-function(x,first,shape2,ncp2,tol,maxiter,debug) \hspace{0.1cm} \{}

\noindent \hspace{0.2cm} \texttt{temp<-dpois(0,ncp2/2)*pbeta(x,first,shape2)}

\noindent \hspace{0.2cm} \texttt{out<-NULL}

\noindent \hspace{0.2cm} \texttt{for(m in seq\_len(maxiter)) \hspace{0.1cm} \{}

\noindent \hspace{0.5cm} \texttt{fac<-dpois(m,ncp2/2)*pbeta(x,first,shape2+m)}

\noindent \hspace{0.5cm} \texttt{series<-temp+fac}

\noindent \hspace{0.5cm} \texttt{if(debug) \hspace{0.1cm} \{}

\noindent \hspace{0.9cm} \texttt{out<-c(out,fac)}

\noindent \hspace{0.5cm} \texttt{\}}

\noindent \hspace{0.5cm} \texttt{if(hypergeo::isgood(series-temp,tol)) \hspace{0.1cm} \{}

\noindent \hspace{0.9cm} \texttt{if(debug) \hspace{0.1cm} \{}

\noindent \hspace{1.3cm} \texttt{return(list(series, out))}

\noindent \hspace{0.9cm} \texttt{\}}

\noindent \hspace{0.9cm} \texttt{else \{}

\noindent \hspace{1.3cm} \texttt{return(series)}

\noindent \hspace{0.9cm} \texttt{\}}

\noindent \hspace{0.5cm} \texttt{\}}

\noindent \hspace{0.5cm} \texttt{temp<-series}

\noindent \hspace{0.2cm} \texttt{\}}

\noindent \hspace{0.2cm} \texttt{if(debug) \hspace{0.1cm} \{}

\noindent \hspace{0.5cm} \texttt{return(list(series, out))}

\noindent \hspace{0.2cm} \texttt{\}}

\noindent \texttt{\}}
\end{funct}

\begin{funct}[Doubly non-central beta distribution function]
\label{funct:pdncbeta}

Arguments:
\begin{itemize}
\item \textit{x}: vector of quantiles
\item \textit{shape1}, \textit{shape2}: shape parameters of the doubly non-central beta distribution
\item \textit{ncp1}, \textit{ncp2}: non-centrality parameters of the doubly non-central beta distribution
\item \textit{lower.tail}: logical, if TRUE, probabilities are $\Pr\left(X \leq x\right)$, otherwise, $\Pr\left(X > x\right)$.
\item \textit{tol}: tolerance with zero meaning to iterate until additional terms to not change the partial sum
\item \textit{maxiter}: maximum number of iterations to perform
\item \textit{debug}: Boolean, with TRUE meaning to return debugging information and FALSE meaning to return just the evaluate
\end{itemize}
\noindent \texttt{pdncbeta<-function(x,shape1,shape2,ncp1,ncp2,lower.tail,tol,maxiter,debug) \hspace{0.1cm} \{}

\noindent \hspace{0.2cm} \texttt{temp<-dpois(0,ncp1/2)*int.pdncbeta(x=x,first=shape1,shape2=shape2,}

\noindent \hspace{1.3cm} \texttt{ncp2=ncp2,tol=0,maxiter=2000,debug=FALSE)}

\noindent \hspace{0.2cm} \texttt{out<-NULL}

\noindent \hspace{0.2cm} \texttt{for(n in seq\_len(maxiter)) \hspace{0.1cm} \{}

\noindent \hspace{0.5cm} \texttt{fac<-dpois(n,ncp1/2)*int.pdncbeta(x=x,first=shape1+n,shape2=shape2,}

\noindent \hspace{1.4cm} \texttt{ncp2=ncp2,tol=0,maxiter=2000,debug=FALSE)}

\noindent \hspace{0.5cm} \texttt{series<-temp+fac}

\noindent \hspace{0.5cm} \texttt{if(debug) \hspace{0.1cm} \{}

\noindent \hspace{0.9cm} \texttt{out<-c(out,fac)}

\noindent \hspace{0.5cm} \texttt{\}}

\noindent \hspace{0.5cm} \texttt{if(hypergeo::isgood(series-temp,tol)) \hspace{0.1cm} \{}

\noindent \hspace{0.9cm} \texttt{if(debug) \hspace{0.1cm} \{}

\noindent \hspace{1.3cm} \texttt{if(lower.tail) \hspace{0.1cm} \{}

\noindent \hspace{1.7cm} \texttt{return(list(series, out))}

\noindent \hspace{1.3cm} \texttt{\}}

\noindent \hspace{1.3cm} \texttt{else \hspace{0.1cm} \{}

\noindent \hspace{1.7cm} \texttt{return(list(1-series, out))}

\noindent \hspace{1.3cm} \texttt{\}}

\noindent \hspace{0.9cm} \texttt{\}}

\noindent \hspace{0.9cm} \texttt{else \hspace{0.1cm} \{}

\noindent \hspace{1.3cm} \texttt{if(lower.tail) \hspace{0.1cm} \{}

\noindent \hspace{1.7cm} \texttt{return(series)}

\noindent \hspace{1.3cm} \texttt{\}}

\noindent \hspace{1.3cm} \texttt{else \hspace{0.1cm} \{}

\noindent \hspace{1.7cm} \texttt{return(1-series)}

\noindent \hspace{1.3cm} \texttt{\}}

\noindent \hspace{0.9cm} \texttt{\}}

\noindent \hspace{0.5cm} \texttt{\}}

\noindent \hspace{0.5cm} \texttt{temp<-series}

\noindent \hspace{0.2cm} \texttt{\}}

\noindent \hspace{0.2cm} \texttt{if(debug) \hspace{0.1cm} \{}

\noindent \hspace{0.5cm} \texttt{if(lower.tail) \hspace{0.1cm} \{}

\noindent \hspace{0.9cm} \texttt{return(list(series, out))}

\noindent \hspace{0.5cm} \texttt{\}}

\noindent \hspace{0.5cm} \texttt{else \hspace{0.1cm} \{}

\noindent \hspace{0.9cm} \texttt{return(list(1-series, out))}

\noindent \hspace{0.5cm} \texttt{\}}

\noindent \hspace{0.2cm} \texttt{\}}

\noindent \texttt{\}}

\noindent The cases $\mbox{\textit{ncp1}}=0$ and $\mbox{\textit{ncp2}}=0$ correspond respectively to the type 2 and the type 1 non-central beta distribution functions.
\end{funct}

\begin{funct}[Generating a doubly non-central beta random variable by means of its definition]
\label{funct:rdncbeta}

Arguments:
\begin{itemize}
\item \textit{n}: number of determinations 
\item \textit{shape1}, \textit{shape2}: shape parameters of the doubly non-central beta distribution
\item \textit{ncp1}, \textit{ncp2}: non-centrality parameters of the doubly non-central beta distribution
\end{itemize}
\noindent \texttt{rdncbeta<-function(n,shape1,shape2,ncp1,ncp2) \hspace{0.1cm} \{}

\noindent \hspace{0.2cm} \texttt{y1<-rchisq(n=n,df=2*shape1,ncp=ncp1)}

\noindent \hspace{0.2cm} \texttt{y2<-rchisq(n=n,df=2*shape2,ncp=ncp2)}

\noindent \hspace{0.2cm} \texttt{x<-y1$/$(y1+y2)}

\noindent \texttt{\}}
\end{funct}

\begin{funct}[Generating a doubly non-central beta random variable by means of its representation as a convex linear combination in Eq.~(\ref{eq:rappr.clc.beta.dnc})]
\label{funct:rndncbeta}

Arguments:
\begin{itemize}
\item \textit{n}: number of determinations 
\item \textit{shape1}, \textit{shape2}: shape parameters of the doubly non-central beta distribution
\item \textit{ncp1}, \textit{ncp2}: non-centrality parameters of the doubly non-central beta distribution
\end{itemize}
\noindent \texttt{rndncbeta<-function(n,shape1,shape2,ncp1,ncp2) \hspace{0.1cm} \{}

\noindent \hspace{0.2cm} \texttt{x<-rbeta(n=n,shape1=shape1,shape2=shape2)}

\noindent \hspace{0.2cm} \texttt{msum<-vector(mode='numeric',length=n)}

\noindent \hspace{0.2cm} \texttt{xp2<-vector(mode='numeric',length=n)}

\noindent \hspace{0.2cm} \texttt{xbin<-vector(mode='numeric',length=n)}

\noindent \hspace{0.2cm} \texttt{xppnc<-vector(mode='numeric',length=n)}

\noindent \hspace{0.2cm} \texttt{for(i in 1:n) \hspace{0.1cm} \{}

\noindent \hspace{0.5cm} \texttt{msum[i]<-rpois(n=1,lambda=((ncp1+ncp2)$/$2))}

\noindent \hspace{0.5cm} \texttt{xp2[i]<-rbeta(n=1,shape1=shape1+shape2,shape2=msum[i])}

\noindent \hspace{0.5cm} \texttt{xbin[i]<-rbinom(n=1,size=msum[i],prob=ncp1$/$(ncp1+ncp2))}

\noindent \hspace{0.5cm} \texttt{xppnc[i]<-rbeta(n=1,shape1=xbin[i],shape2=msum[i]-xbin[i])}

\noindent \hspace{0.2cm} \texttt{\}}

\noindent \hspace{0.2cm} \texttt{xp<-xp2*x+(1-xp2)*xppnc}

\noindent \texttt{\}}
\end{funct}

\begin{funct}[Libby and Novick's generalized beta density in Eq.~(\ref{eq:g3b.dens})]
\label{funct:dlng3beta}

Arguments:
\begin{itemize}
\item \textit{x}: vector of quantiles 
\item \textit{shape1}, \textit{shape2}: shape parameters of the Libby and Novick's generalized beta distribution
\item \textit{gamma}: additional parameter of the Libby and Novick's generalized beta distribution
\end{itemize}
\noindent \texttt{dlng3beta<-function(x,shape1,shape2,gamma) \hspace{0.1cm} \{}

\noindent \hspace{0.2cm} \texttt{dbeta(x=x,shape1=shape1,shape2=shape2)*gamma\^{}(shape1)/}

\noindent \hspace{0.2cm} \texttt{((1-(1-gamma)*x)\^{}(shape1+shape2))}

\noindent \texttt{\}}
\end{funct}

\begin{funct}[Libby and Novick's generalized beta distribution function in Eq.~(\ref{eq:distr.beta.dnc.approx})]
\label{funct:plng3beta}

Arguments:
\begin{itemize}
\item \textit{x}: vector of quantiles 
\item \textit{shape1}, \textit{shape2}: shape parameters of the Libby and Novick's generalized beta distribution
\item \textit{gamma}: additional parameter of the Libby and Novick's generalized beta distribution
\item \textit{lower.tail}: logical, if TRUE, probabilities are $\Pr\left(X \leq x\right)$, otherwise, $\Pr\left(X > x\right)$.
\end{itemize}
\noindent \texttt{plng3beta<-function(x,shape1,shape2,gamma,lower.tail) \hspace{0.1cm} \{}

\noindent \hspace{0.2cm} \texttt{if(lower.tail) \hspace{0.1cm} \{}

\noindent \hspace{0.5cm} \texttt{pbeta(q=(gamma*x)/(gamma*x+1-x),shape1=shape1,shape2=shape2)}

\noindent \hspace{0.2cm} \texttt{\}}

\noindent \hspace{0.2cm} \texttt{else \hspace{0.1cm} \{}

\noindent \hspace{0.5cm} \texttt{1-pbeta(q=(gamma*x)/(gamma*x+1-x),shape1=shape1,shape2=shape2)}

\noindent \hspace{0.2cm} \texttt{\}}

\noindent \texttt{\}}
\end{funct}

\begin{funct}[Moments formula of the doubly non-central beta distribution in Eq.~(\ref{eq:momr.beta.dnc})]
\label{funct:mdncbeta}

Arguments:
\begin{itemize}
\item \textit{order}: vector of integers 
\item \textit{shape1}, \textit{shape2}: shape parameters of the doubly non-central beta distribution
\item \textit{ncp1}, \textit{ncp2}: non-centrality parameters of the doubly non-central beta distribution
\end{itemize}
\noindent \texttt{mdncbeta<-function(order,shape1,shape2,ncp1,ncp2) \hspace{0.1cm} \{}

\noindent \hspace{0.2cm} \texttt{shapesum<-shape1+shape2}

\noindent \hspace{0.2cm} \texttt{ncpsum<-ncp1+ncp2}

\noindent \hspace{0.2cm} \texttt{listvectors<-list(length=length(order))}

\noindent \hspace{0.2cm} \texttt{sumvector<-vector(mode='numeric',length=length(order))}

\noindent \hspace{0.2cm} \texttt{momvector<-vector(mode='numeric',length=length(order))}

\noindent \hspace{0.2cm} \texttt{for(j in seq\_len(length(order))) \hspace{0.1cm} \{}

\noindent \hspace{0.5cm} \texttt{listvectors[[j]]<-vector(mode='numeric',length=order[j]+1)}

\noindent \hspace{0.5cm} \texttt{for(i in 0:order[j]) \hspace{0.1cm} \{}

\noindent \hspace{0.9cm} \texttt{listvectors[[j]][i+1]<-}

\noindent \hspace{1.3cm} \texttt{choose(order[j],i)*orthopolynom::pochhammer(shapesum,i)*}

\noindent \hspace{1.3cm} \texttt{(ncp1$/$2)\^{}i/(orthopolynom::pochhammer(shape1,i)*}

\noindent \hspace{1.3cm} \texttt{orthopolynom::pochhammer(shapesum+order[j],i))*}

\noindent \hspace{1.3cm} \texttt{hypergeo::genhypergeo(U=shapesum+i,L=shapesum+order[j]+i,}

\noindent \hspace{1.3cm} \texttt{z=ncpsum$/$2,tol=0,maxiter=2000,check\_mod=TRUE,polynomial=FALSE,}

\noindent \hspace{1.3cm} \texttt{debug=FALSE)}

\noindent \hspace{0.5cm} \texttt{\}}

\noindent \hspace{0.5cm} \texttt{sumvector[j]<-sum(listvectors[[j]])}

\noindent \hspace{0.5cm} \texttt{momvector[j]<-(orthopolynom::pochhammer(shape1,order[j])/}

\noindent \hspace{2.5cm} \texttt{orthopolynom::pochhammer(shapesum,order[j]))*}

\noindent \hspace{2.5cm} \texttt{exp(-ncpsum/2)*sumvector[j]}

\noindent \hspace{0.2cm} \texttt{\}}
 
\noindent \hspace{0.2cm} \texttt{return(momvector)}

\noindent \texttt{\}}

\noindent The cases $\mbox{\textit{ncp1}}=0$ and $\mbox{\textit{ncp2}}=0$ correspond respectively to the moments formulas of the type 2 and the type 1 non-central beta distributions.
\end{funct}

\end{document}